\documentclass[11pt]{amsart}
\usepackage[utf8]{inputenc}
\usepackage[T1]{fontenc}
\usepackage{lmodern}
\usepackage{microtype}
\usepackage[leqno]{amsmath}
\usepackage{amssymb}
\usepackage{mathtools}
\usepackage{enumerate}
\usepackage{graphicx}
\usepackage[usenames,dvipsnames]{xcolor}
\usepackage[
colorlinks=true,linkcolor=NavyBlue,urlcolor=RoyalBlue,citecolor=PineGreen,bookmarks=true,bookmarksdepth=3,bookmarksopen=true,bookmarksopenlevel=2,unicode=true,linktocpage]{hyperref}

\numberwithin{equation}{section}
\numberwithin{figure}{section}

\newtheorem{theorem}{Theorem}[section]
\newtheorem{remark}[theorem]{Remark}
\newtheorem{lemma}[theorem]{Lemma}
\newtheorem{proposition}[theorem]{Proposition}
\newtheorem{corollary}[theorem]{Corollary}

\newtheorem{definition}[theorem]{Definition}

\let\C\relax

\newcommand{\C}{\mathbf{C}}
\newcommand{\D}{\mathbf{D}}

\newcommand{\h}{\mathbf{H}}
\newcommand{\N}{\mathbf{N}}
\newcommand{\Z}{\mathbf{Z}}
\newcommand{\p}{\mathbf{P}}

\newcommand{\R}{\mathbf{R}}

\newcommand{\FQ}{\mathfrak{Q}}

\newcommand{\CA}{\mathcal {A}}

\newcommand{\CC}{\mathcal {C}}

\newcommand{\CE}{\mathcal {E}}
\newcommand{\CF}{\mathcal {F}}
\newcommand{\CI}{\mathcal {I}}
\newcommand{\CJ}{\mathcal {J}}
\newcommand{\CK}{\mathcal {K}}
\newcommand{\CL}{\mathcal {L}}

\newcommand{\CP}{\mathcal {P}}

\newcommand{\CX}{\mathcal {X}}

\newcommand{\SLE}{{\rm SLE}}
\newcommand{\CLE}{{\rm CLE}}

\newcommand{\dist}{\mathrm{dist}}

\newcommand{\supp}{\mathrm{supp }}

\newcommand{\one}{{\bf 1}}

\newcommand{\giv}{\,|\,}

\newcommand{\wt}{\widetilde}
\newcommand{\wh}{\widehat}
\newcommand{\ol}{\overline}
\newcommand{\ul}{\underline}

\newcommand{\quant}[3][]{{{\mathfrak q}}_{#3}\if\relax\detokenize{#1}\relax\else^{#1}\fi(#2)}
\newcommand{\median}[2][]{{\mathfrak m}_{#2}\if\relax\detokenize{#1}\relax\else^{#1}\fi}
\newcommand{\mediant}[2][]{\wt{\mathfrak m}_{#2}\if\relax\detokenize{#1}\relax\else^{#1}\fi}

\newcommand{\Fd}{\mathfrak d}
\newcommand{\met}[3]{\Fd(#1,#2;#3)}

\newcommand{\metres}[4]{\Fd^{#1}(#2,#3;#4)}

\newcommand{\geoexp}{{\alpha_{\mathrm{g}}}}

\newcommand{\FR}{{\mathfrak R}}

\newcommand{\rmet}[3]{\FR(#1,#2;#3)}

\newcommand{\rmett}[3]{\wt{\FR}(#1,#2;#3)}
\newcommand{\rmetres}[4]{\FR^{#1}(#2,#3;#4)}
\newcommand{\rmettres}[4]{\wt{\FR}^{#1}(#2,#3;#4)}
\newcommand{\rmetapprox}[4]{\FR_{#1}(#2,#3;#4)}
\newcommand{\rmetapproxres}[5]{\FR_{#1}^{#2}(#3,#4;#5)}

\newcommand{\resexp}{{\alpha_{\mathrm{r}}}}

\newcommand{\meas}[2]{\mu(#1; #2)}
\newcommand{\clebm}{X}
\newcommand{\clebmlaw}[1]{\mathrm{P}_{#1}}

\newcommand{\rfdomain}[1]{\CF_{#1}}
\newcommand{\rfdomainres}[2]{\CF_{#1,#2}}

\newcommand{\trfdomainres}[2]{\wt{\CF}_{#1,#2}}

\newcommand{\rform}[3]{\CE(#1,#2;#3)}
\newcommand{\rformres}[4]{\CE^{#1}(#2,#3;#4)}
\newcommand{\rformtr}[4]{\CE|_{#1}(#2,#3;#4)}
\newcommand{\rformrestr}[5]{\CE^{#1}|_{#2}(#3,#4;#5)}

\newcommand{\trform}[3]{\wt{\CE}(#1,#2;#3)}
\newcommand{\trformres}[4]{\wt{\CE}^{#1}(#2,#3;#4)}

\newcommand{\trformrestr}[5]{\wt{\CE}^{#1}|_{#2}(#3,#4;#5)}

\newcommand*{\graph}{\mathfrak{G}}
\newcommand*{\vertexset}{\mathfrak{V}}
\newcommand*{\edgeset}{\mathfrak{E}}

\newcommand*{\approxgraph}[2][]{\graph_{#2}\if\relax\detokenize{#1}\relax\else^{#1}\fi}
\newcommand*{\approxvtcs}[2][]{\vertexset_{#2}\if\relax\detokenize{#1}\relax\else^{#1}\fi}
\newcommand*{\approxedges}[2][]{\edgeset_{#2}\if\relax\detokenize{#1}\relax\else^{#1}\fi}

\newcommand{\bIn}{\partial_{\mathrm{in}}}
\newcommand{\bOut}{\partial_{\mathrm{out}}}

\DeclarePairedDelimiter\abs{\lvert}{\rvert}

\newcommand*{\defeq}{\mathrel{\mathop:}=}

\newcommand*{\mmiddle}[1]{\mathrel{}\middle#1\mathrel{}}

\newcommand*{\Fill}{\operatorname{fill}}

\newcommand*{\sle}[1]{$\SLE_{#1}$}
\newcommand*{\slek}{\sle{\kappa}}
\newcommand*{\slekp}{\sle{\kappa'}}
\newcommand*{\slekr}[1]{$\SLE_{\kappa}(#1)$}
\newcommand*{\slekpr}[1]{$\SLE_{\kappa'}(#1)$}
\newcommand*{\cle}[1]{$\CLE_{#1}$}
\newcommand*{\clek}{\cle{\kappa}}
\newcommand*{\clekp}{\cle{\kappa'}}

\newcommand{\markeddomain}[1]{{\mathfrak D}_{#1}}
\newcommand{\eldomain}[1]{{\mathfrak D}_{#1}^{\mathrm{ext}}}

\newcommand{\mcclelaw}[1]{\p_{(#1)}^{\CLE_{\kappa}}}

\newcommand{\domainpair}[1]{{\mathfrak {P}}_{#1}}

\newcommand{\outside}{{\mathrm{out}}}
\newcommand{\inside}{{\mathrm{in}}}

\newcommand{\resampled}{{\mathrm{res}}}

\newcommand*{\metregions}[1][]{\mathfrak{C}\if\relax\detokenize{#1}\relax\else_{#1}\fi}
\newcommand*{\cserial}{c_{\mathrm{s}}}
\newcommand*{\cparallel}{c_{\mathrm{p}}}
\newcommand*{\ac}[1]{\mathbf{a}_{#1}}

\newcommand*{\distE}{\operatorname{dist_E}}
\newcommand*{\diamE}{\operatorname{diam_E}}
\newcommand*{\dpathY}[1][]{d_{\mathrm{path}}\if\relax\detokenize{#1}\relax\else^{#1}\fi}
\newcommand{\Bpath}{B_{\mathrm{path}}}

\newcommand{\projmap}[1]{\Pi_{#1}}
\newcommand{\gasketspace}[1]{\CX_{#1}}

\newcommand*{\paths}[4]{P(#1,#2;#3;#4)}

\newcommand{\lmet}[1]{L_{\Fd}(#1)}

\newcommand*{\dsle}{{d_{\SLE}}}
\newcommand*{\double}{{\mathrm{dbl}}}
\newcommand*{\ddouble}{{d_{\double}}}
\newcommand*{\angledouble}{{\theta_{\double}}}

\newcommand{\dcle}{d_\CLE}

\newcommand*{\rateexp}{c_0}

\title[The Brownian motion in non-simple CLE gaskets]{Existence and uniqueness of the canonical Brownian motion in non-simple conformal loop ensemble gaskets}

\usepackage[foot]{amsaddr}
\author{Jason Miller}
\author{Yizheng Yuan}
\address{Department of Pure Mathematics and Mathematical Statistics, University of Cambridge}
\address{Faculty of Mathematics, University of Vienna}

\date{\today}

\begin{document}

\begin{abstract}
We construct the canonical Brownian motion on the gasket of conformal loop ensembles (CLE$_\kappa$) for $\kappa \in (4,8)$ (which is the range of parameter values in which loops of the CLE$_\kappa$ can intersect themselves, each other, and the domain boundary). More precisely, we show that there is a unique diffusion process on the CLE$_\kappa$ gasket whose law depends locally on the CLE$_\kappa$ and satisfies certain natural properties such as translation-invariance and scale-invariance (modulo time change). We characterize the diffusion process by its resistance form and show in particular that there is a unique resistance form on the CLE$_\kappa$ gasket that is locally determined by the CLE$_\kappa$ and satisfies certain natural properties such as translation-invariance and scale-covariance. We conjecture that the CLE$_\kappa$ Brownian motion describes the scaling limit of simple random walk on statistical mechanics models in two dimensions that converge to CLE$_\kappa$. In future work the results of this paper will be used to show that this is the case with $\kappa=6$ for critical percolation on the triangular lattice.
\end{abstract}

\maketitle

\setcounter{tocdepth}{1}

\tableofcontents

\parindent 0 pt
\setlength{\parskip}{0.20cm plus1mm minus1mm}

\section{Introduction}
\label{sec:intro}

\subsection{Overview}

The conformal loop ensembles (\clek{}) are canonical models for conformally invariant random planar fractals, indexed by a parameter $\kappa \in [8/3,8]$. They were introduced in \cite{s2009cle,sw2012cle} and are conjectured (and in some cases proved) to describe the scaling limits of a number of critical statistical mechanics models on two-dimensional lattices (such as critical percolation, the Ising model, the uniform spanning tree, and conjecturally the FK models and loop $O(n)$ models). Some convergence results have been established in \cite{bh2019ising,cn2006cle,ks2019fkising,lsw2004lerw}. In the setting of random lattices convergence results have been proved in \cite{s2016hc,gm2021percolation} using the framework developed in \cite{s2016zipper,dms2021mating}.

The goal of this paper is to construct the conjectural scaling limit of the simple random walk on critical models converging to \clek{} for $\kappa \in (4,8)$, giving rise to the \clek{} Brownian motion. The particular case of critical percolation, corresponding to $\kappa=6$, has a long history. It was popularized by de Gennes in 1976 \cite{dg-ant} and he named it the ``Ant in the Labyrinth'' problem. In the forthcoming work \cite{dmmy2025percolation} we will prove that the simple random walk on critical percolation on the triangular lattice converges to the \cle{6} Brownian motion by showing that the associated resistance metrics converge in the scaling limit using the results of the present paper and invoking the main result of \cite{c2018scalinglimits}.

Each \clek{} is a random collection of loops that do not cross themselves and each other. These loops describe the (conjectural) scaling limits of the interfaces in the aforementioned lattice models. The gaskets of a \clek{} are the sets of points connected by paths not crossing any loop, and describe the (conjectural) scaling limits of the clusters in the aforementioned lattice models. They are fractal sets whose dimension is (see \cite{ssw2009radii,msw2014dimension})
\begin{equation}
\label{eqn:cle_dim}
\dcle = 2 - \frac{(8-\kappa)(3\kappa-8)}{32\kappa} = 1 + \frac{2}{\kappa} + \frac{3\kappa}{32} .
\end{equation}
The trivial case \cle{8/3} is the empty collection of loops, and \cle{8} consists of a single space-filling loop. When $\kappa \in (8/3,8)$, \clek{} consists of a countable collection of loops. In the regime $\kappa \in (8/3,4]$ the loops are simple, do not intersect each other, or the domain boundary, while for $\kappa \in (4,8)$ the loops are self-intersecting, intersect each other, and the domain boundary. We note that it can be natural to consider either the nested or non-nested variants of $\CLE_\kappa$. In the former, inside of each loop one has another $\CLE_\kappa$ while the latter consists of just the outermost loops of a nested $\CLE_\kappa$.

In order to construct and characterize the \clek{} Brownian motion, it suffices to construct and characterize a Dirichlet form defined on the $\CLE_{\kappa}$ gasket. On fractals with dimension strictly smaller than~$2$ where the Brownian motion is recurrent in a strong sense, its Dirichlet form is typically given by a \emph{resistance form}. We will therefore construct and characterize the \clek{} Brownian motion using its resistance form. The definition of a resistance form is recalled in Definition~\ref{def:rform_definition}.

To each resistance form there is an associated \emph{resistance metric} and vice versa. We will review this in Section~\ref{se:rmet_rform}. We explain this correspondence in the case of finite graphs. Suppose that $G = (V,E)$ is a simple, finite, connected, undirected graph, and $w \colon E \to (0,\infty)$ is a collection of edge weights. For each function $f\colon V \to \R$ define
\[
 \CE(f,f) = \sum_{\{x,y\} \in E} w(x,y)(f(x)-f(y))^2 ,
\]
and for each $x,y \in V$ distinct define
\begin{equation}\label{eq:reff}
 R(x,y) = \bigl( \min\{ \CE(f,f) \mid f\colon V \to \R,\, f(x)=1,\, f(y)=0 \} \bigr)^{-1} .
\end{equation}
Then $\CE$ defines a symmetric quadratic form on the set of real-valued functions on $V$, called the \emph{Dirichlet form} on the network, and $R$ defines a metric on $V$, called the \emph{effective resistance} on the network. It turns out that each of $w$, $\CE$, and $R$ determine the others \cite[Section~2.1]{k2001analysis}. Further, given a finite measure $\mu$ on $V$, there is a natural Markov process given as follows. For each $x \in V$, let $\mu_0(x) = \sum_{y \sim x} w(x,y)$. At vertex $x$, the process jumps at an exponentially distributed time with rate $\mu_0(x)/\mu(x)$, and it jumps to $y$ with probability $w(x,y)/\mu_0(x)$. This Markov process is symmetric with respect to the measure $\mu$ (see \cite[Section~2.2.1]{cf2012dform}), and is called the Hunt process associated with the Dirichlet form $\CE$ on $L^2(\mu)$. We note that the measure $\mu$ only determines the rate of the jumps, whereas the jump probabilities are determined by $\CE$ (equivalently $R$).

The theory of resistance forms and resistance metrics developed in \cite{k2001analysis} generalizes this notion to continuum spaces. Each resistance form on a space~$F$ is a symmetric quadratic form on the space of real-valued functions on $F$, and is in one-to-one correspondence with a resistance metric given by the same identity~\eqref{eq:reff}. If we are given a finite Borel measure $\mu$ on $F$, we get a $\mu$-symmetric Markov process associated to the resistance form. See Section~\ref{se:rmet_rform} for a review.

Therefore, to construct and characterize a Markov process in the \clek{} gasket, it suffices to construct and characterize a resistance form and a Borel measure. The canonical measure on the \clek{} gasket was constructed in \cite{ms2022clemeasure} and is characterized as the unique measure that is locally determined by the \clek{} and conformally covariant (see Section~\ref{se:cle_measure} for a review). The main aim in this paper is to construct and characterize the canonical resistance metric on the \clek{} gasket. Once we have this, we obtain the canonical \clek{} Brownian motion as its associated Markov process, and we will prove an equivalent characterization of the \clek{} Brownian motion.

As alluded to above, it is shown in \cite{c2018scalinglimits} that convergence of resistance metric spaces equipped with measures implies convergence of the associated Markov processes. The \clek{} gasket measure is the conjectural scaling limit of the uniform measure on the cluster of a lattice model converging to \clek{}. We further conjecture that the \clek{} resistance metric is the scaling limit of the effective resistance metric on the cluster. Therefore it is natural to conjecture that the simple random walk on the lattice model converges to the \clek{} Brownian motion constructed in this paper, and to establish this using \cite{c2018scalinglimits} it suffices to prove the convergence of the resistance metric associated with the simple random walk equipped with the counting measure on the discrete clusters to the resistance metric we construct here together with the \clek{} measure.

\subsection{Main results}
\label{subsec:main_results}

Throughout the paper we let $\kappa' \in (4,8)$ be fixed. We let $\dcle$ be the dimension of the \clekp{} gasket given by~\eqref{eqn:cle_dim}, and
\begin{align}
 \ddouble &= 2 - \frac{(12-\kappa')(4+\kappa')}{8\kappa'} ,\label{eq:ddouble}\\
 \dsle &= 1+\frac{2}{\kappa'} ,\label{eq:dsle}
\end{align}
respectively be the double point dimension of \slekp{} \cite{mw2017intersections} and the outer boundary dimension \cite{b2008dimension}.

In order to start to state our main results, we need to be more precise regarding on which space the resistance metric and resistance form will live.  We first note that the \clekp{} gasket (as a set in the plane) consists of points some of which correspond to multiple ``prime ends''.  We will consider the space that encodes these ``prime ends'' and the manner in which the associated metric space is drawn in the plane. We use the same setup as in \cite{amy2025tightness} which we now briefly review.

For each $U \subseteq \C$, $x,y \in U$, and collection of loops $\Sigma$ we let $\paths{x}{y}{U}{\Sigma}$ denote the set of paths in $U$ from $x$ to $y$ that do not cross any of the loops in $\Sigma$.

Suppose that $\CL$ is a (random) non-self-crossing loop in $\C$, let $C$ denote the set of points that are inside $\CL$ (i.e.\ with winding number $1$), and let $\Gamma_C$ be a conditionally independent (non-nested) \clekp{} in each of the connected components of $C$. We write $\Gamma = \{\CL\} \cup \Gamma_C$. We define the gasket $\Upsilon_\Gamma$ of $\Gamma_C$ as the metric space $(\gasketspace{\Gamma},\dpathY)$ equipped with an embedding $\projmap{\Gamma}\colon \gasketspace{\Gamma} \to \C$ as follows. Let $\wt{\Upsilon}_\Gamma \subseteq C$ be the set of points in $C$ that do not lie on or inside any loop of $\Gamma_C$. For each $x,y \in \wt{\Upsilon}_\Gamma$ define
\begin{equation}\label{eq:dpath}
 \dpathY(x,y) = \inf\{ \diamE(\gamma) : \gamma \in \paths{x}{y}{\ol{C}}{\Gamma} \}
\end{equation}
where $\diamE$ denotes the diameter with respect to the Euclidean metric. We let $(\gasketspace{\Gamma},\dpathY)$ be the metric space completion of $(\wt{\Upsilon}_\Gamma, \dpathY)$, and let $\projmap{\Gamma}\colon \gasketspace{\Gamma} \to \C$ be the continuous extension of the natural embedding map. We let $\Upsilon_\Gamma$ denote the tuple $(\gasketspace{\Gamma},\dpathY,\projmap{\Gamma})$ and (with a slight abuse of notation) identify its points with their images in $\C$ under the embedding $\projmap{\Gamma}$.

Let $\metregions$ be the collection of regions $V \subseteq C$ such that there exists a finite collection of loops $\CL_1,\ldots,\CL_n \in \Gamma_C$ such that $V$ is a union of connected components of
$C \setminus (\CL_1 \cup \cdots \cup \CL_n)$ that are not inside the loops $\CL_1,\ldots,\CL_n$ and such that the completion $\ol{V}$ of $V$ with respect to the metric
\[ d_{\ol{V}}(x,y) = \inf\{ \diamE(\gamma) : \gamma \in \paths{x}{y}{\ol{V}}{\{\CL,\CL_1,\ldots,\CL_n\}} \} \]
is simply connected. Let
\[
 \dpathY[\ol{V}](x,y) = \inf\{ \diamE(\gamma) : \gamma \in \paths{x}{y}{\ol{V}}{\Gamma} \},
 \quad x,y \in \ol{V} \cap \Upsilon_\Gamma .
\]
It is explained in \cite[Lemma~1.10]{amy2025tightness} that $(\ol{V} \cap \Upsilon_\Gamma, \dpathY[\ol{V}])$ is compact.\footnote{Note that $\dpathY[\ol{V}](x,y) = \dpathY(x,y)$ for $\abs{x-y} < \min_i \diamE(\CL_i)$ where $\CL_1,\ldots,\CL_n$ is from the definition of $\metregions$.}

For each open, simply connected $U \subseteq \C$, let $\Gamma_{U^*} \subseteq \Gamma_C$ be the collection of loops that are entirely contained in $U$, and let $U^* \subseteq U \cap C$ be the set of points that are not on or inside any loop of $\Gamma_C\setminus\Gamma_{U^*}$. We will (with a slight abuse of notation) view $U^*$ as the metric space equipped with the metric
\[ d_{U^*}(x,y) = \inf\{ \diamE(\gamma) : \gamma \in \paths{x}{y}{U}{\Gamma\setminus\Gamma_{U^*}} \} . \]
In other words, we include the information how the connected components of $U^*$ are linked together. We further write $\metregions[U] = \{ V \in \metregions : \ol{V} \subseteq U \}$.

Our first main result is a characterization of the Brownian motion on the \clekp{} gasket in terms of its resistance form. We recall that a resistance form $(\CE,\CF)$ is called \emph{regular} if $\CF \cap C_c(F,\FR)$ is dense in $C_c(F,\FR)$ with respect to the supremum norm where $(F,\FR)$ is the resistance metric space associated with $(\CE,\CF)$ and $C_c(F,\FR)$ denotes the space of compactly supported continuous functions on $(F,\FR)$. Together with a measure~$\mu$ whose support is $F$, it gives rise to a regular Dirichlet form on $L^2(F,\mu)$ and thus a $\mu$-symmetric Hunt process (see \cite[Theorem~9.4]{k2012resistance}). A resistance form is called \emph{local} if $\CE(f,g) = 0$ for each $f,g \in \CF$ with $\supp_\FR(f) \cap \supp_\FR(g) = \emptyset$. It is then \emph{strongly local} as a Dirichlet form which means that the Hunt process has continuous sample paths, i.e.\ it is a \emph{diffusion} (see \cite[Theorem~4.3.4]{cf2012dform}).

\begin{figure}[ht]
\centering
\includegraphics[width=0.5\textwidth]{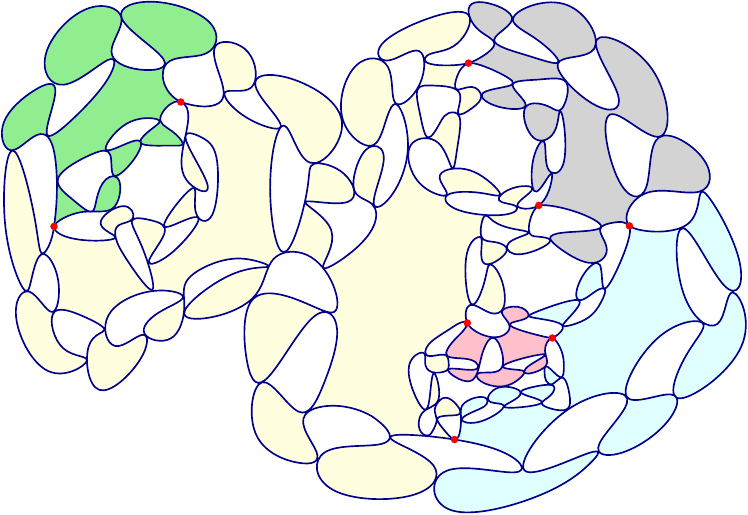}
\caption{Illustration of the locality property of the \clekp{} resistance form. The energy of $f$ is given by the sum of the energies of the restrictions of $f$ to the subregions (indicated with different colors), and the energy of $f$ in each subregion is locally determined by the \clekp{}. The subregions are connected at a finite number of points (indicated by red dots).}
\label{fi:cle_rf_locality}
\end{figure}

\begin{definition}
\label{def:cle_rform}
Suppose that we have the setup described just above. We say that a resistance form $(\rform{\cdot}{\cdot}{\Gamma}, \rfdomain{\Gamma})$ on $\Upsilon_\Gamma$ is a \clekp{} resistance form if the following holds. For each $V \in \metregions$, there exists a resistance form $(\rformres{V}{\cdot}{\cdot}{\Gamma}, \rfdomainres{V}{\Gamma})$ on $\ol{V} \cap \Upsilon_\Gamma$  such that the following properties are satisfied.
\begin{enumerate}[I.]
\item\label{it:rform_resistance} $(\rform{\cdot}{\cdot}{\Gamma}, \rfdomain{\Gamma})$ is a local, regular resistance form, and its associated resistance metric $\rmet{\cdot}{\cdot}{\Gamma}$ is topologically equivalent to $\dpathY$.
\item\label{it:rform_additive} Let $V_1,\ldots,V_n \in \metregions$, $n \in \N$, such that $V_i \cap V_j = \emptyset$ for $i \neq j$ and $\bigcup_i \ol{V_i} \cap \Upsilon_\Gamma = \Upsilon_\Gamma$. Then
 \begin{align*}
  \rfdomain{\Gamma} &= \{ f : f|_{\ol{V_i} \cap \Upsilon_\Gamma} \in \rfdomainres{V_i}{\Gamma} \text{ for each } i \} ,\\
  \text{and}\quad
  \rform{f}{f}{\Gamma} &= \sum_i \rformres{V_i}{f|_{\ol{V_i} \cap \Upsilon_\Gamma}}{f|_{\ol{V_i} \cap \Upsilon_\Gamma}}{\Gamma} .
 \end{align*}
 See Figure~\ref{fi:cle_rf_locality}.
\item\label{it:rform_compatibility} If $V,V' \in \metregions$, $V \subseteq V'$, and $f\colon \ol{V'} \cap \Upsilon_\Gamma \to \R$ is a continuous function that is constant on each $\dpathY$-connected component of $\ol{V'} \cap \Upsilon_\Gamma \setminus \ol{V}$, then
\begin{align*}
 & f \in \rfdomainres{V'}{\Gamma} \text{ if and only if } f|_{\ol{V} \cap \Upsilon_\Gamma} \in \rfdomainres{V}{\Gamma} \\
 \text{and}\quad & \rformres{V'}{f}{f}{\Gamma} = \rformres{V}{f}{f}{\Gamma} .
\end{align*}
\item\label{it:rform_determined} For each open, simply connected $U \subseteq \C$, the collection $(\rformres{V}{\cdot}{\cdot}{\Gamma})_{V \in \metregions[U]}$ is a.s.\ determined by $(\Gamma_{U^*}, \dpathY[U])$.
\item\label{it:rform_translation_invariant} Let $U \subseteq \C$ be open, simply connected, and $z \in \C$. Then
\[ (\rformres{V}{\cdot}{\cdot}{\Gamma_{U^*}})_{V \in \metregions[U]} = (\rformres{V+z}{T_{z}\cdot}{T_{z}\cdot}{\Gamma_{U^*}+z})_{V \in \metregions[U]} \quad\text{a.s.,} \]
where $T_{z} f = f(\cdot -z)$ denotes translation by $z$.
\item\label{it:rform_scale_covariant} There exists a constant $\resexp > 0$ so that the following holds. Let $U \subseteq \C$ be open, simply connected, and $\lambda > 0$. Then
\[ (\rformres{V}{\cdot}{\cdot}{\Gamma_{U^*}})_{V \in \metregions[U]} = (\lambda^{\resexp}\rformres{\lambda V}{S_{\lambda}\cdot}{S_{\lambda}\cdot}{\lambda\Gamma_{U^*}})_{V \in \metregions[U]} \quad\text{a.s.,} \]
where $S_{\lambda} f = f(\lambda^{-1}\cdot)$.
\end{enumerate}
\end{definition}

Condition~\ref{it:rform_additive} can be interpreted as giving that the space $(\Upsilon_\Gamma, \rmet{\cdot}{\cdot}{\Gamma})$ is obtained by ``gluing'' the spaces $(\ol{V_i} \cap \Upsilon_\Gamma, \rmetres{V_i}{\cdot}{\cdot}{\Gamma})$, see Lemma~\ref{le:weights_gluing}. We will see in Proposition~\ref{pr:cle_rmet_char} that the forms $\rformres{V}{\cdot}{\cdot}{\Gamma}$, $V \in \metregions$, are determined by the form $\rform{\cdot}{\cdot}{\Gamma}$.

To give the existence and uniqueness theorem, we start by considering the following particular setup. Let $\Gamma_\D$ be a nested $\CLE_{\kappa'}$ on $\D$, let $\CL$ be the outermost loop of $\Gamma$ such that $0$ is inside $\CL$, and let $C$ be the regions inside $\CL$. We let $\Gamma_C$ be the loops of $\Gamma_\D$ contained in $\ol{C}$ and let $\Gamma = \{ \CL \} \cup \Gamma_C$. We equip the gasket $\Upsilon_\Gamma$ of $\Gamma_C$ with the metric $\dpathY$ defined in~\eqref{eq:dpath}.

\begin{theorem}
\label{thm:unique_metric}
For each $\kappa' \in (4,8)$ there exists a unique (up to a deterministic factor) \clekp{}~resistance form in the sense of Definition~\ref{def:cle_rform}. The exponent $\resexp$ depends only on $\kappa'$ and satisfies $\resexp \in [\ddouble,\dsle]$.
\end{theorem}

In fact, we will prove that the \clekp{} resistance form is already determined by a weaker list of assumptions; we will call them \emph{weak} \clekp{} resistance forms, see Definition~\ref{def:weak_cle_rform}. We will show that every weak \clekp{} resistance form in fact satisfies the stronger list of properties in Definition~\ref{def:cle_rform}.

As explained above, a \clekp{} resistance form determines a Markov process once we specify a measure. It was shown in \cite{ms2022clemeasure} that there is a canonical measure $\meas{\cdot}{\Gamma}$ on the gasket $\Upsilon_\Gamma$ that is conformally covariant and Markovian (see Section~\ref{se:cle_measure}), and this will be our canonical choice of the measure. By the locality of the \clekp{} resistance form, the associated process is a \emph{diffusion} which by definition is a Hunt process with continuous sample paths (up to its extinction time). We now give a direct characterization of this diffusion process.

\begin{definition}
\label{def:cle_brownian_motion}
Suppose that we have the setup described above where $\p$ denotes the law of $\Gamma$. We call a continuous Markov process $\clebm$ with law $(\clebmlaw{x})_{x \in \Upsilon_\Gamma}$ on the gasket $\Upsilon_\Gamma$ a \clekp{} Brownian motion if it satisfies the following properties.
\begin{enumerate}[I.]
\item It is a symmetric diffusion process with infinite lifetime on $(\Upsilon_\Gamma,\dpathY)$ with respect to the gasket measure $\meas{\cdot}{\Gamma}$.
\item Its Dirichlet form is a resistance form, and its resistance metric is topologically equivalent to~$\dpathY$.
\item For each open, simply connected $U \subseteq \C$, consider the process $\clebm_U$ obtained from killing it upon exiting $U$. The law of this Markov process is $\p$-a.s.\ determined by $(\Gamma_{U^*}, \dpathY[U])$.
\item Let $F_U$ be a measurable function that associates to $(\Gamma_{U^*}, \dpathY[U])$ the Markov process $\clebm_U$.
\begin{enumerate}[(a)]
\item For each open, simply connected $U \subseteq \C$ and $z \in \C$, the functions $F_U$, $F_{U+z}$ can be chosen so that $F_{U+z} = (T_z)_* F_U(\cdot -z)$ where $T_z \clebm(t) = \clebm(t)+z$ denotes translation by $z$.
\item There exists a constant $\resexp > 0$ so that the following holds. For each open, simply connected $U \subseteq \C$ and $\lambda > 0$, the functions $F_U$, $F_{\lambda U}$ can be chosen so that $F_{\lambda U} = (S_\lambda)_* F_U(\lambda^{-1}\cdot)$ where $S_\lambda \clebm(t) = \lambda\clebm(\lambda^{-\dcle-\resexp}t)$.
\end{enumerate}
\end{enumerate}
\end{definition}

\begin{theorem}
\label{thm:unique_process}
Consider the same setup as described above Theorem~\ref{thm:unique_metric}. For each $\kappa' \in (4,8)$ there exists a unique (up to time-change by a deterministic factor) \clekp{} Brownian motion in the sense of Definition~\ref{def:cle_brownian_motion}. Its associated resistance form is a \clekp{} resistance form in the sense of Definition~\ref{def:cle_rform}.
\end{theorem}

We will show in Section~\ref{se:non_conf_inv} that the \clekp{} Brownian motion is \emph{not} conformally invariant. This is not surprising considering that the clusters have dimension given in~\eqref{eqn:cle_dim} which is strictly smaller than $2$.

By \cite[Theorem~10.4]{k2012resistance}, for each compact resistance metric space and each finite Borel measure with full support the associated Hunt process possesses a continuous heat kernel. Moreover, by \cite[Theorem~1]{cro-heat-kernel}, the volume growth $\meas{\Bpath(x,r)}{\Gamma} = r^{\dcle+o(1)}$ (see Propositions~\ref{pr:measure_ub}--\ref{pr:measure_lb}) and resistance growth $\rmet{x}{y}{\Gamma} = \dpathY(x,y)^{\resexp+o(1)}$ (see Lemmas~\ref{le:resistance_ub_cle}--\ref{le:resistance_lb_cle}) imply the following on-diagonal heat kernel bounds.

\begin{proposition}
 Let $\resexp$ be the exponent in Theorem~\ref{thm:unique_metric}. Almost surely, the heat kernel of the \clekp{} Brownian motion satisfies
 \[
 p(t,x,x) = t^{-\dcle/(\dcle+\resexp)+o(1)}
 \quad\text{for each } x \in \Upsilon_\Gamma \quad\text{as } t \searrow 0 .
 \]
 That is, the spectral dimension of the \clekp{} Brownian motion is $2\dcle/(\dcle+\resexp)$.
\end{proposition}

We now explain that the setup can be generalized to define the \clekp{} Brownian motion (equivalently \clekp{} resistance form) on each \emph{interior} cluster of a \clekp{} in a simply connected domain or of a whole-plane \clekp{}. This is because the process is determined by the local geometry of the cluster, and the laws of the local geometry of each cluster near each point (away from the domain boundary) are mutually absolutely continuous. Hence, the behavior of the process can be deduced from the special case considered above Theorem~\ref{thm:unique_metric}.

We give a precise statement. In the following, we let $\Gamma^D$ be a nested \clekp{} on a simply connected domain $D \subseteq \C$ or a whole-plane \clekp{}. We consider the countable collection of its gaskets, i.e.\ the $\dpathY[D]$-completions of the spaces of points that do not lie on any loop of $\Gamma^D$. Here, the exterior gasket (i.e.\ the one with $\partial D$ as its exterior boundary) is split by $\partial D$ into a countable number of gaskets. Let $(\Upsilon_k)_{k \in \N}$ be an enumeration of all the gaskets of $\Gamma^D$. For each open $U \subseteq \C$ we let $\Gamma^D|_U$ denote the countable collection of loops and maximal strands of $\Gamma^D$ that are contained in $U$. We define $\metregions[k]$ and $\metregions[k,U]$ as above for each cluster $\Upsilon_k$, and we let $\metregions = \bigcup_{k \in \N} \metregions[k]$ and $\metregions[U] = \bigcup_{k \in \N} \metregions[k,U]$.

\begin{theorem}\label{th:all_clusters_wp}
Consider the setup described just above where we let $\Gamma^\C$ be a nested whole-plane~\clekp{}. There exists a unique (up to a deterministic factor) collection of resistance forms $\rform{\cdot}{\cdot}{\Gamma^\C} = ((\rformres{k}{\cdot}{\cdot}{\Gamma^\C}, \rfdomainres{\Upsilon_k}{\Gamma^\C}))_{k \in \N}$ that is measurable with respect to $\Gamma^\C$ and satisfies the following properties. There is a further collection of resistance forms $((\rformres{V}{\cdot}{\cdot}{\Gamma^\C}, \rfdomainres{V}{\Gamma^\C}))_{V \in \metregions}$ such that the following hold.
\begin{enumerate}[I.]
 \item For each $k \in \N$ we have that $(\rformres{k}{\cdot}{\cdot}{\Gamma^\C}, \rfdomainres{\Upsilon_k}{\Gamma^\C})$ is a local, regular resistance form on $\Upsilon_k$. Its associated resistance metric $\rmetres{k}{\cdot}{\cdot}{\Gamma^\C}$ is topologically equivalent to $\dpathY[\Upsilon_k]$.
 \item Let $V_1,\ldots,V_n \in \metregions[k]$, $n \in \N$, such that $V_i \cap V_j = \emptyset$ for $i \neq j$ and $\bigcup_i \ol{V_i} \cap \Upsilon_k = \Upsilon_k$. Then
 \begin{align*}
  \rfdomainres{\Upsilon_k}{\Gamma^\C} &= \{ f : f|_{\ol{V_i} \cap \Upsilon_k} \in \rfdomainres{V_i}{\Gamma^\C} \text{ for each } i \} ,\\
  \text{and}\quad
  \rformres{k}{f}{f}{\Gamma^\C} &= \sum_i \rformres{V_i}{f|_{\ol{V_i} \cap \Upsilon_k}}{f|_{\ol{V_i} \cap \Upsilon_k}}{\Gamma^\C} .
 \end{align*}
\item If $V,V' \in \metregions[k]$, $V \subseteq V'$, and $f\colon \ol{V'} \cap \Upsilon_k \to \R$ is a continuous function that is constant on each $\dpathY$-connected component of $\ol{V'} \cap \Upsilon_k \setminus \ol{V}$, then
\begin{align*}
 & f \in \rfdomainres{V'}{\Gamma^\C} \text{ if and only if } f|_{\ol{V} \cap \Upsilon_k} \in \rfdomainres{V}{\Gamma^\C} \\
 \text{and}\quad & \rformres{V'}{f}{f}{\Gamma^\C} = \rformres{V}{f}{f}{\Gamma^\C} .
\end{align*}
 \item For each open set $U \subseteq \C$, the collection of forms $(\rformres{V}{\cdot}{\cdot}{\Gamma^\C})_{V \in \metregions[U]}$ is a.s.\ a measurable function of $\Gamma^\C|_U$.
\item For each $z \in \C$, we have
\[ \rform{\cdot}{\cdot}{\Gamma^\C} = \rform{T_{z}\cdot}{T_{z}\cdot}{\Gamma^\C+z} \quad\text{a.s.,} \]
where $T_{z} f = f(\cdot -z)$ denotes translation by $z$.
\item There exists a constant $\resexp > 0$ so that for each $\lambda > 0$, we have
\[ \rform{\cdot}{\cdot}{\Gamma^\C} = \lambda^{\resexp}\rform{S_{\lambda}\cdot}{S_{\lambda}\cdot}{\lambda\Gamma^\C} \quad\text{a.s.,} \]
where $S_{\lambda} f = f(\lambda^{-1}\cdot)$.
\end{enumerate}
The analogous statement holds for the \clekp{} Brownian motion.
\end{theorem}

\begin{theorem}\label{th:all_clusters_domain}
Consider the setup described just above. There exists a unique (up to a deterministic factor) collection of measurable functions $\Gamma^D \mapsto \rform{\cdot}{\cdot}{\Gamma^D}$ for each simply connected domain $D \subseteq \C$ where each $\rform{\cdot}{\cdot}{\Gamma^D} = ((\rformres{k}{\cdot}{\cdot}{\Gamma^D}, \rfdomainres{\Upsilon_k}{\Gamma^D}))_{k \in \N}$ is a collection of resistance forms with the following properties. There is a further collection of resistance forms $((\rformres{V}{\cdot}{\cdot}{\Gamma^D}, \rfdomainres{V}{\Gamma^D}))_{V \in \metregions}$ such that the following hold.
\begin{enumerate}[I.]
 \item For each $k \in \N$ we have that $(\rformres{k}{\cdot}{\cdot}{\Gamma^D}, \rfdomainres{\Upsilon_k}{\Gamma^D})$ is a local, regular resistance form on $\Upsilon_k$. Its associated resistance metric $\rmetres{k}{\cdot}{\cdot}{\Gamma^D}$ is topologically equivalent to $\dpathY[\Upsilon_k]$.
 \item Let $V_1,\ldots,V_n \in \metregions[k]$, $n \in \N$, such that $V_i \cap V_j = \emptyset$ for $i \neq j$ and $\bigcup_i \ol{V_i} \cap \Upsilon_k = \Upsilon_k$. Then
 \begin{align*}
  \rfdomainres{\Upsilon_k}{\Gamma^D} &= \{ f : f|_{\ol{V_i} \cap \Upsilon_k} \in \rfdomainres{V_i}{\Gamma^D} \text{ for each } i \} ,\\
  \text{and}\quad
  \rformres{k}{f}{f}{\Gamma^D} &= \sum_i \rformres{V_i}{f|_{\ol{V_i} \cap \Upsilon_k}}{f|_{\ol{V_i} \cap \Upsilon_k}}{\Gamma^D} .
 \end{align*}
\item If $V,V' \in \metregions[k]$, $V \subseteq V'$, and $f\colon \ol{V'} \cap \Upsilon_k \to \R$ is a continuous function that is constant on each $\dpathY$-connected component of $\ol{V'} \cap \Upsilon_k \setminus \ol{V}$, then
\begin{align*}
 & f \in \rfdomainres{V'}{\Gamma^D} \text{ if and only if } f|_{\ol{V} \cap \Upsilon_k} \in \rfdomainres{V}{\Gamma^D} \\
 \text{and}\quad & \rformres{V'}{f}{f}{\Gamma^D} = \rformres{V}{f}{f}{\Gamma^D} .
\end{align*}
 \item For each open set $U \subseteq \C$, the collection of forms $(\rformres{V}{\cdot}{\cdot}{\Gamma^D})_{V \in \metregions[U]}$ is a.s.\ a measurable function of $\Gamma^D|_U$ that does not depend on the choice of $D$.
\item For each simply connected domain $D \subseteq \C$ and $z \in \C$, we have
\[ \rform{\cdot}{\cdot}{\Gamma^D} = \rform{T_{z}\cdot}{T_{z}\cdot}{\Gamma^D+z} \quad\text{a.s.,} \]
where $T_{z} f = f(\cdot -z)$ denotes translation by $z$.
\item There exists a constant $\resexp > 0$ so that the following holds. For each simply connected domain $D \subseteq \C$ and $\lambda > 0$, we have
\[ \rform{\cdot}{\cdot}{\Gamma^D} = \lambda^{\resexp}\rform{S_{\lambda}\cdot}{S_{\lambda}\cdot}{\lambda\Gamma^D} \quad\text{a.s.,} \]
where $S_{\lambda} f = f(\lambda^{-1}\cdot)$.
\end{enumerate}
The analogous statement holds for the \clekp{} Brownian motion.
\end{theorem}

The equivalence of Theorems~\ref{th:all_clusters_wp} and~\ref{th:all_clusters_domain} with Theorem~\ref{thm:unique_metric} resp.~\ref{thm:unique_process} follows by the exact same argument as in \cite[Theorems~1.7 and~1.8]{my2025geouniqueness}.

\subsection{Related work}

Random walks on percolation models and diffusion processes on fractals have been subject to extensive research. We review some of the literature related to the present work.

\subsubsection*{Two-dimensional critical percolation}

One of the motivations for this work is to gain more understanding on the simple random walk on two-dimensional critical percolation which has remained poorly understood to date. The natural setup for studying the large-scale geometry of critical percolation is the \emph{incipient infinite cluster} (IIC) which is obtained by conditioning the percolation cluster of the origin to be infinite \cite{kes-iic}. In this model, Kesten \cite{kes-rw-percolation} proved that the random walk on the IIC is subdiffusive with respect to Euclidean distance. The work \cite{gl2022chemicalsub} improves the result by showing that it is subdiffusive with respect to the chemical distance metric (which by \cite{ab1999holder} is strictly larger than the Euclidean metric). Another interesting number measuring the geometry of the cluster is the spectral dimension of the walk. The spectral dimension for the IIC is estimated by simulations to be approximately $1.318$ \cite{bah-diffusion-fractals}. (A conjecture $d_s = 4/3$ by Alexander and Orbach \cite{ao-spectral} is believed to be false in dimensions smaller than $6$; see~\cite[Section~7.4]{kum-rw-book}.)

\subsubsection*{High-dimensional percolation}

The random walk and its resistance metric in critical percolation in high dimensions are studied in \cite{bjks-rw-oriented-percolation,kn-alexander-orbach,hhh-rw-iic,cchs-chemical-high-dim}. In high dimensions, the critical percolation behaves very differently, one main difference being that the clusters are tree-like. In that case, it is established that the Alexander-Orbach conjecture $d_s = 4/3$ does hold. The recent work \cite{acf-rw-lattice-trees} establishes a scaling limit result for the random walk for a related model that is expected to resemble high-dimensional critical percolation.

\subsubsection*{Supercritical percolation}

In the case of supercritical percolation much more is known. This model is simpler since it is close to Euclidean space. The random walk on supercritical percolation clusters converges to Brownian motion on the Euclidean space; several variants of this result have been established in \cite{bar-rw-supercritical,ss-rw-percolation,bb-rw-percolation,mp-rw-percolation}.

\subsubsection*{Uniform spanning tree}

The scaling limit of the uniform spanning tree in two dimensions has been established in \cite{lsw2004lerw,lv2021natural,hs2018euclideanmating} and is given by the \cle{8}. The work \cite{bck2017tightness} proves the convergence of the random walk to the Brownian motion on the continuum tree. Note that in this case the construction of the Brownian motion from the metric space is trivial since the space is a tree.

\subsubsection*{Brownian motion on deterministic fractals}

For some deterministic fractals (including the Sierpi\'nski gasket and carpet), an extensive literature has developed which is aimed at constructing and characterizing the Brownian motions on these fractals, see \cite{gol-gasket,kus-gasket,bp-gasket,fs-gasket,lin-nested-fractals,kig-pcf,sab-nested-fractals,metz-nested-fractals,pei-pcf,hmt-pcf,bb-carpet,bb-carpet-hd,kz-carpet,bbkt-carpet}.

\subsection{Outline}

We will now give an overview of how the remainder of this article is structured. In Section~\ref{sec:preliminaries}, we will collect a number of preliminaries. In particular, we review the definition and basic properties of SLE and CLE, the gasket measure and the geodesic metric on the CLE gasket, and finally give an overview of resistance forms and resistance metrics.

The main part of the paper is dedicated to proving Theorem~\ref{thm:unique_metric}. This will consist of two main steps:
\begin{enumerate}[1.]
\item tightness of a certain approximation procedure for the resistance metric, and
\item the uniqueness of the subsequential limits.
\end{enumerate}
The proof of tightness will be given in Section~\ref{sec:weak_cle_rform} where we also construct a particular approximation of the \clekp{} gasket by a sequence of graphs. The main external input that we will make use of is the main result from \cite{amy2025tightness}; we will recall the statement and the general setup in Section~\ref{sec:tightness_general}. This will establish that the family of associated resistance metrics properly normalized is tight, and we will further prove that each subsequential limit is a resistance metric whose resistance form satisfies the definition of a weak \clekp{} resistance form given in Definition~\ref{def:weak_cle_rform}. (At this point, we have not proved that they satisfy the stronger Definition~\ref{def:cle_rform}.)

The proof of the uniqueness of weak \clekp{} resistance forms will be divided into Sections~\ref{sec:bilipschitz} and~\ref{sec:uniqueness}. We show in Section~\ref{sec:bilipschitz} that any two weak \clekp{} resistance forms are bi-Lipschitz equivalent with deterministic constants. We then prove in Section~\ref{sec:uniqueness} that they are scalar multiples of each other by showing that the optimal bi-Lipschitz constants can be chosen to be equal. We then conclude that each weak \clekp{} resistance form satisfies the stronger Definition~\ref{def:cle_rform}, completing the proof of Theorem~\ref{thm:unique_metric}. The general principle in both Section~\ref{sec:bilipschitz} and~\ref{sec:uniqueness} is that we can divide the gasket into small regions which are approximately i.i.d., so that each weak \clekp{} resistance form behaves like the average over these small regions. To make this argument rigorous, we will make use of the independence across scales property established in \cite{amy-cle-resampling} to find a covering of the space by small regions in which the resistance forms are locally comparable to their average. We will show along the way that the resistance form is essentially determined by a dense collection of regions each of which are bounded between two intersecting \clekp{} loops.

Finally, we show in Section~\ref{se:diffusion} the equivalence between the characterization of the \clekp{} resistance form in Definition~\ref{def:cle_rform} and the characterization of the \clekp{} Brownian motion in Definition~\ref{def:cle_brownian_motion}, thus showing Theorem~\ref{thm:unique_process}. We will also show that the \clekp{} Brownian motion is \emph{not} conformally invariant.

\subsection*{Notation}

For $z\in\C$ and $r_2>r_1>0$, we denote $A(z,r_1,r_2) = B(z,r_2)\setminus\ol{B}(z,r_1)$. Let $\bIn A(z,r_1,r_2) = \partial B(z,r_1)$ and $\bOut A(z,r_1,r_2) = \partial B(z,r_2)$. For a bounded set $A \subseteq \C$ we let $\Fill(A)$ denote the union of $A$ with the bounded connected components of $\C \setminus A$. We write $U \Subset V$ to denote that $U$ is compactly contained in $V$, i.e.\ $\ol{U}$ is compact and $\ol{U} \subseteq V$. We sometimes write $\distE,\diamE,B_E$, etc.\ to refer to the Euclidean metric (as opposed to the CLE metric or the $\dpathY$ metric).

For a function $f$ depending on $\delta \in (0,1]$, we write $f(\delta) = O(\delta^b)$ if $f$ is bounded by a constant times $\delta^b$.  We also write $f(\delta) = o^\infty(\delta)$ if $f(\delta) = O(\delta^b)$ for any fixed $b>0$. We write $a \lesssim b$ if $a \le cb$ for a constant $c>0$, and we write $a \asymp b$ if $a \lesssim b$ and $b \lesssim a$.

We assume throughout the paper that $\kappa' \in (4,8)$ is fixed. We always consider nested \clekp{} in the proofs (which is required to use the resampling tools from \cite{amy-cle-resampling}; see Section~\ref{se:mcle}).

For the resistance form $\rformres{V}{\cdot}{\cdot}{\Gamma}$ whose domain $\rfdomainres{V}{\Gamma}$ consists of functions defined on $\ol{V} \cap \Upsilon_\Gamma$, we will often write $\rformres{V}{f}{f}{\Gamma}$ for $\rformres{V}{f|_{\ol{V} \cap \Upsilon_\Gamma}}{f|_{\ol{V} \cap \Upsilon_\Gamma}}{\Gamma}$ when $f$ is a function defined on a larger domain (e.g.\ $\Upsilon_\Gamma$ or $\ol{V_1} \cap \Upsilon_\Gamma$ with $V_1 \supseteq V$). We will also write $f|_V$ for $f|_{\ol{V} \cap \Upsilon_\Gamma}$.

\subsection*{Acknowledgements}
J.M. and Y.Y.\ were supported by ERC starting grant SPRS (804116) and from ERC consolidator grant ARPF (Horizon Europe UKRI G120614). Y.Y.\ in addition received support from the Royal Society.

\section{Preliminaries}
\label{sec:preliminaries}

\subsection{Schramm-Loewner evolution}
\label{subsec:sle}

The Schramm-Loewner evolution ($\SLE_\kappa$) was introduced by Schramm in 1999 \cite{s2000sle} as a candidate to describe the scaling limit of discrete models from statistical mechanics in two dimensions. For a simply connected domain $D \subsetneq \C$ and $x,y \in \partial D$ distinct, for each $\kappa \ge 0$, the \slek{} in $D$ from $x$ to $y$ is a random non-self-crossing curve in $\ol{D}$ from $x$ to $y$. The \slek{} in $\h$ from $0$ to $\infty$ is characterized by the following property. There is a parameterization so that if for each $t \ge 0$ we let $\h_t$ be the unbounded connected component of $\h \setminus \eta([0,t])$ and $g_t$ is the unique conformal map $\h_t \to \h$ with $g_t(z) - z \to 0$ as $z \to \infty$, then
\begin{equation}
\label{eqn:loewner_ode}
\partial_t g_t(z) = \frac{2}{g_t(z) - U_t} \quad\text{for } z \in \h_t, \quad g_0(z) = z ,
\end{equation}
where $U_t = \sqrt{\kappa}B_t$ and $B$ is a standard real-valued Brownian motion. (The existence of the $\SLE_\kappa$ curve was proved for $\kappa \neq 8$ in \cite{rs2005basic}, for $\kappa = 8$ in \cite{lsw2004lerw} using the convergence of the UST Peano curve to $\SLE_8$, and alternatively for $\kappa = 8$ using continuum methods in \cite{am2022sle8}). The \slek{} in a general simply connected domain $D$ from $x$ to $y$ is then defined as the image of an \slek{} in $\h$ from~$0$ to~$\infty$ under any conformal map $D \to \h$ that takes $0$ to $x$ and $y$ to $\infty$.

The \slekr{\rho} processes are an important variant of \slek{} first introduced in \cite[Section~8.3]{lsw2003restriction}. They depend on an additional marked point $v_0 \in \partial\h$ (called the \emph{force point}) and $\rho$ (the \emph{weight} of the force point). In the case $\rho > -2$, the \slekr{\rho} in $\h$ from $0$ to $\infty$ with force point at $v_0$ is constructed by solving~\eqref{eqn:loewner_ode} where $U$ is taken to be the solution of the SDE
\begin{equation}
\label{eqn:sle_kappa_rho_sde}
 dU_t = \sqrt{\kappa} dB_t + \frac{\rho}{U_t - V_t} dt,\quad dV_t = \frac{2}{V_t - U_t} dt,\quad V_0 = v_0.
\end{equation}
If $D \subsetneq \C$ is another simply connected domain and $x,y,z \in \partial D$, the $\SLE_\kappa(\rho)$ process in $D$ from $x$ to $y$ with force point at $z$ is given by the image under the unique conformal map $\h \to D$ that takes $0$, $\infty$, and $v_0$ to $x$, $y$, and $z$, respectively. The existence of the $\SLE_\kappa(\rho)$ curve was proved in \cite{ms2016ig1}.

\subsection{Conformal loop ensembles}
\label{subsec:cle}

The conformal loop ensembles ($\CLE_\kappa$) are the loop version of $\SLE_\kappa$.  They were introduced in \cite{s2009cle,sw2012cle} and aim to describe the joint law of the scaling limit of all of the interfaces in a discrete model from statistical mechanics in two dimensions, whereas an $\SLE_\kappa$ only describes the scaling limit of a single interface.  A $\CLE_\kappa$ consists of a countable collection of loops each of which locally look like an $\SLE_\kappa$.  The $\CLE_\kappa$ are defined for $\kappa \in [8/3,8]$.  At the extreme $\kappa =8/3$ it is the empty collection of loops and for $\kappa = 8$ it is a single space-filling loop.  In this paper, we will be focused on the regime $\kappa' \in (4,8)$.  In what follows, we will recall the construction and basic facts regarding \clekp{}, then the so-called partial explorations of \clekp{}, and then resampling operations that one can perform on \clekp{}.

\subsubsection{Definition of \clekp{}}

We now review an explicit construction of \clekp{} for $\kappa' \in (4,8)$. Suppose that the domain is $\h$, and we fix the point $0 \in \partial\h$, referred to as the \emph{root} of the exploration. Let $\eta'$ be an \slekpr{\kappa'-6} from $0$ to $\infty$ with a force point at $0^+$. Suppose that $0 < \sigma < \tau$ are such that $\eta'(\sigma),\eta'(\tau) \in \R_+$ and $\eta'((\sigma,\tau)) \cap \R_+ = \emptyset$. For each such $(\sigma,\tau)$, we let $\CL_{\sigma,\tau}$ be the loop consisting of $\eta'([\sigma,\tau])$ followed by a conditionally independent \slekp{} from $\eta'(\tau)$ to $\eta'(\sigma)$ in the complementary connected component. Given the collection of loops $\CL_{\sigma,\tau}$ just constructed, we repeat the procedure in each complementary connected component of the union of these loops. Let $\Gamma$ be the countable collection of loops constructed from each iteration of this procedure.

It is shown in \cite{s2009cle,ms2016ig3} that the law of $\Gamma$ does not depend on the choice of the root, and is invariant under conformal transformations of the domain. We call $\Gamma$ the \emph{nested \clekp{}}, and we call the collection of outermost loops of $\Gamma$ (i.e.\ when we do not repeat the exploration in the components inside the loops) the \emph{non-nested \clekp{}}.

The \emph{whole-plane nested \clekp{}} is constructed in \cite{mww2016extreme} as the limit of nested \clekp{} in any increasing sequence of domains $D_n$ with $\dist(0,\partial D_n) \to \infty$.

\subsubsection{Multichordal $\CLE_{\kappa'}$ and partial explorations}
\label{se:mcle}

We are now going to recall some of the results from \cite{amy-cle-resampling} about partially explored $\CLE_{\kappa'}$, as they will be used extensively in the present article.  Let us begin by recalling \cite[Definition~1.1]{amy-cle-resampling}.
\begin{definition}
\label{def:marked_domain}
Let $D \subsetneq \C$ be a simply connected domain.  Fix $N \in \N_0$ and let $\ul{x} = (x_1,\ldots,x_{2N})$ be a collection of distinct prime ends which are ordered counterclockwise.
\begin{enumerate}[(i)]
\item We call $(D;\ul{x})$ a \emph{marked domain}. Let $\markeddomain{2N}$ be the collection of marked domains with $2N$ marked points.
\item Suppose that $(D;\ul{x}) \in \markeddomain{2N}$, and we have a non-crossing collection of paths $\ul{\gamma} = (\gamma_1,\ldots,\gamma_N)$ outside of $D$ where each $\gamma_i$ connects a distinct pair of marked points. The paths $\ul{\gamma}$ induce a planar link pattern $\beta$ on the exterior of $D$, formally described by a partition of the marked points into pairs $\beta = \{\{a_1,b_1\},\ldots,\{a_N,b_N\}\}$ where $\{a_1,b_1,\ldots,a_N,b_N\} = \{1,\ldots,2N\}$, $a_r<b_r$, and such that the configuration $a_r<a_s<b_r<b_s$ does not occur. We let $\eldomain{2N}$ be the collection of exterior link pattern decorated marked domains.
\item Suppose that $(D;\ul{x}) \in \markeddomain{2N}$, and we have a non-crossing collection of paths $\ul{\eta} = (\eta_1,\ldots,\eta_N)$ in $D$ where each $\eta_i$ connects a distinct pair of marked points. We will also consider the planar link pattern $\alpha$ on the interior of $D$ induced by the paths $\ul{\eta}$.
\end{enumerate}
\end{definition}

We review the definition of the multichordal \clekp{} in a marked domain $(D;\ul{x};\beta) \in \eldomain{2N}$ whose law we will denote by $\mcclelaw{D;\ul{x};\beta}$. A sample from $\mcclelaw{D;\ul{x};\beta}$ consists of a pair $(\ul{\eta}, \Gamma)$ where $\ul{\eta} = (\eta_1,\ldots,\eta_N)$ is a family of curves that connect the marked points $\ul{x}$ and, given $\ul{\eta}$, the collection $\Gamma$ consists of a conditionally independent nested $\CLE_{\kappa'}$ in each of the complementary components of the~$\ul{\eta}$. (Sometimes we write just $\Gamma$ to denote both $(\ul{\eta},\Gamma)$.) In the case $N=1$ which we call a \emph{monochordal~\clekp{}}, the law of $\eta_1$ is given by an \slekp{} in $D$ from $x_1$ to $x_2$. In the case $N=2$ which we call a \emph{bichordal \clekp{}}, the interior link pattern $\alpha$ between the four points $x_1,\ldots,x_4$ is sampled from an explicit law determined in \cite{msw2020nonsimple,mw2018connection} (see \cite[equation~(1.1)]{amy-cle-resampling}), and given $\alpha$, the two chords $(\eta_1,\eta_2)$ are sampled from the unique law such that the conditional law of $\eta_1$ given $\eta_2$ is that of a \slekp{} in the complementary component and vice versa. It is shown in \cite{msw2020nonsimple} that the law of $(\eta_1,\eta_2)$ is characterized as the unique invariant law of a certain Markov chain. For general $N \ge 2$, the law of $\ul{\eta}$ under the multichordal \clekp{} is the unique law that is invariant under the following resampling kernels: Fix any $1 \leq i_1 < i_2 < i_3 < i_4 \leq 2N$ and condition on the $\eta_k$ that do not start at $x_{i_1},\ldots,x_{i_4}$. If there are exactly two chords connecting these four points and they are in the same complementary component of the remaining chords, resample the two chords according to the law of a bichordal \clekp{} with exterior link pattern induced by $\beta$ and the remaining chords.

The multichordal \clekp{} law is conformally invariant, i.e., if $\varphi\colon D \to \wt{D}$ is a conformal transformation, then the law $\mcclelaw{\wt{D};\varphi(\ul{x});\beta}$ is given by the law of $\varphi(\Gamma)$ under $\mcclelaw{D;\ul{x};\beta}$.

It is shown in \cite[Theorem~1.11]{amy-cle-resampling} that multichordal \clekp{} arises naturally from partial explorations of $\CLE_{\kappa'}$.  Suppose that $D \subseteq \C$ is a simply connected domain and let $\varphi \colon D \to \D$ a conformal transformation.  Let $\domainpair{D}$ consist of all pairs $(U,V)$ of simply connected domains $U \subseteq V \subseteq D$ with $\dist(\varphi(D \setminus V), \varphi(U)) > 0$.  (Note that $\domainpair{D}$ does not depend on $\varphi$.)  Suppose that $\Gamma$ is a nested $\CLE_{\kappa'}$ in $D$ and $(U,V) \in \domainpair{D}$.   Let $\Gamma_\outside^{*,V,U}$ be the collection of maximal arcs of loops of $\Gamma$ that are disjoint from $U$ and intersect $D \setminus V$.  We call $\Gamma_\outside^{*,V,U}$ the \emph{partial exploration} of the loops of $\Gamma$ intersecting $D \setminus V$ up until they hit $U$. Let $V^{*,U}$ be the connected component containing $U$ after removing the loops and strands of $\Gamma_\outside^{*,V,U}$, and let $\Gamma_\inside^{*,V,U}$ denote the remainder of $\Gamma$ in $V^{*,U}$.

\begin{theorem}[{\cite[Theorem~1.11]{amy-cle-resampling}}]\label{thm:cle_partially_explored}
In the setup described above, the conditional law given $\Gamma_\outside^{*,V,U}$ of the remainder $\Gamma_\inside^{*,V,U}$ has the law of a multichordal $\CLE_{\kappa'}$ in $V^{*,U}$ where the exterior link pattern $\beta$ is induced by the arcs of~$\Gamma_\outside^{*,V,U}$.
\end{theorem}

It is also shown in \cite{amy2025tightness} that the law of a multichordal $\CLE_{\kappa'}$ is continuous as one moves the marked points and the domain. We recall two such statements. The next proposition concerns the link pattern $\alpha$ induced by chords $\ul{\eta}$ of a multichordal \clekp{}.

\begin{proposition}[{\cite[Theorem~1.6(iv)]{amy-cle-resampling}}]\label{pr:link_probability}
For each $(D;\ul{x};\beta) \in \eldomain{2N}$ and each interior link pattern $\alpha_0$, we have $\mcclelaw{D;\ul{x};\beta}[\alpha = \alpha_0] > 0$, and this probability depends continuously on $\ul{x}$.
\end{proposition}

We say that a sequence of $(D_n;\ul{x}_n) \in \markeddomain{2N}$ converges to $(D;\ul{x}) \in \markeddomain{2N}$ in the Carathéodory sense if $\varphi_n \to \varphi$ locally uniformly and $\varphi_n^{-1}(\ul{x}_n) \to \varphi^{-1}(\ul{x})$ pointwise where $\varphi_n\colon \D \to D_n$ (resp.\ $\varphi\colon \D \to D$) is the unique conformal map with $\varphi_n(0) = z_0$, $\varphi_n'(0) > 0$ (resp.\ $\varphi(0) = z_0$, $\varphi'(0) > 0$) for some fixed $z_0 \in D$.

\begin{proposition}[{\cite[Proposition~6.1]{amy-cle-resampling}}]
\label{prop:mccle_tv_convergence_int}
Let $(D;\ul{x};\beta) \in \eldomain{2N}$, and suppose that $((D_n;\ul{x}_n))$ is a sequence in $\markeddomain{2N}$ converging to $(D;\ul{x})$ in the Carathéodory sense. For each $n$ let $\Gamma_n$ have law $\mcclelaw{D_n;\ul{x}_n;\beta}$, and let $\Gamma$ have law $\mcclelaw{D;\ul{x};\beta}$. Let $(U,V)\in\domainpair{D}$ with $V \Subset D$, then the law of $(\Gamma_n)_\inside^{*,V,U}$ converges to the law of $\Gamma_\inside^{*,V,U}$ in total variation as $n \to \infty$.
\end{proposition}

\subsubsection{Resampling target pivotals}
\label{subsubsec:target_pivotals}

We will now describe the results of \cite[Section~5]{amy-cle-resampling}, which is focused on using a ``resampling procedure'' that alters the connections in a (multichordal) \clekp{}.  Let us now describe more precisely what we mean by a \emph{resampling procedure}.  Suppose that $\Gamma$ is a multichordal \clekp{} in $(D;\ul{x};\beta)$, and let $W \subseteq D$.  Then we can construct another multichordal \clekp{} in $(D;\ul{x};\beta)$ coupled with $\Gamma$ by repeatedly applying the following procedure. We select $(U,V) \in \domainpair{D}$ with $V \subseteq W$ randomly in a way that is independent of the CLE configuration within $W$, and let $\wt{\Gamma}$ be such that $\wt{\Gamma}_\outside^{*,V,U} = \Gamma_\outside^{*,V,U}$, and $\wt{\Gamma}_\inside^{*,V,U}$ is sampled from its conditional law given $\wt{\Gamma}_\outside^{*,V,U}$ (independently of $\Gamma_\inside^{*,V,U}$). We say that $\Gamma^\resampled$ is a \emph{resampling of $\Gamma$ within $W$} if it is obtained by applying such a procedure a (random) number of times.

Intersection points between two strands of loops are called \emph{pivotal points} because changing the connections between them changes the macroscopic behavior of the loops. Suppose that we select small regions randomly and resample the CLE in these small regions. If we are given a finite choice of pivotal points, by Proposition~\ref{pr:link_probability} there is a positive probability that the resampling changes the connections at the desired pivotal points, and the new CLE is the same as the old one away from these small regions. The results in \cite[Section~5]{amy-cle-resampling} give uniform bounds for the success probabilities. We can use this to compare probabilities of certain events to probabilities where the loops are additionally required to be linked in some particular way. (This tool was already extensively used in \cite{amy2025tightness}.)

We recall one particular application where we aim to break the loops that cross a given annulus.

\begin{lemma}[{\cite[Lemma~5.9]{amy-cle-resampling}}]
\label{lem:break_loops}
Suppose that $\Gamma$ is a nested \clekp{} in a simply connected domain $D$. For each $a\in(0,1)$, $b>0$ there exist $p > 0$ and $c>0$ such that the following holds.

Let $z\in D$ and $j_0\in\Z$ be such that $B(z,2^{-j_0})\subseteq D$. There exists a resampling $\Gamma^\resampled_{z,j}$ of $\Gamma$ within $A(z,2^{-j},2^{-j+1})$ for each $j>j_0$ with the following property. Let $F_{z,j}$ be the event that
\begin{itemize}
 \item no loop in $\Gamma^\resampled_{z,j}$ crosses the annulus $A(z,2^{-j},2^{-j+1})$,
 \item denoting $\CC$ the loops in $\Gamma$ that cross $A(z,2^{-j},2^{-j+1})$, the collection of loops of $\Gamma$ and $\Gamma^\resampled_{z,j}$ remain the same in each connected component of $\C \setminus \bigcup\CC$ that is not surrounded by a loop in $\CC$ and intersects $B(z,2^{-j})$, and the gasket of $\Gamma$ in these components is contained in the gasket of $\Gamma^\resampled_{z,j}$.
\end{itemize}
For $k\in\N$, let $\wt G_{z,j_0,k}$ be the event that the number of $j=j_0+1,\ldots,j_0+k$ so that
\[ \p[ F_{z,j} \mid \Gamma ] \geq p\]
is at least $(1-a)k$.  Then
\[ \p[(\wt G_{z,j_0,k})^c] \le c e^{-b k} . \]
\end{lemma}

\subsection{CLE gasket measure}
\label{se:cle_measure}

Let $D \subseteq \C$ be a simply connected domain, let $\Gamma$ be a non-nested \clekp{} in $D$, and let $\Upsilon_{\Gamma}$ be its gasket. Let $\dcle$ be as in~\eqref{eqn:cle_dim}. It is shown in \cite{ms2022clemeasure} that there exists a unique (up to a deterministic constant) measure $\meas{\cdot}{\Upsilon_\Gamma}$ on $\Upsilon_{\Gamma}$ that has finite expectation, is conformally covariant with the exponent~$\dcle$, and respects the Markov property of the \clekp{}. It is shown later in \cite{my2026cleminkowski} that the measure $\meas{\cdot}{\Upsilon_\Gamma}$ has an explicit description as the Minkowski content measure of~$\Upsilon_\Gamma$. That is, for a deterministic constant $c>0$, almost surely,
\[
\meas{A}{\Upsilon_\Gamma} = c\lim_{\delta\searrow 0} \delta^{\dcle-2} \,\operatorname{Leb}{\left( \bigcup_{x\in A \cap \Upsilon_\Gamma} B_E(x,\delta) \right)} 
\]
for each $A \Subset D$ with, say, rectifiable boundary.

It is further proved that the measure has local dimension~$\dcle$ at each point with superpolynomial concentration. In particular, we recall the following upper and lower bounds which will be used later in this work. Suppose that $\Gamma$ is a nested \clekp{} in $D$. Let $(\Upsilon_k)$ be an enumeration of its gaskets, i.e.\ each $\Upsilon_k$ is the gasket associated with the non-nested \clekp{} inside some loop of $\Gamma$. Let $\meas{\cdot}{\Upsilon_k}$ denote the gasket measure on $\Upsilon_k$.

\begin{proposition}[{\cite[Proposition~1.4]{my2026cleminkowski}}]\label{pr:measure_ub}
Let $\Gamma$ be a nested \clekp{} in a simply connected domain $D \subseteq \C$. Fix any $\ell \in \N$, $a>0$, and a compact set $K \subseteq D$. For each $r>0$, let $E_r$ be the event that
\[ \sup_{z \in K} \meas{B_E(z,r)}{\Upsilon_k} \le r^{\dcle-a} \]
for each gasket $\Upsilon_k$ of $\Gamma$ of level at most $\ell$. Then
\[ \p[(E_r)^c] = o^\infty(r) \quad\text{as } r \searrow 0 . \]
\end{proposition}

\begin{proposition}[{\cite[Proposition~1.5]{my2026cleminkowski}}]\label{pr:measure_lb}
Let $\Gamma$ be a nested \clekp{} in a simply connected domain $D \subseteq \C$. Fix any $a>0$ and a compact set $K \subseteq D$. For each $r>0$, let $E_r$ be the event that the following holds. Let $\Upsilon_k$ be any of the gaskets of~$\Gamma$, let $x \in \Upsilon_k \cap K$ and $w \in \partial B_E(x,r) \cap \partial\Bpath(x,r)$. Let $B_{x,w} \subseteq \ol{B_E}(x,r)$ be a simply connected set containing all simple admissible paths within $\ol{B_E}(x,r)$ from $x$ to $w$. Then
\[ \meas{B_{x,w}}{\Upsilon_k} \ge r^{\dcle+a} \quad\text{for each such $x,w$, and $k$.} \]
Then
\[ \p[(E_r)^c] = o^\infty(r) \quad\text{as } r \searrow 0 . \]
\end{proposition}

\subsection{Geodesic \clekp{} metrics}
\label{se:geodesic_metric}

Let $D \subseteq \C$ be a simply connected domain, let $\Gamma$ be a nested \clekp{} in $D$. Recall the notation $\paths{x}{y}{U}{\Gamma}$ and $\dpathY$ defined in Section~\ref{subsec:main_results}. If $U \subseteq \C$ is an open set, we let $\Gamma|_U$ denote the (countable) collection of loops and maximal strands of $\Gamma$ contained in $U$.

It is shown in \cite{amy2025tightness,my2025geouniqueness} that one can define uniquely (up to a deterministic constant) a geodesic metric $\met{\cdot}{\cdot}{\Gamma}$ on each of the gaskets\footnote{As in Theorem~\ref{th:all_clusters_domain}, the exterior gasket is split up by $\partial D$.} of $\Gamma$ satisfying the following properties. Let $\lmet{\gamma}$ denote the length of a path $\gamma$ with respect to $\met{\cdot}{\cdot}{\Gamma}$.
\begin{itemize}
 \item $\met{\cdot}{\cdot}{\Gamma}$ is continuous with respect to $\dpathY$.
 \item If we define for each open set $U \subseteq \C$ the internal metric
 \[ \metres{U}{x}{y}{\Gamma} = \inf_{\gamma \in \paths{x}{y}{U}{\Gamma}} \lmet{\gamma} , \]
 then $\metres{U}{\cdot}{\cdot}{\Gamma}$ is a measurable function of $\Gamma|_U$ and for each $z \in \C$ we have
 \[ \metres{U+z}{\cdot+z}{\cdot+z}{\Gamma|_{U+z}} = \metres{U}{\cdot}{\cdot}{\Gamma|_{U+z}-z} \quad\text{almost surely.} \]
\end{itemize}
Further, it is shown in \cite{my2025geouniqueness} that the metric is scale covariant\footnote{It is also conformally covariant, but we will not need this fact in the present paper.} in the following sense. There exists a constant $\geoexp > 0$ (depending only on $\kappa'$) such that for each $\lambda > 0$ and each open set $U \subseteq \C$ we have
\[ \metres{\lambda U}{\lambda\cdot}{\lambda\cdot}{\Gamma|_{\lambda U}} = \lambda^{\geoexp}\metres{U}{\cdot}{\cdot}{\lambda^{-1}\Gamma|_{\lambda U}} \quad\text{almost surely.} \]
For each gasket $\Upsilon$ of $\Gamma$ we have that $(\Upsilon, \met{\cdot}{\cdot}{\Gamma})$ is a compact geodesic metric space. Moreover, the metric $\met{\cdot}{\cdot}{\Gamma}$ is H\"older continuous with respect to $\dpathY$ with superpolynomial tails. More precisely, we have the following concentration estimates.

\begin{proposition}[{\cite[Proposition~1.9]{my2025geouniqueness}}]
\label{pr:cle_metric_estimates}
Fix $a>0$ and a compact set $K \subseteq D$. Then
\[
 \p\left[ \sup_{\Upsilon} \sup_{\substack{x,y \in \Upsilon \cap K\\ \dpathY(x,y) < \delta}} \metres{B(x,3\delta)}{x}{y}{\Gamma} \ge \delta^{\geoexp-a} \right] = o^\infty(\delta) \quad\text{as } \delta \searrow 0 ,
\]
and
\[
 \p\left[ \inf_{\Upsilon} \inf_{\substack{x,y \in \Upsilon \cap K\\ \dpathY(x,y) \ge \delta}} \met{x}{y}{\Gamma} \le \delta^{\geoexp+a} \right] = o^\infty(\delta) \quad\text{as } \delta \searrow 0 .
\]
\end{proposition}

\subsection{Resistance metrics, resistance forms, and Dirichlet forms}
\label{subsec:dirichlet_resistance}

\subsubsection{Electrical networks and effective resistance}

All graphs in this section are assumed to be finite and simple.

Let $\vertexset$ be a finite set, and $w\colon \vertexset\times\vertexset \to [0,\infty)$ a function with $w(x,y) = w(y,x)$ and $w(x,x) = 0$ for each $x,y \in \vertexset$. We consider $(\vertexset,w)$ as a weighted graph (also called an \emph{electrical network}) with edge set given by $\edgeset = \{ \{x,y\} \mid w(x,y) > 0 \}$, and refer to $w$ as \emph{edge conductances}. (Sometimes it is useful to parameterize the network by the \emph{edge resistances} $r(x,y) = 1/w(x,y)$ instead.) We say that $(\vertexset,w)$ is connected if the induced graph $(\vertexset,\edgeset)$ is connected.

There is a natural Markov chain on $\vertexset$ associated with $w$. At each step, the chain jumps from $x$ to $y$ with probability $w(x,y)/\lambda(x)$ where $\lambda(x) = \sum_z w(x,z)$. This Markov chain is irreducible if and only if $(\vertexset,w)$ is connected. In that case, we associate with it a (discrete) Dirichlet form and an effective resistance metric as follows. For a function $f\colon \vertexset \to \R$, let
\begin{equation}\label{eq:df_discrete}
 \CE(f,f) = \frac{1}{2} \sum_{x,y \in V} w(x,y)(f(x)-f(y))^2 .
\end{equation}
The effective resistance $R\colon \vertexset\times\vertexset \to [0,\infty)$ is defined as
\[
R(x,y) = 
\begin{cases}
 \bigl( \min\{ \CE(f,f) \mid f\colon \vertexset \to \R,\, f(x)=1,\, f(y)=0 \} \bigr)^{-1} ,& x \neq y ,\\
 0 ,& x = y .
\end{cases}
\]
If $(\vertexset,w)$ is connected, then $R$ is finite and defines a metric on $\vertexset$ \cite[Theorem~2.1.14]{k2001analysis}. Conversely, the edge conductances $w$ are uniquely determined by $R$ \cite[Theorem~2.1.12]{k2001analysis}.

Clearly, the map that associates to $w$ the effective resistance $R$ is continuous. The following lemma shows that the converse is also true.

\begin{lemma}\label{le:reff_lim}
 Let $\vertexset$ be a finite set and $w_n\colon \vertexset\times\vertexset \to [0,\infty)$ be a sequence of edge conductances such that each $(\vertexset,w_n)$ is connected. Let $R_n\colon \vertexset\times\vertexset \to [0,\infty)$ be the associated effective resistance metrics. Suppose that $R_n \to R$ for some $R\colon \vertexset\times\vertexset \to [0,\infty)$ with $R(x,y) > 0$ for each $x \neq y$. Then there exists $w\colon \vertexset\times\vertexset \to [0,\infty)$ such that $w_n \to w$ and $R$ is the effective resistance associated with $w$.
\end{lemma}

\begin{proof}
We view $w_n$ as functions taking values in the space $[0,\infty]$ equipped with the topology of a compact interval. Then, by compactness, every subsequence $(w_{n_k})$ of $(w_n)$ contains a convergent subsequence $(w_{n_{k_l}})$ where the limit is a function $w\colon \vertexset\times\vertexset \to [0,\infty]$. But in fact, $w$ must take values in $[0,\infty)$ since if $w_n(x,y) \to \infty$ for some $x,y \in \vertexset$, then we would have $R_n(x,y) \to 0$ which is excluded in our assumption.

Now, since the map that associates to the edge conductance function the effective resistance metric is continuous, we have that $(R_{n_{k_l}})$ converges to the effective resistance associated with $w$. But by the assumption, the latter agrees with $R$. Recalling that $w$ is uniquely determined by $R$, we have shown that every subsequence of $(w_n)$ contains a subsequence converging to $w$. This implies that $w_n \to w$.
\end{proof}

\begin{lemma}\label{le:reff_cut_point}
 Let $\vertexset_1,\vertexset_2$ be two finite sets with $\vertexset_1 \cap \vertexset_2 = \{z\}$. For $i=1,2$, let $w_i\colon \vertexset_i \times \vertexset_i \to [0,\infty)$ be edge conductances, each inducing a connected subgraph, and let $R_i\colon \vertexset_i \times \vertexset_i \to [0,\infty)$ be the associated effective resistance. Define $R\colon (\vertexset_1 \cup \vertexset_2) \times (\vertexset_1 \cup \vertexset_2) \to [0,\infty)$ by
 \[
 R(x,y) = \begin{cases}
 R_1(x,y) ,& x,y \in \vertexset_1,\\
 R_2(x,y) ,& x,y \in \vertexset_2,\\
 R_1(x,z)+R_2(z,y) ,& x \in \vertexset_1,\, y \in \vertexset_2 . 
 \end{cases}
 \]
 Then $R$ is the effective resistance associated with $w$ given by
 \begin{equation}\label{eq:wgraph_cut_point}
  w(x,y) =
  \begin{cases}
   w_1(x,y) ,& x,y \in \vertexset_1,\\
   w_2(x,y) ,& x,y \in \vertexset_2,\\
   0 & \text{otherwise.}
  \end{cases}
 \end{equation}
\end{lemma}

\begin{proof}
 Define $w$ by~\eqref{eq:wgraph_cut_point}. Then $z$ is a cut point of the associated weighted graph, and therefore the associated effective resistance is exactly $R$.
\end{proof}

\begin{lemma}
\label{le:gen_parallel_law_graph}
Let $(\vertexset,w)$ be a weighted graph, and let $\graph = (\vertexset,\edgeset)$ be the subgraph induced by the edges with positive conductances. Let $x,y,z_1,\ldots,z_N \in \vertexset$ such that $\{z_1,\ldots,z_N\}$ separates $x$ from $y$ in $\graph$. Let $K_x$ be the subgraph induced by removing the vertices that are strictly separated from $x$ by $\{z_1,\ldots,z_N\}$. Let $R_\graph$ (resp.\ $R_{K_x}$) be the effective resistance associated with $(\graph,w)$ (resp.\ $(K_x, w|_{K_x \times K_x})$). Then
\[ R_\graph(x,y)^{-1} \le \sum_i R_{K_x}(x,z_i)^{-1} . \]
\end{lemma}

\begin{proof}
This is a consequence of \cite[Proposition~2.4]{dw2024tightness}. To apply the proposition, let $\CP$ be the collection of paths in $\graph$ from $x$ to $y$, and for each $i$ let $\CA_i$ be the collection of paths from $x$ to $y$ that visit $z_i$ before $\{z_j : j \neq i\}$. Then $\CP = \bigcup_i \CA_i$ and by \cite[Proposition~2.4]{dw2024tightness} we have $R_\graph(x,y)^{-1} = R_\graph(\CP)^{-1} \le \sum_i R_\graph(\CA_i)^{-1}$. Further, if we let $\CA'_i$ be the collection of paths from $x$ to $z_i$ in $\graph\setminus\{z_j : j \neq i\}$, then by monotonicity we have $R_{K_x}(x,z_i) \le R_{\graph\setminus\{z_j : j \neq i\}}(x,z_i) = R_\graph(\CA'_i) \le R_\graph(\CA_i)$.
\end{proof}

\begin{lemma}\label{le:reff_contraction}
 Let $(\vertexset,w)$ be a weighted graph, let $x_0,y_0 \in \vertexset$ two distinct vertices, and let
 \[
  w'(x,y) =
  \begin{cases}
   w(x,y) ,& (x,y) \neq (x_0,y_0) ,\\
   \infty ,& (x,y) = (x_0,y_0) .
  \end{cases}
 \]
 (In other words, we consider the weighted graph $(\vertexset',w')$ obtained from $(\vertexset,w)$ by contracting $(x_0,y_0)$ to a single vertex.)
 Let $R$ (resp.\ $R'$) the effective resistance associated with $w$ (resp.\ $w'$).
 Then
 \[ R(x,y) \le R'(x,y)+\frac{1}{w(x_0,y_0)} \quad\text{for each } x,y \in \vertexset . \]
\end{lemma}

\begin{proof}
 This can be seen from the representation of the effective resistance given in~\cite[Proposition~2.1]{dw2024tightness}. Indeed, suppose that $P_1,\ldots,P_n$ are paths in $\vertexset'$ from $x$ to $y$, and $\alpha_1,\ldots,\alpha_n > 0$, $\sum \alpha_k = 1$. We can extend each path (if necessary) to a path in $\vertexset$ by adding the edge $(x_0,y_0)$ and set $r_{(x_0,y_0),P_k} = \frac{1}{\alpha_k w(x_0,y_0)}$. Then the constraint $\sum_k \frac{1}{r_{(x_0,y_0),P_k}} = \sum_k \alpha_k w(x_0,y_0) \le w(x_0,y_0)$ is satisfied. By~\cite[Proposition~2.1]{dw2024tightness} we have
 \[
  R(x,y) \le \sum_k \alpha_k^2 \sum_{e \in P_k} r_{e,P_k}
  \le \sum_k \alpha_k^2 \sum_{e \in P_k,\, e \neq (x_0,y_0)} r_{e,P_k} + \frac{1}{w(x_0,y_0)} .
 \]
 Taking the infimum over all choices of $(P_k),(\alpha_k),(r_{e,P_k})$ gives the result.
\end{proof}

\subsubsection{Resistance metrics and resistance forms}
\label{se:rmet_rform}

We will now review some of the basics of resistance forms and Dirichlet forms.

The following general definitions of a resistance metric and of a resistance form was introduced by Kigami \cite[Section~2.3]{k2001analysis}.

\begin{definition}
\label{def:resistance_metric}
Let $F$ be a non-empty set.  A metric $R$ on $F$ is called a \emph{resistance metric} if the following is true.  For every $B \subseteq F$ finite there exists a choice of weights $w\colon B \times B \to [0,\infty)$ so that $R|_{B \times B}$ is equal to the effective resistance metric associated with $(B,w)$.
\end{definition}

\begin{definition}
\label{def:rform_definition}
Let $F$ be a non-empty set.  A pair $(\CE,\CF)$ is called a \emph{resistance form} if the following are true:
\begin{enumerate}[(RF1)]
\item\label{it:rf_domain} $\CF$ is a linear subspace of the space of $\R$-valued functions on $F$ that contains the constant functions, and $\CE$ is a non-negative symmetric quadratic form on $\CF$.  Moreover, $\CE(f,f) = 0$ if and only if $f \in \CF$ is a constant function.
\item\label{it:rf_closed} Let $\sim$ be the equivalence relation on $\CF$ given by $f \sim g$ if and only if $f-g$ is constant.  Then $(\CF/{\sim}, \CE)$ is a Hilbert space.
\item\label{it:rf_separating} For each $x \neq y$ there exists $f \in \CF$ such that $f(x) \neq f(y)$.
\item\label{it:rf_rmet} For each $x,y \in F$, we have
\begin{equation}\label{eq:rform_to_met}
 R(x,y) \defeq \sup \left\{ \frac{| f(x) - f(y)|^2}{\CE(f,f)} : f \in \CF,\, \CE(f,f) > 0 \right\} < \infty .
\end{equation}
\item\label{it:rf_markov} We have $\ol{f} \in \CF$ and $\CE(\ol{f},\ol{f}) \leq \CE(f,f)$ for each $f \in \CF$ where $\ol{f} = (f \vee 0) \wedge 1$.
\end{enumerate}
\end{definition}

Suppose that we are given a resistance form $(\CE,\CF)$ on a set $F$. Then the function $R$ given by~\eqref{eq:rform_to_met} is a resistance metric on $F$. Conversely, for each \emph{separable} resistance metric space $(F,R)$ there is a unique resistance form $(\CE,\CF)$ such that~\eqref{eq:rform_to_met} holds.

For each $f \in \CF$ and $a,b \in \R$ we have $\CE(af+b,af+b) = a^2 \CE(f,f)$, and therefore~\eqref{eq:rform_to_met} can be equivalently written as
\begin{equation}\label{eq:rform_to_met10}
R(x,y) =
\begin{cases}
 \bigl( \inf\{ \CE(f,f) \mid f\in\CF,\, f(x)=1,\, f(y)=0 \} \bigr)^{-1} ,& x \neq y ,\\
 0 ,& x = y ,
\end{cases}
\end{equation}
In fact, the infimum in~\eqref{eq:rform_to_met10} is uniquely attained as it can be realized as an orthogonal projection. The definition implies also
\begin{equation}\label{eq:fin_energy_continuous}
 \abs{f(x)-f(y)}^2 \le \CE(f,f)R(x,y) .
\end{equation}
In particular, every $f \in \CF$ is continuous with respect to $R$.

In the following, let $(F,R)$ be a separable resistance metric space and $(\CE,\CF)$ its associated resistance form. Then the following properties hold (see~\cite[Section~2.3]{k2001analysis}).

If $B$ is a finite subset of $F$ and $w\colon B \times B \to [0,\infty)$ is the weight function associated with $R|_{B \times B}$, then the energy of a function $g\colon B \to \R$ with respect to $(B,w)$ as defined in~\eqref{eq:df_discrete} is given by
\begin{equation*}
\CE|_B(g,g) = \min\{ \CE(f,f) \mid f\in\CF,\, f(x) = g(x) \text{ for each } x \in B \} .
\end{equation*}
Here, the minimum is uniquely attained as it can be realized as an orthogonal projection. If $(A_m)$ is an increasing sequence of finite subsets of $F$ such that $\bigcup_m A_m$ is dense in $F$, then
\begin{align}\begin{split}\label{eq:rmet_to_form_sep}
 \CF &= \{ f \in C(F,R) : \lim_m \CE|_{A_m}(f|_{A_m},f|_{A_m}) < \infty \} ,\\
 \CE(f,f) &= \lim_m \CE|_{A_m}(f|_{A_m},f|_{A_m}) .
\end{split}\end{align}

For each non-empty $B \subseteq F$, one can define more generally the \emph{trace} $(\CE|_B,\CF|_B)$ of $(\CE,\CF)$ on $B$ by
\begin{align}\begin{split}\label{eq:trace_form}
 \CF|_B &= \{ f|_B : f \in \CF \} ,\\
 \CE|_B(g,g) &= \min\{ \CE(f,f) \mid f\in\CF,\, f(x) = g(x) \text{ for each } x \in B \} .
\end{split}\end{align}
In the case when $B$ is finite, it follows from~(RF\ref{it:rf_domain}),(RF\ref{it:rf_separating}),(RF\ref{it:rf_markov}) that $\CF|_B$ contains all functions $g\colon B \to \R$.
A function $f \in \CF$ minimizing~\eqref{eq:trace_form} for some $g \in \CF|_B$ is called \emph{$\CE$-harmonic} on $F \setminus B$.

We will be interested in the case when $(\CE,\CF)$ defines a regular Dirichlet form and therefore a symmetric Hunt process \cite[Theorem~1.5.1]{cf2012dform}.

In the remainder of this section, we suppose that $(F,R)$ is a \emph{compact} resistance metric space with associated resistance form $(\CE,\CF)$. Let $\mu$ be a finite Borel measure on $(F,R)$ with $\supp_R(\mu) = F$. By \cite[Corollary~6.4 and Theorem~9.4]{k2012resistance}, $(\CE,\CF)$ is then a regular Dirichlet form on $L^2(F,\mu)$ (see~\cite[Chapter~1]{cf2012dform} for the definition of a regular Dirichlet form). By the general theory of Dirichlet forms (see e.g.\ \cite[Theorem~1.5.1]{cf2012dform}), there is a $\mu$-symmetric Hunt process $X$ on $F$ associated with $(\CE,\CF)$. This means that the transition functions $(P_t)_{t \ge 0}$ of the process yield symmetric operators on $L^2(F,\mu)$ and
\begin{align*}
 \CF &= \{ f \in L^2(F,\mu) \mid \lim_{t \searrow 0} \frac{1}{t}(f-P_t f, f)_{L^2(\mu)} < \infty \} ,\\
 \CE(f,g) &= \lim_{t \searrow 0} \frac{1}{t}(f-P_t f, g)_{L^2(\mu)} .
\end{align*}
Moreover, in this setup the process has infinite lifetime, and it holds that $\mathrm{P}_x[\sigma_y < \infty] = 1$ for every $x,y \in F$ where $\sigma_y = \inf\{ t > 0 \mid X_t = y \}$ (see~\cite[Proof of Lemma~3.14]{chk-time-change-resistance}). We remark also that in this setup every non-empty subset of $F$ has positive $\CE_1$-capacity, and every $\CE_1$-quasi-continuous function is continuous (see~\cite[Section~9]{k2012resistance}). In particular, by~\cite[Theorems~1.3.14,~3.1.10, and~Lemma~A.2.18]{cf2012dform}, every point is \emph{regular}, i.e.\ $\mathrm{P}_x[\sigma_x = 0] = 1$.

If $B \subseteq F$ is a compact subset and $\nu$ is a positive Radon measure on $(F,R)$ with $\supp_R(\nu) = B$, then one can define the trace $\check{X}$ of the process $X$ with respect to $\nu$ which is given by a time change of $X$ (restricted to the times where it is on $B$; see~\cite[Section~5.2]{cf2012dform} and~\cite[Section~2.1]{chk-time-change-resistance}). By~\cite[Theorem~5.2.2]{cf2012dform} and~\cite[Lemma~2.6]{chk-time-change-resistance}, the Dirichlet form of $\check{X}$ is given by the trace $(\CE|_B,\CF|_B)$ of the resistance form $(\CE,\CF)$ defined in~\eqref{eq:trace_form}, and its associated resistance metric is $R|_{B \times B}$. By~\cite[Proposition~3.4.1]{cf2012dform} and since every point is regular, the unique minimizer in~\eqref{eq:trace_form} is given by
\begin{align}\label{eq:harm_ext}
 H_B g(x) = \mathrm{E}_x[g(X_{\sigma_B})]
 \quad\text{where } \sigma_B = \inf\{t > 0 \mid X_t \in B\} .
\end{align}
In particular, for finite $B$, the electrical network associated with $R|_{B \times B}$ corresponds to the Markov chain obtained from restricting $X$ to the times where it visits $B$.

The resistance form $(\CE,\CF)$ is called \emph{local} if $\CE(f,g) = 0$ for each $f,g \in \CF$ with $\supp_R(f) \cap \supp_R(g) = \emptyset$. Since $\CE(f,g+c) = \CE(f,g)$ for each $c \in \R$, this implies that $(\CE,\CF)$ is strongly local in the Dirichlet form sense \cite[Section~2.4]{cf2012dform}. By~\cite[Theorem~4.3.4]{cf2012dform}, the associated Hunt process has continuous sample paths (up until its extinction time which is infinite in our setup). Such a process is called a \emph{diffusion}.

We conclude this subsection with a few lemmas which will be used later in the paper.

\begin{lemma}\label{le:rforms_cutpoint}
 Let $F_1,F_2$ be two sets with $F_1 \cap F_2 = \{z\}$. For $i=1,2$, let $(\CE_i,\CF_i)$ be a resistance form on $F_i$ with associated resistance metric $R_i$. Define $R\colon (F_1 \cup F_2) \times (F_1 \cup F_2) \to [0,\infty)$ by
 \[
 R(x,y) = \begin{cases}
 R_1(x,y) ,& x,y \in F_1,\\
 R_2(x,y) ,& x,y \in F_2,\\
 R_1(x,z)+R_2(z,y) ,& x \in F_1,\, y \in F_2 .
 \end{cases}
 \]
 Then $R$ is a resistance metric on $F_1 \cup F_2$ and the associated resistance form $(\CE,\CF)$ is given by
 \begin{align}\begin{split}\label{eq:rform_cutpoint}
  \CF &= \{ f : f|_{F_1} \in \CF_1,\, f|_{F_2} \in \CF_2 \} ,\\
  \CE(f,f) &= \CE_1(f|_{F_1},f|_{F_1})+\CE_2(f|_{F_2},f|_{F_2}) .
 \end{split}\end{align}
\end{lemma}

\begin{proof}
Define $(\CE,\CF)$ by~\eqref{eq:rform_cutpoint}. Then one verifies that this defines a resistance form and its associated resistance metric is given by $R$. Indeed, for $i=1,2$ and $x,y \in F_i$, the minimizer in~\eqref{eq:rform_to_met10} is given by the function that minimizes~\eqref{eq:rform_to_met10} on~$F_i$ and is constant on $F_{3-i}$; and for $x \in F_1$, $y \in F_2$, the minimizer is given by the function with value $\frac{R_2(z,y)}{R_1(x,z)+R_2(z,y)}$ on $z$.
\end{proof}

Note that in the setup of Lemma~\ref{le:rforms_cutpoint}, if $B \subseteq F_1 \cup F_2$ is a finite set with $z \in B$, then the edge conductances associated with $R|_B$, $R_1|_{B \cap F_1}$, $R_2|_{B \cap F_2}$, respectively, are related by Lemma~\ref{le:reff_cut_point}.

\begin{lemma}\label{le:weights_gluing}
 Let $F_1,F_2$ be two sets such that $F_1 \cap F_2$ is a non-empty finite set. For $i=1,2$, let $(\CE_i,\CF_i)$ be a resistance form on $F_i$, and let $(\CE,\CF)$ be the resistance form on $F_1 \cup F_2$ defined by
 \begin{align*}\begin{split}
  \CF &= \{ f : f|_{F_1} \in \CF_1,\, f|_{F_2} \in \CF_2 \} ,\\
  \CE(f,f) &= \CE_1(f|_{F_1},f|_{F_1})+\CE_2(f|_{F_2},f|_{F_2}) .
 \end{split}\end{align*}
 Then the following holds. Let $A \subseteq F_1 \cup F_2$ be a finite set containing $F_1 \cap F_2$, let $A_i = A \cap F_i$ for $i=1,2$, and let $w_A, w_{A,i}$ be the weight functions associated with $\CE|_A, \CE_i|_{A_i}$, respectively.
 Then
 \[
  w_A(x,y) =
  \begin{cases}
   w_{A,1}(x,y)+w_{A,2}(x,y) ,& x,y \in F_1 \cap F_2 ,\\
   w_{A,1}(x,y) ,& x,y \in F_1 ,\, \{x,y\} \nsubseteq F_1 \cap F_2 ,\\
   w_{A,2}(x,y) ,& x,y \in F_2 ,\, \{x,y\} \nsubseteq F_1 \cap F_2 ,\\
   0 ,& \text{otherwise.}
  \end{cases}
 \]
\end{lemma}

\begin{proof}
Since $F_1 \cap F_2$ is a non-empty finite set, $(\CE,\CF)$ is a resistance form. Suppose $A \subseteq F_1 \cup F_2$ is a finite set containing $F_1 \cap F_2$. The weights $w_A(x,y)$ are given by $-\CE(f_x,f_y)$ where $f_x$ (resp.\ $f_y$) is the harmonic function with boundary values $f_x(u) = \delta_{ux}$ (resp.\ $f_y(u) = \delta_{uy}$) on $A$. For $i=1,2$ and $x \in A_i$ let $f_{i,x}$ be the $\CE_i$-harmonic function with boundary values $f_{i,x}(u) = \delta_{ux}$ on $A_i$. Since $A$ contains $F_1 \cap F_2$, we have for $x \in A_i \setminus A_{3-i}$ that
\[
 f_x(u) =
 \begin{cases}
  f_{i,x}(u) ,& u \in F_i ,\\
  0 ,& u \in F_{3-i} ,
 \end{cases}
\]
and for $x \in A_1 \cap A_2$ that
\[
 f_x(u) =
 \begin{cases}
  f_{1,x}(u) ,& u \in F_1 ,\\
  f_{2,x}(u) ,& u \in F_2 .
 \end{cases}
\]
We then conclude by the polarization
\[\begin{split}
\CE(f_x,f_y) &= \frac{1}{4}(\CE(f_x+f_y,f_x+f_y)-\CE(f_x-f_y,f_x-f_y)) \\
&= \CE_1(f_x|_{F_1},f_y|_{F_1})+\CE_2(f_x|_{F_2},f_y|_{F_2}) .
\end{split}\]
\end{proof}

\section{Weak \clekp{} resistance forms}
\label{sec:weak_cle_rform}

In this section, we introduce the definition of weak \clekp{} resistance forms. We will then prove their existence by describing an approximation scheme for the \clekp{} gasket. Since resistance forms are in bijection with resistance metrics (recall Section~\ref{se:rmet_rform}), we will first construct \clekp{} resistance metrics using the main result of~\cite{amy2025tightness} which we review in Section~\ref{sec:tightness_general}. We then describe in Section~\ref{subsec:resistance_tightness} an approximation scheme for resistance metrics on the \clekp{} gasket, and show that the main result in~\cite{amy2025tightness} applies to the approximation scheme, i.e.\ it has subsequential limits. We then show in Section~\ref{se:rmet_construction} that the subsequential limits give rise to resistance metrics, and their associated resistance forms satisfy the definition of weak \clekp{} resistance forms.

\subsection{Definition of weak \clekp{} resistance forms}

Suppose that $\Gamma$ is a \clekp{} in $C$ as described in Section~\ref{subsec:main_results} and $(\Upsilon_\Gamma,\dpathY)$ is the space defined in Section~\ref{subsec:main_results}. Recall the definition of $\metregions$ and $\metregions[U]$ given in Section~\ref{subsec:main_results}. Suppose that for each $V \in \metregions$ we have a resistance form $(\rformres{V}{\cdot}{\cdot}{\Gamma}, \rfdomainres{V}{\Gamma})$ on $\ol{V} \cap \Upsilon_\Gamma$. In the following, if $f$ is a function defined on a set containing $\ol{V} \cap \Upsilon_\Gamma$, we will frequently write $f|_V$ for $f|_{\ol{V} \cap \Upsilon_\Gamma}$ and $\rformres{V}{f}{f}{\Gamma}$ for $\rformres{V}{f|_V}{f|_V}{\Gamma}$.

\begin{figure}[ht]
\centering
\includegraphics[width=0.5\textwidth]{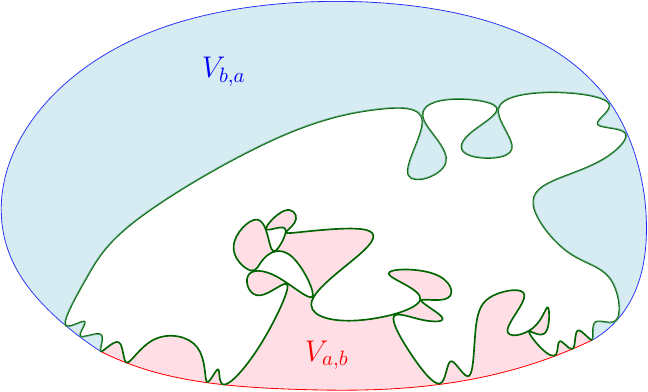}
\caption{Illustration of the locality property~\eqref{it:weak_rform_cut_loop} of weak \clekp{} resistance forms. (A simplified sketch of) a loop $\CL$ contained in $\ol{V}$ is shown in green.}
\label{fi:cut_loop}
\end{figure}

\begin{definition}
\label{def:weak_cle_rform}
 Suppose that we have the setup described above, and $\rform{\cdot}{\cdot}{\Gamma} = (\rformres{V}{\cdot}{\cdot}{\Gamma})_{V \in \metregions} = ((\rformres{V}{\cdot}{\cdot}{\Gamma}, \rfdomainres{V}{\Gamma}))_{V \in \metregions}$ is a random family of resistance forms on $\ol{V} \cap \Upsilon_\Gamma$ for $V \in \metregions$, respectively, coupled with~$\Gamma$ (it is not required that they are determined by $\Gamma$). We say that $\rform{\cdot}{\cdot}{\Gamma}$ is a weak \clekp{} resistance form if the following properties are satisfied almost surely.

\begin{enumerate}[(i)]
\item\label{it:weak_rform_rmet} The resistance metric associated with $(\rformres{C}{\cdot}{\cdot}{\Gamma}, \rfdomainres{C}{\Gamma})$ is continuous with respect to $\dpathY$.

\item\label{it:weak_rform_markov} Let $U \subseteq \C$ be open, simply connected.  The conditional law of the collection $( \rformres{V}{\cdot}{\cdot}{\Gamma})_{V \in \metregions[U]}$ given $\Gamma \setminus \Gamma_{U^*}$ and $(\rformres{V'}{\cdot}{\cdot}{\Gamma})_{V' \in \metregions[\C \setminus \ol{U}]}$ is almost surely measurable with respect to $U^*$.

\item\label{it:weak_rform_invariance} Let $U \subseteq \C$ be open, simply connected. There exists a probability kernel $\mu^{U^*}$ such that for each $z \in \C$, the conditional law of $(\rformres{V}{\cdot}{\cdot}{\Gamma})_{V \in \metregions[U+z]}$ given $(U+z)^*$ is $(T_z)_* \mu^{U^*}(\cdot - z)$ where $T_z \CE(f,f) = \CE(f(\cdot +z),f(\cdot +z))$ denotes translation by $z$.

\item\label{it:weak_rform_compatibility} Suppose that $V,V' \in \metregions$, $V \subseteq V'$, such that for every $u \in V' \setminus V$ there is a point $z \in \ol{V}$ that separates $u$ from $V$ in $\ol{V'} \cap \Upsilon_\Gamma$. Let $f\colon \ol{V'} \cap \Upsilon_\Gamma \to \R$ be a continuous function that is constant on each $\dpathY$-connected component of $\ol{V'} \cap \Upsilon_\Gamma \setminus \ol{V}$. Then
\begin{align*}
 & f \in \rfdomainres{V'}{\Gamma} \text{ if and only if } f|_V \in \rfdomainres{V}{\Gamma} \\
 \text{and}\quad & \rformres{V'}{f}{f}{\Gamma} = \rformres{V}{f}{f}{\Gamma} .
\end{align*}

\item\label{it:weak_rform_cut_point} Suppose that $V_1,V_2 \in \metregions$ with $\ol{V_1} \cap \ol{V_2} = \{z\}$ (where the closure is taken with respect to $d_{\ol{V}}$, see Section~\ref{subsec:main_results}). Then
 \begin{align*}
  \rfdomainres{V_1 \cup V_2}{\Gamma} &= \{ f : f|_{V_1} \in \rfdomainres{V_1}{\Gamma},\, f|_{V_2} \in \rfdomainres{V_2}{\Gamma} \} \\
  \text{and}\quad
  \rformres{V_1 \cup V_2}{f}{f}{\Gamma} &= \rformres{V_1}{f}{f}{\Gamma} + \rformres{V_2}{f}{f}{\Gamma} .
 \end{align*}

\item\label{it:weak_rform_cut_loop} Suppose that $V \in \metregions$ consists of a single connected component and let $\CL \in \Gamma$ be a loop contained in $\ol{V}$ and intersecting $\partial V$. Let $a,b \in \CL \cap \partial V$ be distinct, let $V_{a,b}$ (resp.\ $V_{b,a}$) be the collection of components of $V \setminus \CL$ that are connected to the counterclockwise arc of $\partial V$ from $a$ to $b$ (resp.\ from $b$ to $a$) by an admissible path within $V$ (see Figure~\ref{fi:cut_loop}). Then
 \begin{align*}
  \rfdomainres{V}{\Gamma} &= \{ f : f|_{V_{a,b}} \in \rfdomainres{V_{a,b}}{\Gamma},\, f|_{V_{b,a}} \in \rfdomainres{V_{b,a}}{\Gamma} \} \\
  \text{and}\quad
  \rformres{V}{f}{f}{\Gamma} &= \rformres{V_{a,b}}{f}{f}{\Gamma} + \rformres{V_{b,a}}{f}{f}{\Gamma} .
 \end{align*}
\end{enumerate}
\end{definition}

We note that the items~\eqref{it:weak_rform_markov} and~\eqref{it:weak_rform_invariance} in the definition of a weak \clekp{} resistance form are weaker than their counterparts~\ref{it:rform_determined},~\ref{it:rform_translation_invariant} in the Definition~\ref{def:cle_rform} of a (strong) \clekp{} resistance form. Here, we do not require that the resistance forms within $U$ are determined by the CLE gasket within $U$ but only that their conditional laws are determined. Since $\Gamma_{U^*}$ is conditionally independent of $\Gamma \setminus \Gamma_{U^*}$ given $U^*$, we have that every \clekp{} resistance form satisfies~\eqref{it:weak_rform_markov} and~\eqref{it:weak_rform_invariance}. The items~\eqref{it:weak_rform_cut_point}, \eqref{it:weak_rform_cut_loop} are equivalent to~\ref{it:rform_additive} in Definition~\ref{def:cle_rform} (see Lemma~\ref{le:rform_additivity}), but we find it useful to break it down into these two special cases. In particular, every \clekp{} resistance form in the sense of Definition~\ref{def:cle_rform} is a weak \clekp{} resistance form. We will prove in Sections~\ref{sec:bilipschitz} and~\ref{sec:uniqueness} that, in fact, every weak \clekp{} resistance form is a \clekp{} resistance form, and is determined by the \clekp{}.

Recall~\eqref{eq:fin_energy_continuous} that each function in the domain of a resistance form is continuous with respect to the resistance metric. We note that requiring the resistance metric $\rmetres{C}{\cdot}{\cdot}{\Gamma}$ associated with $(\rformres{C}{\cdot}{\cdot}{\Gamma}, \rfdomainres{C}{\Gamma})$ to be continuous in~\eqref{it:weak_rform_rmet} already implies that each $\rmetres{V}{\cdot}{\cdot}{\Gamma}$, $V \in \metregions$, is continuous with respect to $\dpathY$ (see Lemma~\ref{le:rmet_continuous}). In particular, each function $f \in \rfdomainres{V}{\Gamma}$ is continuous with respect to $\rmetres{V}{\cdot}{\cdot}{\Gamma}$ and hence also with respect to $\dpathY[\ol{V}]$.

We give an equivalent characterization of weak \clekp{} resistance forms in terms of their associated resistance metrics. This will be used later as we will construct a weak \clekp{} resistance form from its resistance metric. The Proposition~\ref{pr:cle_rmet_char} below also implies that $\rmetres{V'}{\cdot}{\cdot}{\Gamma}$ determines $\rmetres{V}{\cdot}{\cdot}{\Gamma}$ for each $V \subseteq V'$.

\begin{proposition}\label{pr:cle_rmet_char}
 Suppose that we have a family $(\rmetres{V}{\cdot}{\cdot}{\Gamma})_{V \in \metregions}$ of resistance metrics coupled with $\Gamma$ where $\rmetres{V}{\cdot}{\cdot}{\Gamma}\colon (\ol{V} \cap \Upsilon_\Gamma) \times (\ol{V} \cap \Upsilon_\Gamma) \to [0,\infty)$. Then it is associated with a weak \clekp{} resistance form if and only if the following hold.
 \begin{enumerate}[(a)]
  \item\label{it:rmet_cont} The metric $\rmetres{C}{\cdot}{\cdot}{\Gamma}$ is continuous with respect to $\dpathY$.

  \item\label{it:rmet_markov} Let $U \subseteq \C$ be open, simply connected. The conditional law of the collection $(\rmetres{V}{\cdot}{\cdot}{\Gamma})_{V \in \metregions[U]}$ given $\Gamma\setminus\Gamma_{U^*}$ and $(\rmetres{V'}{\cdot}{\cdot}{\Gamma})_{V' \in \metregions[\C\setminus\ol{U}]}$ is almost surely measurable with respect to~$U^*$.

  \item\label{it:rmet_translation} Let $U \subseteq \C$ be open, simply connected. There exists a probability kernel $\mu^{U^*}$ such that for each $z \in \C$, the conditional law of $(\rmetres{V}{\cdot}{\cdot}{\Gamma})_{V \in \metregions[U+z]}$ given $(U+z)^*$ is $(T_z)_* \mu^{U^*}(\cdot -z)$ where $T_z \FR(x,y) = \FR(x-z,y-z)$ denotes translation by $z$.

  \item\label{it:rmet_compatibility} Let $V,V' \in \metregions$, $V \subseteq V'$, and $x,y \in \ol{V} \cap \Upsilon_\Gamma$ such that for every $u \in V' \setminus V$ there is a point $z \in \ol{V}$ that separates $u$ from $x,y$ in $\ol{V'} \cap \Upsilon_\Gamma$. Then
  \[ \rmetres{V}{x}{y}{\Gamma} = \rmetres{V'}{x}{y}{\Gamma} . \]

  \item\label{it:rmet_cut_point} Suppose that $V_1,V_2 \in \metregions$ with $\ol{V_1} \cap \ol{V_2} = \{z\}$ (where the closure is taken with respect to $d_{\ol{V}}$, see Section~\ref{subsec:main_results}). If $x \in \ol{V_1} \cap \Upsilon_\Gamma$, $y \in \ol{V_2} \cap \Upsilon_\Gamma$, then
  \[ \rmetres{V_1 \cup V_2}{x}{y}{\Gamma} = \rmetres{V_1}{x}{z}{\Gamma} + \rmetres{V_2}{z}{y}{\Gamma} . \]

  \item\label{it:rmet_cut_loop} There exists a countable set $\{z_m\}_{m \in \N} \subseteq \Upsilon_\Gamma$ that is dense with respect to $\dpathY$ and such that the following holds. Suppose that $V \in \metregions$ consists of a single connected component and let $\CL \in \Gamma$ be a loop contained in $\ol{V}$ and intersecting $\partial V$. Then there exist two distinct points $a,b \in \CL \cap \partial V$ such that the following holds for sufficiently large $m \in \N$. Let $V_{a,b}$ (resp.\ $V_{b,a}$) be the collection of components of $V \setminus \CL$ that are connected to the counterclockwise arc of $\partial V$ from $a$ to $b$ (resp.\ from $b$ to $a$) by an admissible path within $V$. Let $A = \{z_1,\ldots,z_m\}$, $A_{a,b} = A \cap \ol{V_{a,b}}$, $A_{b,a} = A \cap \ol{V_{b,a}}$, and let $w, w_{a,b}, w_{b,a}$ be the weight functions associated with $\rmetres{V}{\cdot}{\cdot}{\Gamma}|_{A}$, $\rmetres{V_{a,b}}{\cdot}{\cdot}{\Gamma}|_{A_{a,b}}$, $\rmetres{V_{b,a}}{\cdot}{\cdot}{\Gamma}|_{A_{b,a}}$, respectively. Then
  \[
   w(x,y) =
   \begin{cases}
    w_{a,b}(a,b)+w_{b,a}(a,b) ,& \{x,y\} = \{a,b\} ,\\
    w_{a,b}(x,y) ,& x,y \in A_{a,b} ,\, \{x,y\} \neq \{a,b\} ,\\
    w_{b,a}(x,y) ,& x,y \in A_{b,a} ,\, \{x,y\} \neq \{a,b\} ,\\
    0 ,& \text{otherwise.}
   \end{cases}
  \]
 \end{enumerate}
\end{proposition}

\begin{proof}
The properties~\eqref{it:rmet_cont}--\eqref{it:rmet_translation} correspond exactly to the properties~\eqref{it:weak_rform_rmet}--\eqref{it:weak_rform_invariance} in Definition~\ref{def:weak_cle_rform}. We need to show that the properties~\eqref{it:rmet_compatibility}--\eqref{it:rmet_cut_loop} are equivalent to~\eqref{it:weak_rform_compatibility}--\eqref{it:weak_rform_cut_loop}.

Suppose that $(\rformres{V}{\cdot}{\cdot}{\Gamma})_{V \in \metregions}$ satisfies~\eqref{it:weak_rform_rmet}--\eqref{it:weak_rform_cut_loop}, and let $(\rmetres{V}{\cdot}{\cdot}{\Gamma})_{V \in \metregions}$ be the associated family of resistance metrics.

We show~\eqref{it:rmet_compatibility}. Let $f$ be the $\rformres{V}{\cdot}{\cdot}{\Gamma}$-harmonic function with $f(x) = 1$, $f(y) = 0$. As a consequence of~\eqref{it:weak_rform_cut_point}, $f$ must be constant on each component of $V$ that is separated from $x,y$ by a single point $z \in \ol{V}$. Therefore we can extend $f$ to $\ol{V'}$ by setting it constant on each component that is separated from $x,y$ by a single point $z \in \ol{V}$, and by~\eqref{it:weak_rform_compatibility} we have $\rformres{V'}{\wt{f}}{\wt{f}}{\Gamma} = \rformres{V}{f}{f}{\Gamma}$. By the same argument, if $f'$ is the $\rformres{V'}{\cdot}{\cdot}{\Gamma}$-harmonic function with $f'(x) = 1$, $f'(y) = 0$, then $\rformres{V}{f'}{f'}{\Gamma} = \rformres{V'}{f'}{f'}{\Gamma}$. This shows that $\rmetres{V'}{x}{y}{\Gamma} = \rmetres{V}{x}{y}{\Gamma}$.

Property~\eqref{it:rmet_cut_point} follows from~\eqref{it:weak_rform_cut_point} by Lemma~\ref{le:rforms_cutpoint}.

We now show~\eqref{it:rmet_cut_loop}. Consider the countable collection $\CK$ of loop segments $\ell \subseteq \CL$ starting and ending at rational times of all $\CL \in \Gamma$. For each non-overlapping, intersecting pair $\ell,\ell' \in \CK$ the intersection set $\ell \cap \ell'$ is compact. Let $\{z_m\}$ be any countable dense set containing a dense set of $\ell \cap \ell'$ for each such pair $\ell,\ell' \in \CK$. Then $\{z_m\}$ is also dense in $\CL \cap \partial V$ for each choice of $V$ and $\CL$ in~\eqref{it:rmet_cut_loop}. Choose any distinct pair of $a,b \in \CL \cap \partial V \cap \{z_m\}$. Then~\eqref{it:rmet_cut_loop} follows by~\eqref{it:weak_rform_cut_loop} and Lemma~\ref{le:weights_gluing}.

Conversely, suppose now that $(\rmetres{V}{\cdot}{\cdot}{\Gamma})_{V \in \metregions}$ satisfies~\eqref{it:rmet_cont}--\eqref{it:rmet_cut_loop}.

Property~\eqref{it:weak_rform_cut_point} follows from~\eqref{it:rmet_compatibility} and~\eqref{it:rmet_cut_point} by Lemma~\ref{le:rforms_cutpoint}.

We show~\eqref{it:weak_rform_compatibility}. Property~\eqref{it:rmet_compatibility} implies that $\rformres{V}{\cdot}{\cdot}{\Gamma}$ is the trace of $\rformres{V'}{\cdot}{\cdot}{\Gamma}$ on $\ol{V} \cap \Upsilon_\Gamma$. By~\eqref{it:weak_rform_cut_point}, the minimizer in~\eqref{eq:trace_form} must be constant on each component of $\ol{V'} \cap \Upsilon_\Gamma \setminus \ol{V}$. This shows~\eqref{it:weak_rform_compatibility}.

We now show~\eqref{it:weak_rform_cut_loop}. Let $\{z_m\}$ be given as in~\eqref{it:rmet_cut_loop}. Note that given $V$ and $\CL$, it suffices to show the statement for one particular choice of $a,b$ since we can shift them using~\eqref{it:weak_rform_cut_point}. Indeed, if $a' \in \CL \cap \partial V$ lies on the counterclockwise arc of $\partial V$ from $a$ to $b$, then
 \[\begin{split}
  \rformres{V_{a,b}}{f}{f}{\Gamma} + \rformres{V_{b,a}}{f}{f}{\Gamma}
  &= \rformres{V_{a,a'}}{f}{f}{\Gamma} + \rformres{V_{a',b}}{f}{f}{\Gamma} + \rformres{V_{b,a}}{f}{f}{\Gamma} \\
  &= \rformres{V_{a',b}}{f}{f}{\Gamma} + \rformres{V_{b,a'}}{f}{f}{\Gamma} .
 \end{split}\]
Therefore we can assume that $a,b$ are as given in~\eqref{it:rmet_cut_loop}.

Let $m \in \N$ and $A = \{z_1,\ldots,z_m\}$, $A_{a,b} = A \cap \ol{V_{a,b}}$, $A_{b,a} = A \cap \ol{V_{b,a}}$. By~\eqref{it:rmet_cut_loop} we have for sufficiently large $m$ that
 \[
  \rformrestr{V}{A}{f}{f}{\Gamma} = \frac{1}{2} \sum_{x,y \in A} w(x,y)(f(x)-f(y))^2 = \rformrestr{V_{a,b}}{A_{a,b}}{f}{f}{\Gamma}+\rformrestr{V_{b,a}}{A_{b,a}}{f}{f}{\Gamma} .
 \]
Sending $m \to \infty$, we obtain by~\eqref{eq:rmet_to_form_sep} that
 \[
  \rformres{V}{f}{f}{\Gamma} = \rformres{V_{a,b}}{f}{f}{\Gamma} + \rformres{V_{b,a}}{f}{f}{\Gamma} .
 \]
\end{proof}

In the remainder of this subsection, we show a few basic properties of weak \clekp{} resistance forms and metrics.

\begin{lemma}\label{le:rmet_continuous}
Almost surely, for each $V \in \metregions$, the metric $\rmetres{V}{\cdot}{\cdot}{\Gamma}$ is topologically equivalent to~$\dpathY$ on $\ol{V} \cap \Upsilon_\Gamma$.
\end{lemma}

\begin{proof}
We will explain that the continuity of $\rmetres{C}{\cdot}{\cdot}{\Gamma}$ already implies the continuity of each $\rmetres{V}{\cdot}{\cdot}{\Gamma}$, $V \in \metregions$. Suppose that there is an event with positive probability on which some $\rmetres{V}{\cdot}{\cdot}{\Gamma}$, $V \in \metregions$, is discontinuous. By performing the random resampling procedure described in \cite[Section~5.2]{amy-cle-resampling} and using the Markovian property~\eqref{it:rmet_markov}, we conclude that there is a positive probability that some $\rmetres{V}{\cdot}{\cdot}{\Gamma}$ is discontinuous and further that every point in $\Upsilon_\Gamma \setminus \ol{V}$ is separated from $V$ by a single point $z$. But by property~\eqref{it:rmet_compatibility} we then have $\rmetres{V}{\cdot}{\cdot}{\Gamma} = \rmet{\cdot}{\cdot}{\Gamma}|_{\ol{V} \cap \Upsilon_\Gamma}$ which contradicts the continuity of $\rmetres{C}{\cdot}{\cdot}{\Gamma}$.

Finally, since $(\ol{V} \cap \Upsilon_\Gamma, \dpathY[\ol{V}])$ is compact and $\rmetres{V}{\cdot}{\cdot}{\Gamma}$ is a true metric, we also have that $\dpathY[\ol{V}]$ is continuous with respect to $\rmetres{V}{\cdot}{\cdot}{\Gamma}$, i.e.\ the two metrics are topologically equivalent.
\end{proof}

The next lemma generalizes the properties~\eqref{it:weak_rform_compatibility},~\eqref{it:weak_rform_cut_point} to non-constant functions and a countable number of ``dead ends'', respectively.

\begin{lemma}\label{le:weak_rform_dead_ends}
 Suppose that $n \in \N \cup \{\infty\}$ and $V_i \in \metregions$, $0 \le i < n$, are such that for each $i,j > 0$, $i \neq j$, we have $\dpathY(V_i,V_j) > 0$ and $\ol{V_i} \cap \ol{V_0}$ consists of a single point. Let $V = \bigcup_{i=0}^n V_i$. Then
\begin{align*}
 \rfdomainres{V}{\Gamma} &= \left\{ f : f|_{V_i} \in \rfdomainres{V_i}{\Gamma} \text{ for each } 0 \le i < n ,\ \sum_{i=0}^n \rformres{V_i}{f}{f}{\Gamma} < \infty \right\} \\
 \text{and}\quad
 \rformres{V}{f}{f}{\Gamma} &= \sum_{i=0}^n \rformres{V_i}{f}{f}{\Gamma} .
\end{align*}
\end{lemma}

\begin{proof}
By property~\eqref{it:weak_rform_cut_point}, the statement holds when $n$ is finite. When $n=\infty$, we argue as follows. Choose a countable dense set $\{z_m\}$ in $\ol{V} \cap \Upsilon_\Gamma$ including the points $z'_i \in \ol{V_i} \cap \ol{V_0}$ and enumerate them so that $z'_i$ comes first among $\{z_m\} \cap \ol{V_i}$ for each $i$. For each $m$, let $A_m = \{z_1,\ldots,z_m\}$, and let $\rformrestr{V}{A_m}{\cdot}{\cdot}{\Gamma}$ be the trace of the resistance form $\rformres{V}{\cdot}{\cdot}{\Gamma}$ (see Section~\ref{se:rmet_rform}). By property~\eqref{it:weak_rform_cut_point} we have
\[
 \rformrestr{V}{A_m}{f}{f}{\Gamma} = \sum_{i: A_m \cap \ol{V_i} \neq \emptyset} \rformrestr{V_i}{A_m \cap \ol{V_i}}{f}{f}{\Gamma} \le \sum_{i=0}^\infty \rformres{V_i}{f}{f}{\Gamma}
\]
for each $m$, hence by~\eqref{eq:rmet_to_form_sep},
\[
 \rformres{V}{f}{f}{\Gamma} \le \sum_{i=0}^\infty \rformres{V_i}{f}{f}{\Gamma} .
\]
Conversely, due to~\eqref{it:weak_rform_compatibility} and~\eqref{it:weak_rform_cut_point} we have
\[
 \rformres{V}{f}{f}{\Gamma} \ge \rformres{\cup_{i=0}^k V_i}{f}{f}{\Gamma} = \sum_{i=0}^k \rformres{V_i}{f}{f}{\Gamma}
\]
for each $k \in \N$, and hence
\[
 \rformres{V}{f}{f}{\Gamma} \ge \sum_{i=0}^\infty \rformres{V_i}{f}{f}{\Gamma} .
\]
\end{proof}

The next lemma generalizes~\eqref{it:weak_rform_cut_point},~\eqref{it:weak_rform_cut_loop} to the decomposition into a finite number of subregions.

\begin{lemma}\label{le:rform_additivity}
 Suppose that $\rform{\cdot}{\cdot}{\Gamma}$ is a weak \clekp{} resistance form. Then almost surely the following holds. Let $V \in \metregions$ and $V_1,\ldots,V_n \in \metregions$, $n \in \N$, such that $V_i \cap V_j = \emptyset$ for $i \neq j$ and $\bigcup_i \ol{V_i} \cap \Upsilon_\Gamma = \ol{V} \cap \Upsilon_\Gamma$. Then
 \begin{align*}
  \rfdomainres{V}{\Gamma} &= \{ f : f|_{V_i} \in \rfdomainres{V_i}{\Gamma} \text{ for each } i \} ,\\
  \text{and}\quad
  \rformres{V}{f}{f}{\Gamma} &= \sum_i \rformres{V_i}{f}{f}{\Gamma} .
 \end{align*}
\end{lemma}

\begin{proof}
This follow by repeatedly applying~\eqref{it:weak_rform_cut_loop} and Lemma~\ref{le:weak_rform_dead_ends}. First, we can decompose $V$ into its connected components via Lemma~\ref{le:weak_rform_dead_ends} (recall that $(\ol{V},d_{\ol{V}})$ is simply connected by the definition of~$\metregions$). Suppose now that $V$ consists of a single component and contains at least two of the $V_i$. By the assumption, for each $i \neq j$, the set $\ol{V_i} \cap \ol{V_j}$ is finite. We inductively decompose $\ol{V} \cap \Upsilon_\Gamma$ along these finitely many intersection points. Note that at least one of the intersection points must lie on $\partial V$ since each $(\ol{V_i},d_{\ol{V_i}})$ is supposed to be simply connected. We can therefore apply~\eqref{it:weak_rform_cut_loop}. Proceeding by induction, we obtain the desired decomposition.
\end{proof}

\begin{lemma}
 Suppose that $(\rformres{V}{\cdot}{\cdot}{\Gamma})_{V \in \metregions}$ is a weak \clekp{} resistance form. Then each $\rformres{V}{\cdot}{\cdot}{\Gamma}$ is a local, regular resistance form.
\end{lemma}

\begin{proof}
Recall that each resistance form associated to a compact resistance metric space is regular. It is shown in \cite[Lemma~1.10]{amy2025tightness} that each $(\ol{V} \cap \Upsilon_\Gamma, \dpathY[\ol{V}])$ is compact, and since $\rmetres{V}{\cdot}{\cdot}{\Gamma}$ is continuous with respect to $\dpathY$ (Lemma~\ref{le:rmet_continuous}), also $(\ol{V} \cap \Upsilon_\Gamma, \rmetres{V}{\cdot}{\cdot}{\Gamma})$ is compact.

We now argue that $\rformres{V}{\cdot}{\cdot}{\Gamma}$ is local. Suppose that $f,g \in \rfdomainres{V}{\Gamma}$ and $\supp(f) \cap \supp(g) = \emptyset$. By \cite[Lemma~5.12,~5.14]{amy-cle-resampling}, we can find a finite number of points in $\ol{V} \cap \Upsilon_\Gamma$ that separate $\supp(f)$ from $\supp(g)$. Let $V_1,V_2 \in \metregions$ be disjoint subregions containing $\supp(f),\supp(g)$, respectively. Applying Lemma~\ref{le:rform_additivity} twice, we see that
\[
 \rformres{V}{f+g}{f+g}{\Gamma}
 = \rformres{V_1}{f}{f}{\Gamma}+\rformres{V_2}{g}{g}{\Gamma}
 = \rformres{V}{f}{f}{\Gamma}+\rformres{V}{g}{g}{\Gamma}
\]
and therefore
\[
 4\rformres{V}{f}{g}{\Gamma} = \rformres{V}{f+g}{f+g}{\Gamma}-\rformres{V}{f-g}{f-g}{\Gamma} = 0 .
\]
\end{proof}

The next lemma states that the associated family of resistance metrics $\rmet{\cdot}{\cdot}{\Gamma} = (\rmetres{V}{\cdot}{\cdot}{\Gamma})_{V \in \metregions}$ satisfies the definition of a \clekp{} metric in~\cite[Definition~1.5]{amy2025tightness} with $\epsilon = 0$ (which is recalled in Section~\ref{se:approx_scheme_general_def} below). As a result, all results from \cite{amy2025tightness} apply to $\rmet{\cdot}{\cdot}{\Gamma}$.

\begin{lemma}
 Suppose that $\rmet{\cdot}{\cdot}{\Gamma} = (\rmetres{V}{\cdot}{\cdot}{\Gamma})_{V \in \metregions}$ is the family of resistance metrics associated with a weak \clekp{} resistance form. Then $\rmet{\cdot}{\cdot}{\Gamma}$ is a \clekp{} metric in the sense of~\cite[Definition~1.5]{amy2025tightness} with $\epsilon = 0$.
\end{lemma}

\begin{proof}
The continuity, the Markovian property, the translation invariance, and the compatibility property are exactly items~\eqref{it:rmet_cont}--\eqref{it:rmet_compatibility} in Proposition~\ref{pr:cle_rmet_char} and Lemma~\ref{le:rmet_continuous}. The remaining properties are stated in the Lemmas~\ref{le:weak_rmet_separability}--\ref{le:weak_rmet_gen_parallel_law} below.
\end{proof}

\begin{lemma}[Separability]\label{le:weak_rmet_separability}
The following holds almost surely. For every $V \in \metregions$,
\[ \rmetres{V}{x}{y}{\Gamma} = \lim_{V' \searrow V} \rmetres{V'}{x}{y}{\Gamma} ,\quad x,y \in \ol{V} \cap \Upsilon_\Gamma . \]
\end{lemma}

\begin{proof}
 The definition of~$\metregions$ implies that if $V' \supseteq V$ is contained in a sufficiently small neighborhood of $V$, then for every $u \in V' \setminus V$ there is a point $z \in \ol{V}$ that separates $u$ from $V$ in $\ol{V'}$. In that case, we have $\rmetres{V}{x}{y}{\Gamma} = \rmetres{V'}{x}{y}{\Gamma}$ for $x,y \in \ol{V} \cap \Upsilon_\Gamma$ by Proposition~\ref{pr:cle_rmet_char}\eqref{it:rmet_compatibility}.
\end{proof}

\begin{lemma}[Monotonicity]\label{le:weak_rmet_monotonicity}
The following holds almost surely. Let $V,V' \in \metregions$, $V \subseteq V'$. Then
\[ \rmetres{V'}{x}{y}{\Gamma} \le \rmetres{V}{x}{y}{\Gamma} ,\quad x,y \in \ol{V} \cap \Upsilon_\Gamma . \]
\end{lemma}

\begin{proof}
Let $f$ be the $\rformres{V'}{\cdot}{\cdot}{\Gamma}$-harmonic function with $f(x) = 1$, $f(y) = 0$. The Lemma~\ref{le:weak_rform_dead_ends} and~\ref{le:rform_additivity} imply that $\rformres{V}{f}{f}{\Gamma} \le \rformres{V'}{f}{f}{\Gamma}$. Therefore $\rmetres{V}{x}{y}{\Gamma} \ge \rformres{V}{f}{f}{\Gamma}^{-1} \ge \rformres{V'}{f}{f}{\Gamma}^{-1} = \rmetres{V'}{x}{y}{\Gamma}$.
\end{proof}

\begin{lemma}[Series law]\label{le:weak_rmet_series_law}
The following holds almost surely. Let $V \in \metregions$ and $x,y,z \in \ol{V} \cap \Upsilon_\Gamma$ such that $z$ separates $x$ from $y$ in $\ol{V} \cap \Upsilon_\Gamma$. Then
\[ \rmetres{V}{x}{y}{\Gamma} = \rmetres{V}{x}{z}{\Gamma}+\rmetres{V}{z}{y}{\Gamma} . \]
\end{lemma}

\begin{proof}
In case $z \in \partial V$, this follows from~\eqref{it:weak_rform_cut_point} and Proposition~\ref{pr:cle_rmet_char}\eqref{it:rmet_compatibility} by Lemma~\ref{le:rforms_cutpoint}. In case $z$ is in the interior of $V$, we decompose $\ol{V} \cap \Upsilon_\Gamma$ via Lemma~\ref{le:rform_additivity}.
\end{proof}

\begin{lemma}[Generalized parallel law]\label{le:weak_rmet_gen_parallel_law}
The following holds almost surely. Let $V \in \metregions$ and let $x,y,z_1,\ldots,z_N \in \ol{V} \cap \Upsilon_\Gamma$ such that $x,y$ are separated in $\ol{V} \cap \Upsilon_\Gamma \setminus\{z_1,\ldots,z_N\}$. Then
\[ N\rmetres{V}{x}{y}{\Gamma} \ge \min_{i}\rmetres{V_x}{x}{z_i}{\Gamma} \]
where $V_x \in \metregions$ is such that $\ol{V_x}$ contains the connected component of $\ol{V} \cap \Upsilon_\Gamma \setminus\{z_1,\ldots,z_N\}$ containing $x$.
\end{lemma}

\begin{proof}
Let $w_V$ (resp.\ $w_{V_x}$) be the edge weight functions associated with the resistance metrics $\rmetres{V}{\cdot}{\cdot}{\Gamma}|_{\{x,y,z_1,\ldots,z_N\}}$ (resp.\ $\rmetres{V_x}{\cdot}{\cdot}{\Gamma}|_{\{x,z_1,\ldots,z_N\}}$). It follows from Lemma~\ref{le:rform_additivity} that $w_V(x,y) = 0$ and $w_V(x,z_i) = w_{V_x}(x,z_i)$ for each $i$. The result therefore follows by Lemma~\ref{le:gen_parallel_law_graph}.
\end{proof}

The following property will be used later.

\begin{lemma}\label{le:energy_small_nbhd}
 Suppose that $\rform{\cdot}{\cdot}{\Gamma}$ is a weak \clekp{} resistance form. Then almost surely the following holds. Let $V \in \metregions$ and let $f$ be a function that is $\rformres{V}{\cdot}{\cdot}{\Gamma}$-harmonic on $\ol{V} \cap \Upsilon_\Gamma \setminus A$ for some finite set $A$. Then, for each $x \in \ol{V} \cap \Upsilon_\Gamma$ and $\varepsilon > 0$, there exists $V_1 \in \metregions$ containing a $\dpathY[\ol{V}]$-neighborhood of $x$ such that $\rformres{V_1}{f}{f}{\Gamma} < \varepsilon$.
\end{lemma}

\begin{proof}
As a consequence of the independence across scales~\cite{amy-cle-resampling} (see e.g.\ Corollary~\ref{co:covering}), there exists a constant $M$ such that the following holds almost surely. For each $x \in \Upsilon_\Gamma$ and $\delta > 0$ there is $0<\delta_1<\delta$ and a set of at most $M$ points $\{v_i\} \subseteq A(x,\delta_1/2,\delta_1)$ that separate $\partial B(x,\delta_1/2)$ from $\partial B(x,\delta_1)$ in $\Upsilon_\Gamma$. Suppose that $\delta < \dist(x, A \setminus \{x\})$. Let $\delta_1$ and $\{v_i\}$ be as above, and let $V_\delta \in \metregions$ be the region bounded by the connected component of $\ol{V} \cap \Upsilon_\Gamma \setminus \{v_i\}$ containing $x$. Suppose for notational simplicity that $V \setminus V_\delta \in \metregions$ (i.e.\ it is simply connected), otherwise we split it into two simply connected regions and apply the same argument. Let $w_\delta(y,v_i)$, $y \in A \setminus \{x\}$, be the edge conductances associated with $\rformrestr{V \setminus V_\delta}{(A \cup \{v_i\}) \setminus \{x\}}{\cdot}{\cdot}{\Gamma}$. We claim that the $w_\delta(y,v_i)$ stay bounded as $\delta \to 0$. Indeed, let $g$ be the function with $g(x)=1$, $g(y)=0$ for $y \in A \setminus \{x\}$, and $\rformres{V}{\cdot}{\cdot}{\Gamma}$-harmonic otherwise. Since $g$ is continuous~\eqref{eq:fin_energy_continuous}, we have $g(v_i) > 1/2$ for each $i$ when $\delta$ is sufficiently small. Therefore
\begin{align*}
w_\delta(y,v_i) \cdot \left(\frac{1}{2}\right)^2
&\le \rformrestr{V \setminus V_\delta}{(A \cup \{v_i\}) \setminus \{x\}}{g}{g}{\Gamma} \quad\text{(energy across one edge $\leq$ total energy)}\\
&\le \rformres{V}{g}{g}{\Gamma} \quad\text{(by Lemma~\ref{le:rform_additivity})}.
\end{align*}

Now let $f$ be $\rformres{V}{\cdot}{\cdot}{\Gamma}$-harmonic on $\ol{V} \cap \Upsilon_\Gamma \setminus A$. Suppose that $\rformres{V_\delta}{f}{f}{\Gamma}$ does not converge to $0$ as $\delta \to 0$. Let $\wt{f}$ be constant $f(x)$ on $V_\delta$, have the same values as $f$ on $A$, and harmonically interpolated to the rest of $V \setminus V_\delta$. By the continuity of $f$~\eqref{eq:fin_energy_continuous} and using that $\wt{f}(v_i) = f(x)$ we have $\max_i \abs{f(v_i)-\wt{f}(v_i)} \to 0$ as $\delta \to 0$. Since the $w_\delta(y,v_i)$ stay bounded, we get $\abs{\rformres{V \setminus V_\delta}{\wt{f}}{\wt{f}}{\Gamma} - \rformres{V \setminus V_\delta}{f}{f}{\Gamma}} \to 0$, and hence $\rformres{V}{\wt{f}}{\wt{f}}{\Gamma} < \rformres{V}{f}{f}{\Gamma}$ for small $\delta$ due to Lemma~\ref{le:rform_additivity}. This contradicts the harmonicity of $f$.
\end{proof}

\subsection{General tightness statement}
\label{sec:tightness_general}

We review the general tightness result for metrics on the \clekp{} gasket proved in \cite{amy2025tightness}.

\subsubsection{Setup and Topology}
\label{subsec:tightness_setup}

Throughout, we will assume the following setup. Let $\Gamma_\D$ be a nested $\CLE_{\kappa'}$ on $\D$, let $\CL$ be the outermost loop of $\Gamma$ such that $0$ is inside $\CL$, and let $C$ be the regions inside $\CL$. We let $\Gamma_C$ be the loops of $\Gamma_\D$ contained in $\ol{C}$ and let $\Gamma = \{ \CL \} \cup \Gamma_C$. We equip the gasket $\Upsilon_\Gamma$ of $\Gamma_C$ with the metric $\dpathY$ defined in~\eqref{eq:dpath}.

We now describe the topology for which the tightness statement is proved (cf.~\cite[Section~1.3]{amy2025tightness}). As described in Section~\ref{subsec:main_results}, the gasket $\Upsilon_\Gamma$ is an abstract metric space $(\gasketspace{\Gamma},\dpathY)$ equipped with an embedding $\projmap{\Gamma}\colon \gasketspace{\Gamma} \to \C$. In addition, we specify the sets of points in the gasket that are ``separating''. These sets play an essential role in the definition of the \clekp{} metrics from \cite{amy2025tightness} which we will recall just below. Concretely, for each $n \in \N$, let
\[ A_n = \{ (x,y,z_1,\ldots,z_n) \in \gasketspace{\Gamma}^{2+n} : \text{The set $\{z_1,\ldots,z_n\}$ separates $x$ from $y$} \} . \]
We consider the space of sequences $((\gasketspace{\Gamma}^2, \dpathY, \projmap{\Gamma}, \FR), (\gasketspace{\Gamma}^{2+n}, \dpathY, \one_{A_n})_{n \in \N})$ where $(\gasketspace{\Gamma}^2, \dpathY)$ (resp.\ $(\gasketspace{\Gamma}^{2+n}, \dpathY)$) denotes the product metric space, $\FR\colon \gasketspace{\Gamma}^2 \to \R$, and we equip the space of $(\gasketspace{\Gamma}^2, \dpathY, \projmap{\Gamma}, \FR)$ (resp.\ $(\gasketspace{\Gamma}^{2+n}, \dpathY, \one_{A_n})$) with the Gromov-Hausdorff-function (GHf) topology defined in \cite[Appendix~A]{amy2025tightness}. Note that by \cite[Lemma~A.9]{amy2025tightness}, when taking limits of such sequences in the product topology, the spaces in the entries of the limiting sequence remain compatible.

\subsubsection{Good approximation schemes}
\label{se:approx_scheme_general_def}

We now give the definition of the class of approximation schemes for which tightness is proved in \cite{amy2025tightness}. We first give the definition of an approximate \clekp{} metric according to \cite[Definition~1.5]{amy2025tightness}.

Let $\epsilon \ge 0$ and $\cserial \ge 0$, $\cparallel(N) > 0$ for each $N \in \N$ be some fixed constants. Let $\metregions$ be as defined in Section~\ref{subsec:main_results}. Suppose that we have a family of random functions $\rmetapprox{\epsilon}{\cdot}{\cdot}{\Gamma} = (\rmetapproxres{\epsilon}{V}{\cdot}{\cdot}{\Gamma})_{V \in \metregions}$\footnote{This family is determined by a countable subfamily due to the separability assumption given below.} coupled with $\Gamma$ (but not necessarily determined by $\Gamma$). We say that $\rmetapprox{\epsilon}{\cdot}{\cdot}{\Gamma}$ is an \emph{approximate \clekp{} metric} (with respect to the parameters $\epsilon,\cserial,\cparallel(N)$) if the following properties are satisfied.
\begin{enumerate}[\textbullet]
\item For each $V \in \metregions$ we have
\[ \rmetapproxres{\epsilon}{V}{\cdot}{\cdot}{\Gamma}\colon (\ol{V} \cap \Upsilon_\Gamma) \times (\ol{V} \cap \Upsilon_\Gamma) \to [0,\infty] , \]
and
\[
 \rmetapproxres{\epsilon}{V}{x}{y}{\Gamma} = \rmetapproxres{\epsilon}{V}{y}{x}{\Gamma}
 \quad\text{and}\quad
 \rmetapproxres{\epsilon}{V}{x}{y}{\Gamma} \le \rmetapproxres{\epsilon}{V}{x}{z}{\Gamma}+\rmetapproxres{\epsilon}{V}{z}{y}{\Gamma} .
\]

\item In the case $\epsilon = 0$, we assume additionally that each $\rmetres{V}{\cdot}{\cdot}{\Gamma}$ is continuous with respect to $\dpathY[\ol{V}]$, and $\rmetres{V}{x}{x}{\Gamma} = 0$ for every $x \in \ol{V} \cap \Upsilon_\Gamma$.

\item \textbf{Separability:}
For every $V \in \metregions$,
\[
 \rmetapproxres{\epsilon}{V}{x}{y}{\Gamma} = \lim_{V' \searrow V} \rmetapproxres{\epsilon}{V'}{x}{y}{\Gamma} ,\quad x,y \in \ol{V} \cap \Upsilon_\Gamma .
\]
The limit is in the sense that convergence holds for any decreasing sequence of $V'_n \in \metregions$ with $V'_n \supseteq V$ and $\sup_{u \in \ol{V'_n} \cap \Upsilon_\Gamma} \dpathY[\ol{V'_n}](u,V) \to 0$.

\item \textbf{Markovian property:} Let $U \subseteq \C$ be open, simply connected. The conditional law of the collection $(\rmetapproxres{\epsilon}{V}{\cdot}{\cdot}{\Gamma})_{V \in \metregions[U]}$ given $\Gamma\setminus\Gamma_{U^*}$ and $(\rmetapproxres{\epsilon}{V'}{\cdot}{\cdot}{\Gamma})_{V' \in \metregions[\C\setminus\ol{U}]}$ is almost surely measurable with respect to~$U^*$.

\item \textbf{Translation invariance:} Let $U \subseteq \C$ be open, simply connected. There exists a probability kernel $\mu^{U^*}$ such that for each $z \in \C$, the conditional law of $(\rmetapproxres{\epsilon}{V}{\cdot}{\cdot}{\Gamma})_{V \in \metregions[U+z]}$ given $(U+z)^*$ is $(T_z)_* \mu^{U^*}(\cdot -z)$ where $T_z \FR(x,y) = \FR(x-z,y-z)$ denotes translation by $z$.

\item \textbf{Compatibility:} Let $V,V' \in \metregions$, $V \subseteq V'$, and $x,y \in \ol{V} \cap \Upsilon_\Gamma$ such that for every $u \in V' \setminus V$ there is a point $z \in \ol{V}$ with $\dpathY[\ol{V'}](z, V' \setminus V) \ge \cserial\epsilon$ that separates $u$ from $x,y$ in $\ol{V'} \cap \Upsilon_\Gamma$. Then $\rmetapproxres{\epsilon}{V}{x}{y}{\Gamma} = \rmetapproxres{\epsilon}{V'}{x}{y}{\Gamma}$.

\item \textbf{Monotonicity:}
Let $V,V' \in \metregions$, $V \subseteq V'$. Let $x,y \in \ol{V} \cap \Upsilon_\Gamma$ and suppose one of the following:
\begin{enumerate}[(i)]
 \item\label{it:mon_dead_ends} For every $u \in V' \setminus V$ there is a point $z \in \ol{V}$ that separates $u$ from $x,y$ in $\ol{V'} \cap \Upsilon_\Gamma$.
 \item\label{it:mon_large_loops} There exist $z_1,z_2,\ldots \in \ol{V}$ with $\abs{z_i-z_{i'}} \ge \cserial\epsilon$ for $i\neq i'$ such that no \emph{simple} admissible path in $\ol{V}$ from $x$ to $y$ intersects $\{z_1,z_2,\ldots\}$, and the set $\{z_1,z_2,\ldots\}$ separates $x,y$ from $V' \setminus V$ in $\ol{V'} \cap \Upsilon_\Gamma$.
\end{enumerate}
Then there are points $x',y' \in \ol{V} \cap \Upsilon_\Gamma$ with $\dpathY[\ol{V}](x',x) \le \epsilon$, $\dpathY[\ol{V}](y',y) \le \epsilon$ such that
\[
 \rmetapproxres{\epsilon}{V'}{x'}{y'}{\Gamma} \le \rmetapproxres{\epsilon}{V}{x}{y}{\Gamma} .
\]

\item \textbf{Series law:} Let $V \in \metregions$, and $x,y,z_1,z_2 \in \ol{V} \cap \Upsilon_\Gamma$ be such that $z_1$ separates $x$ from $y,z_2$ in $\ol{V} \cap \Upsilon_\Gamma$, and $z_2$ separates $y$ from $x,z_1$ in $\ol{V} \cap \Upsilon_\Gamma$, and such that $\distE(K_x,K_y) \ge \cserial\epsilon$ where $K_x$ (resp.\ $K_y$) are the connected components of $\ol{V} \cap \Upsilon_\Gamma \setminus \{z_1,z_2\}$ containing $x$ (resp.\ $y$). (Here, $\distE$ denotes the distance with respect to the Euclidean metric.) Then
\[ \rmetapproxres{\epsilon}{V}{x}{y}{\Gamma} \ge \rmetapproxres{\epsilon}{V}{x}{z_1}{\Gamma}+\rmetapproxres{\epsilon}{V}{z_2}{y}{\Gamma}.\]
Further, if $z$ separates $x$ from $y$ in $\ol{V} \cap \Upsilon_\Gamma$, then
\[ \rmetapproxres{\epsilon}{V}{x}{y}{\Gamma} \ge \rmetapproxres{\epsilon}{V}{x}{z}{\Gamma} . \]

\item \textbf{Generalized parallel law:} Let $V \in \metregions$, and let $x,y,z_1,\ldots,z_N \in \ol{V} \cap \Upsilon_\Gamma$ such that $\abs{z_i-z_{i'}} \ge \cserial\epsilon$ for $i \neq i'$, and $x,y$ are separated in $\ol{V} \cap \Upsilon_\Gamma \setminus\{z_1,\ldots,z_N\}$. Let $K_x$ be the connected component of $\ol{V} \cap \Upsilon_\Gamma \setminus\{z_1,\ldots,z_N\}$ containing $x$, and let $V_x \subseteq V$ such that $\ol{V_x} \supseteq \{u \in V : \dpathY[\ol{V}](u,K_x) \le \cserial\epsilon\}$. Then
\[ \cparallel(N)\rmetapproxres{\epsilon}{V}{x}{y}{\Gamma} \ge \min_{i}\rmetapproxres{\epsilon}{V_x}{x}{z_i}{\Gamma} . \]
\end{enumerate}

We now recall the definition of a good approximation scheme \cite[Definition~1.6]{amy2025tightness} which includes the additional assumption that asymptotically as $\epsilon \searrow 0$, the $\rmetapprox{\epsilon}{\cdot}{\cdot}{\Gamma}$-distances in regions smaller than $\epsilon$ are much smaller than the renormalization constants $\median{\epsilon}$ defined in~\cite[Section~3.1]{amy2025tightness} which we now recall.

Let
\begin{equation}\label{eq:angledouble}
 \angledouble = \frac{\pi(16/\kappa'-2)}{2-8/\kappa'} .
\end{equation}
Let $h$ be a GFF on $\D$ with boundary values so that if $\varphi\colon \h \to \D$ is a conformal map, then $h \circ \varphi - \chi\arg\varphi'$ has the boundary values $-\lambda-\angledouble\chi$ (resp.\ $+\lambda$) on $\R_-$ (resp.\ $\R_+$). Let $\eta_1$ (resp.\ $\eta_2$) be the angle $\angledouble$ (resp.\ $0$) flow line of $h$ from $-i$ to $i$. Let $U$ be the regions bounded between $\eta_1,\eta_2$. Given $\eta_1,\eta_2$, let $\Gamma$ be the conditionally independent collection of \clekp{}'s in the connected components of $U$. Let $I$ be the event that $\eta_1 \cap \eta_2 \cap B(0,1/2) \neq \emptyset$. Let $\median{\epsilon}$ be the median of the random variable
\begin{equation}\label{eq:median_def}
 \sup_{\substack{x=\eta_1(s)=\eta_2(s')\\ y=\eta_1(t)=\eta_2(t')\\ \eta_1[s,t], \eta_2[s',t'] \subseteq B(0,3/4)}} \rmetapproxres{\epsilon}{U}{x}{y}{\Gamma} \quad\text{conditioned on } I .
\end{equation}
As explained in~\cite[Section~3.1]{amy2025tightness}, the law of $(\eta_1,\eta_2)$, restricted a compact subset of $\D$, is absolutely continuous with respect to the law of the outer boundaries of two intersecting \clekp{} loops. Therefore we can regard $U$ as a region belonging to $\metregions$, and $\rmetapproxres{\epsilon}{U}{\cdot}{\cdot}{\Gamma}$ is well-defined.

\newcommand*{\epsexp}{a_0}

Suppose $\cserial \ge 0$, $\cparallel(N) > 0$ are fixed constants, and that for each $\epsilon \in (0,1]$, we have an approximate \clekp{} metric $\rmetapprox{\epsilon}{\cdot}{\cdot}{\Gamma}$ with respect to $\epsilon,\cserial,\cparallel(N)$. We say that $(\rmetapprox{\epsilon}{\cdot}{\cdot}{\Gamma})_{\epsilon \in {(0,1]}}$ is a \emph{good approximation scheme} if there exists a constant $\ac{\epsilon} > 0$ for each $\epsilon \in {(0,1]}$ and a constant $\epsexp > 0$ such that
\begin{equation}\label{eq:eps_bound_ass}
 \lim_{\epsilon \searrow 0} \frac{\ac{\epsilon}}{\epsilon^{\epsexp}\median{\epsilon}} = 0
\end{equation}
and for each $r>0$ we have
\begin{equation}\label{eq:approx_error_asymp}
 \lim_{\epsilon \searrow 0} \p\left[ \sup_{V \in \metregions} \sup_{\substack{\dpathY[\ol{V}](x,y) < \epsilon ,\\ \dpathY(x, \Upsilon_\Gamma \setminus \ol{V}) \ge r}} \rmetapproxres{\epsilon}{V}{x}{y}{\Gamma} \le \ac{\epsilon} \right] = 1 .
\end{equation}

\subsubsection{Tightness theorem}

We are now ready to recall the main result from \cite{amy2025tightness}.

\begin{theorem}[{\cite[Theorems~1.12,~1.13]{amy2025tightness}}]
\label{thm:tightness_theorem}
Suppose that we have the setup described in Section~\ref{subsec:tightness_setup}, let $(\rmetapprox{\epsilon}{\cdot}{\cdot}{\Gamma})_{\epsilon \in (0,1]}$ be a good approximation scheme as defined in Section~\ref{se:approx_scheme_general_def}, and let $\median{\epsilon} > 0$ be as defined just above. Let $(\epsilon_n)$ be a sequence with $\epsilon_n \to 0$.
\begin{enumerate}[(i)]
 \item\label{it:tightness} Let $V \in \metregions$ be chosen according to some probability distribution given $\Gamma$. Then laws of $\median{\epsilon_n}^{-1} \rmetapproxres{\epsilon_n}{V}{\cdot}{\cdot}{\Gamma}$ restricted to subsets with positive $\dpathY$-distance to $\Upsilon_\Gamma \setminus \ol{V}$ are tight with respect to the topology described just above.

 \item\label{it:limit_cle_met} Each subsequential limit of $\median{\epsilon_n}^{-1} \rmetapproxres{\epsilon_n}{C}{\cdot}{\cdot}{\Gamma}$ has the law of $\rmetres{C}{\cdot}{\cdot}{\Gamma}$ for a \clekp{} metric $\rmet{\cdot}{\cdot}{\Gamma}$ in the sense of Section~\ref{se:approx_scheme_general_def} with $\epsilon = 0$.
\end{enumerate}
\end{theorem}

\subsection{Tightness of resistance metrics}
\label{subsec:resistance_tightness}

In this section we will consider a certain graph approximation of the $\CLE_{\kappa'}$ gasket and then show that the associated resistance metrics are tight. We wish to apply Theorem~\ref{thm:tightness_theorem} to this approximation scheme, and for this we would like to show that it satisfies the definition of a good approximation scheme given in Section~\ref{se:approx_scheme_general_def}. However, as we will explain later, due to a technical complication we will modify our approximation scheme slightly so that the conditions in Section~\ref{se:approx_scheme_general_def} are not satisfied exactly as stated. We will explain that the proof of Theorem~\ref{thm:tightness_theorem} is not affected by our modification, and the theorem still holds for our approximation scheme.

\subsubsection{Approximation scheme}
\label{se:rmet_approx_scheme}

\newcommand*{\rmvdeadends}[1]{s(#1)}

To build an approximation scheme for the \clekp{} resistance metric, we will make use of the geodesic \clekp{} metric $\met{\cdot}{\cdot}{\Gamma}$ (see Section~\ref{se:geodesic_metric}). In fact, we only need a \emph{weak geodesic \clekp{} metric} (see~\cite{my2025geouniqueness} for the definition) which is proved to exist in~\cite[Theorem~1.17]{amy2025tightness}. (It is shown in~\cite{my2025geouniqueness} that the weak geodesic \clekp{} metric is unique up to a constant factor and is in fact a (strong) geodesic \clekp{} metric.) We note that the use of the geodesic \clekp{} metric is for convenience to simplify some technical issues, but there are other ways around them.

Consider the setup described just above in Section~\ref{subsec:tightness_setup}. In the following, we let $\met{\cdot}{\cdot}{\Gamma}$ denote the geodesic \clekp{} metric, and let $\lmet{\gamma}$ denote the length of a path $\gamma$ with respect to this metric.

Let $\dcle$ be as in~\eqref{eqn:cle_dim}.  Fix a small constant $\rateexp>0$, and let
\begin{equation}\label{eq:ppp_rate}
 \lambda_\epsilon = \epsilon^{-\dcle-\rateexp} ,\quad \epsilon > 0 .
\end{equation}
Let $\meas{\cdot}{\Gamma}$ denote the gasket measure (see Section~\ref{se:cle_measure}), and let $\approxvtcs{\epsilon}$ be a Poisson point process with intensity measure $\lambda_\epsilon \meas{\cdot}{\Gamma}$.  As we will explain just below, the concentration results from \cite{my2026cleminkowski} for $\meas{\cdot}{\Gamma}$ imply that every Euclidean ball of radius $\epsilon$ is likely to contain many points of $\approxvtcs{\epsilon}$ (i.e., at least a negative power of $\epsilon$).

We also fix a function $\rmvdeadends{\epsilon} > 0$ that converges to $0$ slowly as $\epsilon \to 0$. For instance, we choose
\begin{equation}\label{eq:rmvdeadends_size}
 \rmvdeadends{\epsilon} = \epsilon^{a_0}
\end{equation}
where $a_0>0$ is a small constant.

For each $V \in \metregions$, we now define a graph approximation of $\Upsilon_\Gamma \cap \ol{V}$. Let $\Gamma_{V,\rmvdeadends{\epsilon}} \subseteq \Gamma$ be the set of loops $\CL \subseteq \ol{V}$ with $\diamE(\CL) \ge \rmvdeadends{\epsilon}$. Let $V_{(\rmvdeadends{\epsilon})} \subseteq V$ be the set of points $u$ such that there is a loop $\CL \in \Gamma_{V,\rmvdeadends{\epsilon}}$ and a point $z \in \CL$ such that $u$ is contained in a connected component of $\ol{\C \setminus \CL} \setminus \{z\}$ (here the closure is taken with respect to $\dpathY$) of Euclidean diameter at most $\rmvdeadends{\epsilon}$. That is, $V_{(\rmvdeadends{\epsilon})}$ consists of the points in $V$ that are contained in a ``dead end'' of diameter at most $\rmvdeadends{\epsilon}$ formed by a loop of size at least $\rmvdeadends{\epsilon}$.  Let $\approxvtcs[V]{\epsilon} = \approxvtcs{\epsilon} \cap V \setminus V_{(\rmvdeadends{\epsilon})}$. Note that if $V \subseteq V'$, then $V_{(\rmvdeadends{\epsilon})} \subseteq V'_{(\rmvdeadends{\epsilon})}$. However, it is not always true that $V \setminus V_{(\rmvdeadends{\epsilon})} \subseteq V' \setminus V'_{(\rmvdeadends{\epsilon})}$ because some components of $V$ may be contained in $V'_{(\rmvdeadends{\epsilon})}$ (see Figure~\ref{fi:mod_not_monotone} below).

\begin{remark}
 We removed the dead end vertices in $V_{(\rmvdeadends{\epsilon})}$ in order to simplify the proof that the subsequential limits are non-trivial (Lemma~\ref{le:nondegeneracy}). Let us explain this briefly; see the paragraphs above Lemma~\ref{le:nondegeneracy} and the beginning of Section~\ref{se:rmet_construction} for further details. In order to show that the limit is non-zero, we observe that each pair of points is separated by a finite chain of intersecting \clekp{} loops. By the parallel law, it suffices to show that the resistances in the regions that lie between the intersection of each pair of loops is non-zero. Recall that the renormalization constant $\median{\epsilon}$ defined in Section~\ref{se:approx_scheme_general_def} and used in \cite{amy2025tightness} is given by the median resistance in a typical region bounded between the \emph{outer boundaries} of two intersecting \clekp{} loops. This region does not contain the vertices in the dead ends that are formed by each of the two loops. Note that by including the vertices in the dead ends, we also add edges between the other vertices and hence reduce the resistance (see Figure~\ref{fi:dead_ends_median} below). We therefore need to show that including the vertices in the dead ends does not change the median by more than a bounded factor. To simplify this problem, we decided to modify the construction so that only the dead ends of size at least $\rmvdeadends{\epsilon}$ are included, and their number is now easily bounded. (Note that we ultimately do need to include all dead ends in order to obtain a resistance metric on the entire gasket, but we can choose the rate $\rmvdeadends{\epsilon}$ to be as slow as we want.)

 We remark that we expect that the same is true without removing the vertices in $V_{(\rmvdeadends{\epsilon})}$, and by the uniqueness of the \clekp{} resistance form proved in this paper, the limits obtained from these approximations will agree up to a deterministic factor.
\end{remark}

\begin{figure}[ht]
\centering
\includegraphics[width=0.5\textwidth]{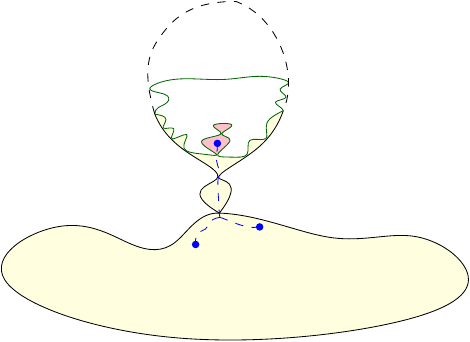}
\caption{Suppose $V$ consists of the regions shaded in yellow and pink. If $V' \supseteq V$ contains the green loop, then the pink region becomes part of $V'_{(\rmvdeadends{\epsilon})}$, so $V \setminus V_{(\rmvdeadends{\epsilon})} \nsubseteq V' \setminus V'_{(\rmvdeadends{\epsilon})}$. The dashed edges are also lost in $\approxgraph[V']{\epsilon}$.}
\label{fi:mod_not_monotone}
\end{figure}

We now define a \emph{cable graph approximation} $\approxgraph[V]{\epsilon}$ of the gasket $\ol{V} \cap \Upsilon_\Gamma$ and the associated resistance metric $\rmetapproxres{\epsilon}{V}{\cdot}{\cdot}{\Gamma}$ on $\approxgraph[V]{\epsilon}$. The precise construction seems a bit artificial, but is made to avoid technical complications as much as possible.

Let $\approxedges[V]{\epsilon} = \{ (v,w) \in \approxvtcs[V]{\epsilon} \times \approxvtcs[V]{\epsilon} : \dpathY[\ol{V}](v,w) < \epsilon \}$. For each $(v,w) \in \approxedges[V]{\epsilon}$ we choose $B_{v,w} \in \metregions$ such that $B_{v,w} \subseteq (\ol{B}_E(v,2\epsilon) \cup \ol{B}_E(w,2\epsilon)) \cap \ol{V}$ and such that every simple admissible path in $\ol{V}$ of Euclidean diameter at most $\epsilon$ from $v$ to $w$ is contained in $\ol{B}_{v,w}$. (The exact way of choosing $B_{v,w}$ does not matter as long as it is measurable with respect to the CLE configuration within $(\ol{B}_E(v,2\epsilon) \cup \ol{B}_E(w,2\epsilon)) \cap \ol{V}$ and is translation-invariant.)

For each $(v,w) \in \approxedges[V]{\epsilon}$ consider the shortest (w.r.t.\ the geodesic metric) path between $v,w$ contained in $\ol{B}_{v,w}$. (Such a path exists since $(\ol{B}_{v,w}, \metres{B_{v,w}}{\cdot}{\cdot}{\Gamma})$ is a compact geodesic space.) In case there is more than one shortest path between $v,w$ contained in $\ol{B}_{v,w}$, then we draw two paths, the left-most one seen from $v$ and from $w$. (We now explain that a left-most shortest path seen from $v$ exists. Note that almost surely, for each $v,w \in \approxvtcs{\epsilon}$ there are infinitely many loops in $\Gamma$ disconnecting $v$ from $\infty$, and likewise for $w$.\footnote{One can show that for each $r>0$, the set of points $x$ such that there are two disjoint admissible paths from $x$ to $\partial B(x,r)$ has dimension smaller than $\dcle$. The Proposition~\ref{pr:measure_ub} then implies that this set has $\meas{\cdot}{\Gamma}$-measure $0$.} These loops split $\ol{B}_{v,w}$ into a countable number of simply connected regions and annular regions. Each such annular region has two distinct boundary points through which every admissible path between $v,w$ needs to pass. Consider one such annular region, and let $v',w'$ be the distinguished boundary points. Consider its universal cover which is a periodic strip-like region, and fix a copy of $v'$ in the universal cover. Consider the set of shortest paths from $v'$ to one of the copies of $w'$ which is bounded since $\met{\cdot}{\cdot}{\Gamma}$ is a true metric. The left boundary of this set is again a shortest path which can be seen from the following two facts. 1.~If $\gamma_1,\gamma_2$ are two shortest paths, then the left boundary of $\gamma_1 \cup \gamma_2$ is also a shortest path. 2.~The limit of a sequence of shortest paths is again a shortest path.)

We let $\approxgraph[V]{\epsilon}$ denote the ``cable graph'' consisting of all of these paths which may in fact merge, and we consider the merged parts as a single edge.\footnote{This is in order to preserve the cut points of the gasket which is essential for the series / parallel laws / compatibility.}  We equip the edges with edge resistances given by their length with respect to $\lmet{\cdot}$ (this is important in order to be compatible with the merging). By the construction, given a pair $(v,w) \in \approxedges[V]{\epsilon}$, the associated cable(s) are determined by the gasket structure within the $2\epsilon$-neighborhood of $\{v,w\}$. For notational simplicity, we will use the same notation $\approxgraph[V]{\epsilon}$ to refer to the set consisting of its cables.

Consider the Brownian motion on the cable graph $\approxgraph[V]{\epsilon}$. We now explain that it is well-defined. If we know that the cables merge and separate only at finitely many points, then it is just a finite network. But as we have not proved this fact, we argue as follows. Suppose that $\gamma_1,\gamma_2$ are two intersecting cables where for $i=1,2$, $\gamma_i$ connects $(v_i,w_i) \in \approxedges[V]{\epsilon}$. Say that $\gamma_i$ is the left-most shortest path in $\ol{B}_{v_i,w_i}$ seen from $v_i$, and is oriented from $v_i$ to $w_i$. Note that by the definition of $\metregions$, the cable $\gamma_{3-i}$ can cross in and out of $\ol{B}_{v_i,w_i}$ only finitely many times. It follows that if $\gamma_1,\gamma_2$ intersect in the same orientation, they can merge and separate only finitely many times. If $\gamma_1,\gamma_2$ intersect in the opposite orientation, then they might (although we do not expect it) intersect and separate infinitely often, but each pair of intersecting segments within $\ol{B}_{v_1,w_1} \cap \ol{B}_{v_2,w_2}$ will have the same length. If a third cable intersects $\gamma_1 \cap \gamma_2$, then its orientation must be the same as $\gamma_1$ or $\gamma_2$, and will merge with one of them. Therefore the network can be reduced to a finite network.

Let $\rmetapproxres{\epsilon}{V}{\cdot}{\cdot}{\Gamma}$ be the associated effective resistance metric defined on $\approxgraph[V]{\epsilon}$. We extend it to $
\ol{V} \cap \Upsilon_\Gamma$ as follows. For $x \in \ol{V} \cap \Upsilon_\Gamma$ with $\dpathY[\ol{V}](x,\approxgraph[V]{\epsilon}) < \epsilon$, we interpolate $\rmetapproxres{\epsilon}{V}{\cdot}{\cdot}{\Gamma}$ by selecting the point on $\approxgraph[V]{\epsilon}$ with the minimal $\dpathY[\ol{V}]$-distance to $x$. In case $\dpathY[\ol{V}](x,\approxgraph[U]{\epsilon}) \ge \epsilon$ or if $x,y$ lie on different connected components of $\approxgraph[V]{\epsilon}$, we set $\rmetapproxres{\epsilon}{V}{x}{y}{\Gamma} = \infty$.
In the connected components of $V_{(\rmvdeadends{\epsilon})}$ we set $\rmetapproxres{\epsilon}{V}{\cdot}{\cdot}{\Gamma}$ to $0$.

\begin{remark}
For each $\epsilon > 0$ there are some exceptional $V \in \metregions$ that contain no or very few vertices (e.g.\ thin tubes). Such regions will not be needed in our proofs (there will be fewer and fewer of them as $\epsilon \to 0$), but in order to match the formulation in~\cite{amy2025tightness}, we nevertheless choose to define $\rmetapproxres{\epsilon}{V}{\cdot}{\cdot}{\Gamma}$ for \emph{each} $V \in \metregions$ in a way that is consistent with our axioms.
\end{remark}

\subsubsection{Statement and proof}
\label{se:rmet_approx_axioms}

We now explain that the aforementioned resistance metric approximations $(\rmetapprox{\epsilon}{\cdot}{\cdot}{\Gamma})_{\epsilon > 0}$ \emph{almost} satisfy the conditions of the tightness result in Theorem~\ref{thm:tightness_theorem}. The separability and the monotonicity assumptions considered in Section~\ref{se:approx_scheme_general_def} only hold under a small additional restriction. We will explain that this restriction does not affect the tightness of the approximation scheme and show the following result.

\begin{proposition}\label{pr:rmet_tightness}
Consider the setup described in Section~\ref{subsec:tightness_setup}, and for each $\epsilon \in (0,1]$ let $\rmetapprox{\epsilon}{\cdot}{\cdot}{\Gamma} = (\rmetapproxres{\epsilon}{V}{\cdot}{\cdot}{\Gamma})_{V \in \metregions}$ be as defined in Section~\ref{se:rmet_approx_scheme}. Let $\median{\epsilon} > 0$ be as defined Section~\ref{se:approx_scheme_general_def}.
Let $V \in \metregions$ be chosen according to some probability distribution given $\Gamma$, and for each $\epsilon$ let $V_\epsilon \in \metregions$ be such that $V_\epsilon \supseteq V$.

Then each sequence $(\epsilon_n)$ with $\epsilon_n \to 0$ contains a subsequence $(\epsilon_{n_m})$ along which we have the weak convergence of the laws of $\median{\epsilon}^{-1} \rmetapproxres{\epsilon}{V_{\epsilon}}{\cdot}{\cdot}{\Gamma}\big|_{(\ol{V} \cap \Upsilon_\Gamma) \times (\ol{V} \cap \Upsilon_\Gamma)}$ in the topology described in Section~\ref{subsec:tightness_setup}.
\end{proposition}

We begin by verifying the axioms defining an approximate \clekp{} metric for each $\rmetapprox{\epsilon}{\cdot}{\cdot}{\Gamma} = (\rmetapproxres{\epsilon}{V}{\cdot}{\cdot}{\Gamma})_{V \in \metregions}$ (except for a slight additional restriction). The Markovian property and the translation invariance are hard-wired into our construction, since both $\meas{\cdot}{\Gamma}$ and $\met{\cdot}{\cdot}{\Gamma}$ satisfy these properties. We need to verify the remaining properties.

\begin{lemma}[Separability]
\label{le:approx_separability}
For every $V \in \metregions$,
\[
 \rmetapproxres{\epsilon}{V}{x}{y}{\Gamma} = \lim_{V' \searrow V} \rmetapproxres{\epsilon}{V'}{x}{y}{\Gamma} ,\quad x,y \in \ol{V} \cap \Upsilon_\Gamma .
\]
The limit is in the sense that convergence holds for any decreasing sequence of $V'_n \in \metregions$ with $V'_n \supseteq V$ and $\sup_{u \in \ol{V'_n} \cap \Upsilon_\Gamma} \dpathY[\ol{V'_n}](u,V) \to 0$.
\end{lemma}
\begin{proof}
 Let $(V'_n)$ be a sequence as in the lemma statement. Since the set of vertices $\approxvtcs{\epsilon}$ is finite, we have $\approxvtcs[V'_n]{\epsilon} = \approxvtcs[V]{\epsilon}$ for large $n$. Also, the definition of $\metregions$ implies that for large $n$ the set of simple admissible paths between $v,w \in \approxvtcs[V]{\epsilon}$ is the same in $\ol{V'_n}$ as in $\ol{V}$. Therefore $\approxgraph[V'_n]{\epsilon} = \approxgraph[V]{\epsilon}$ and hence $\rmetapproxres{\epsilon}{V'_n}{x}{y}{\Gamma} = \rmetapproxres{\epsilon}{V}{x}{y}{\Gamma}$ for large $n$.
\end{proof}

\begin{lemma}[Monotonicity]
\label{le:approx_monotonicity}
Let $V,V' \in \metregions$, $V \subseteq V'$, and suppose that $V \setminus V_{(\rmvdeadends{\epsilon})} \subseteq V' \setminus V'_{(\rmvdeadends{\epsilon})}$. Let $x,y \in \ol{V} \cap \Upsilon_\Gamma$ and suppose one of the following:
\begin{enumerate}[(i)]
 \item\label{it:dead_ends} For every $u \in V' \setminus V$ there is a point $z \in \ol{V}$ that separates $u$ from $x,y$ in $\ol{V'} \cap \Upsilon_\Gamma$.
 \item\label{it:large_loops} There exist $z_1,z_2,\ldots \in \ol{V}$ with $\abs{z_i-z_{i'}} \ge 5\epsilon$ for $i\neq i'$ such that no \emph{simple} admissible path in $\ol{V}$ from $x$ to $y$ intersects $\{z_1,z_2,\ldots\}$, and the set $\{z_1,z_2,\ldots\}$ separates $x,y$ from $V' \setminus V$ in $\ol{V'} \cap \Upsilon_\Gamma$.
\end{enumerate}
Then there are points $x',y' \in \ol{V} \cap \Upsilon_\Gamma$ with $\dpathY[\ol{V}](x',x) \le \epsilon$, $\dpathY[\ol{V}](y',y) \le \epsilon$ such that
\[
 \rmetapproxres{\epsilon}{V'}{x'}{y'}{\Gamma} \le \rmetapproxres{\epsilon}{V}{x}{y}{\Gamma} .
\]
\end{lemma}

We remark that in the setting of Lemma~\ref{le:approx_monotonicity} we have that $x'=x$ (resp.\ $y'=y$) in case $\dpathY(x, V' \setminus V) \ge 6\epsilon$ (resp.\ $\dpathY(y, V' \setminus V) \ge 6\epsilon$).

\begin{proof}[Proof of Lemma~\ref{le:approx_monotonicity}]
By our construction, to determine the graph structure locally near a point, one only needs to observe the vertices and the gasket within $\dpathY[\ol{V}]$-distance $5\epsilon$. Assuming $V \setminus V_{(\rmvdeadends{\epsilon})} \subseteq V' \setminus V'_{(\rmvdeadends{\epsilon})}$, we have $\approxvtcs[V]{\epsilon} \subseteq \approxvtcs[V']{\epsilon}$, and therefore $\approxgraph[V]{\epsilon} \cap A = \approxgraph[V']{\epsilon} \cap A$ for every $A \subseteq \ol{V}$ with $\dpathY[\ol{V'}](A, V'\setminus V) \ge 5\epsilon$. However, the graph might change in the $5\epsilon$-neighborhood of $V' \setminus V$. This is why the extra restrictions~\eqref{it:dead_ends}, \eqref{it:large_loops} are needed.

Consider the connected component $K$ of $\ol{V'} \cap \Upsilon_\Gamma \setminus \{z_1,z_2,\ldots\}$ containing $x,y$ in the case~\eqref{it:large_loops} (resp.\ the connected component after removing all the points $z(u)$ for $u \in V' \setminus V$ in the case~\eqref{it:dead_ends}). We compare the graphs $\approxgraph[V]{\epsilon}$ and $\approxgraph[V']{\epsilon}$ restricted to $K$. We argue that for every pair $(v,w) \in \approxedges[V]{\epsilon}$ with $v,w \in K$, the corresponding cables do not change from $V$ to $V'$. Recall that our construction of the cable depends only on the gasket structure within $B(v,2\epsilon) \cup B(w,2\epsilon)$. The assumption~\eqref{it:dead_ends} (resp.\ \eqref{it:large_loops}) implies that any extra point $u \in (V' \setminus V) \cap (B(v,2\epsilon) \cup B(w,2\epsilon))$ is separated from $K$ in $B(v,2\epsilon) \cup B(w,2\epsilon)$ by a single point $z$. This does not affect our construction of the cables. By similar considerations, for each $(v,w) \in \approxedges[V]{\epsilon}$ for which the cable intersects $K$, that part of the cable does not change from $\approxgraph[V]{\epsilon}$ to $\approxgraph[V']{\epsilon}$. We conclude that $\approxgraph[V]{\epsilon} \cap K \subseteq \approxgraph[V']{\epsilon} \cap K$.

Finally, in $\approxgraph[V]{\epsilon}$, we assumed that no simple admissible path between $x,y$ can exit $K$, so that $\rmetapproxres{\epsilon}{V}{x}{y}{\Gamma}$ agrees with the resistance metric on the subgraph $\approxgraph[V]{\epsilon} \cap K$. Therefore we have $\rmetapproxres{\epsilon}{V}{\cdot}{\cdot}{\Gamma} \ge \rmetapproxres{\epsilon}{V'}{\cdot}{\cdot}{\Gamma}$ on $\approxgraph[V]{\epsilon} \cap K$.

The only possible concern is that when $x,y \notin \approxgraph[V]{\epsilon}$, then possibly their nearest interpolation points $x'$ (resp.\ $y'$) on $\approxgraph[V]{\epsilon}$ change from $\approxgraph[V]{\epsilon}$ to $\approxgraph[V']{\epsilon}$. In case $x',y' \in K$, we can switch from $x,y$ to $x',y'$ in the lemma statement. If $x'$ or $y'$ is outside $K$, this can only be when $x'$ (resp.\ $y'$) is not connected to $K$ in $\approxgraph[V]{\epsilon}$, otherwise there would be a cable going through the cut point. In case $x',y'$ are separated from $K$ by the same point $z$, then $x'=y'$. Otherwise, if either $y' \in K$ or $y'$ is separated from $K$ by $z' \neq z$, since we have assumed that no simple admissible path in $\ol{V}$ connects $z,z'$ outside $K$, we have $\rmetapproxres{\epsilon}{V}{x'}{y'}{\Gamma} = \infty \ge \rmetapproxres{\epsilon}{V'}{x'}{y'}{\Gamma}$.
\end{proof}

\begin{lemma}[Compatibility]
\label{le:approx_compatibility}
Let $V,V' \in \metregions$, $V \subseteq V'$, and suppose that $V \setminus V_{(\rmvdeadends{\epsilon})} \subseteq V' \setminus V'_{(\rmvdeadends{\epsilon})}$. Let $x,y \in \ol{V} \cap \Upsilon_\Gamma$ such that for every $u \in V' \setminus V$ there is a point $z \in \ol{V}$ with $\dpathY[\ol{V'}](z, V' \setminus V) \ge 6\epsilon$ that separates $u$ from $x,y$ in $\ol{V'} \cap \Upsilon_\Gamma$. Then $\rmetapproxres{\epsilon}{V}{x}{y}{\Gamma} = \rmetapproxres{\epsilon}{V'}{x}{y}{\Gamma}$.
\end{lemma}
\begin{proof}
Consider the connected component $K$ of $V' \cap \Upsilon_\Gamma$ containing $x,y$ after removing all the points $z(u)$ for $u \in V' \setminus V$ as in the proof of Lemma~\ref{le:approx_monotonicity}. Assuming $V \setminus V_{(\rmvdeadends{\epsilon})} \subseteq V' \setminus V'_{(\rmvdeadends{\epsilon})}$, we have $\approxvtcs[V]{\epsilon} \subseteq \approxvtcs[V']{\epsilon}$. Note that $\dpathY[\ol{V'}](K, V' \setminus V) \ge 3\epsilon$. Therefore, for each $(v,w) \in \approxedges[V]{\epsilon}$ for which the cable intersects $K$, the connected component of $(B(v,2\epsilon) \cup B(w,2\epsilon)) \cap \ol{V} \cap \Upsilon_\Gamma$ containing $v,w$ is the same as the connected component of $(B(v,2\epsilon) \cup B(w,2\epsilon)) \cap \ol{V'} \cap \Upsilon_\Gamma$ containing $v,w$. Therefore the cable is unchanged from $\approxgraph[V]{\epsilon}$ to $\approxgraph[V']{\epsilon}$, i.e.\ we have $\approxgraph[V]{\epsilon} \cap K = \approxgraph[V']{\epsilon} \cap K$. Since $V'$ only adds dead ends to $K$, the effective resistance on $K$ for $\approxgraph[V']{\epsilon}$ is the same as for $\approxgraph[V']{\epsilon} \cap K$.

We might again be concerned about the case when $x,y \notin \approxgraph[V]{\epsilon}$ and their interpolation points $x',y' \in \approxgraph[V]{\epsilon}$ are outside $K$. As we explained in the previous proof, this can only happen when $x'$ (resp.\ $y'$) is not connected to $K$ in $\approxgraph[V]{\epsilon}$. Since $\dpathY[\ol{V'}](K, V' \setminus V) \ge 6\epsilon$, the interpolation points $x',y'$ do not change from $\approxgraph[V]{\epsilon}$ to $\approxgraph[V']{\epsilon}$, so we either have $x'=y'$ or $\rmetapproxres{\epsilon}{V}{x}{y}{\Gamma} = \rmetapproxres{\epsilon}{V'}{x}{y}{\Gamma} = \infty$.
\end{proof}

\begin{lemma}[Series law]
\label{le:approx_series_law}
Let $V \in \metregions$, and $x,y,z \in \ol{V} \cap \Upsilon_\Gamma$ are such that $z$ separates $x$ from $y$ in $\ol{V} \cap \Upsilon_\Gamma$. Then
\[ \rmetapproxres{\epsilon}{V}{x}{y}{\Gamma} \ge \rmetapproxres{\epsilon}{V}{x}{z}{\Gamma}+\rmetapproxres{\epsilon}{V}{z}{y}{\Gamma} . \]
\end{lemma}
\begin{proof}
We only need to consider the case when $\rmetapproxres{\epsilon}{V}{x}{y}{\Gamma} < \infty$, in which case $z \in \approxgraph[V]{\epsilon}$. Therefore the result follows immediately from the series law for graphs.
\end{proof}

\begin{lemma}[Generalized parallel law]
\label{le:approx_parallel_law}
Let $V \in \metregions$, and let $x,y,z_1,\ldots,z_N \in \ol{V} \cap \Upsilon_\Gamma$ such that $x,y$ are separated in $\ol{V} \cap \Upsilon_\Gamma \setminus\{z_1,\ldots,z_N\}$. Then
\[ N\,\rmetapproxres{\epsilon}{V}{x}{y}{\Gamma} \ge \min_{i}\rmetapproxres{\epsilon}{V}{x}{z_i}{\Gamma}.\]
Further, if $K_x$ is the connected component of $\ol{V} \cap \Upsilon_\Gamma \setminus\{z_1,\ldots,z_N\}$ containing $x$, and $V_x \subseteq V$ is such that $\ol{V_x} \supseteq \{u \in V \setminus V_{(\rmvdeadends{\epsilon})} : \dpathY[\ol{V}](u,K_x) \le 5\epsilon\}$, then
\[ N\,\rmetapproxres{\epsilon}{V}{x}{y}{\Gamma} \ge \min_{i}\rmetapproxres{\epsilon}{V_x}{x}{z_i}{\Gamma}.\]
\end{lemma}

\begin{proof}
Suppose that $V_x$ is as in the lemma statement. Then every vertex $v \in \approxvtcs[V]{\epsilon}$ with $\dpathY[\ol{V}](v,K_x) \le 2\epsilon$ is also contained in $\approxvtcs[V_x]{\epsilon}$. Moreover, for each $(v,w) \in \approxedges[V]{\epsilon}$ for which the cable intersects $K_x$, the connected component of $(B(v,2\epsilon) \cup B(w,2\epsilon)) \cap \ol{V} \cap \Upsilon_\Gamma$ containing $v,w$ is the same as the connected component of $(B(v,2\epsilon) \cup B(w,2\epsilon)) \cap \ol{V_x} \cap \Upsilon_\Gamma$ containing $v,w$ except for possibly removing some dead ends in $V_{(\rmvdeadends{\epsilon})}$. However, the latter do not affect the shortest paths in $(B(v,2\epsilon) \cup B(w,2\epsilon)) \cap \ol{V} \cap \Upsilon_\Gamma$. Therefore $\approxgraph[V]{\epsilon} \cap K_x \subseteq \approxgraph[V_x]{\epsilon} \cap K_x$. Now the result follows by the generalized parallel law for graphs Lemma~\ref{le:gen_parallel_law_graph} applied to the set $\{z_1,\ldots,z_N\} \cap \approxgraph[V]{\epsilon}$ which separates $x$ from $y$ in $\approxgraph[V]{\epsilon}$.
\end{proof}

The family $(\rmetapproxres{\epsilon}{V}{\cdot}{\cdot}{\Gamma})_{V \in \metregions}$ does quite satisfy the axioms in Section~\ref{se:approx_scheme_general_def} due to the extra condition $V \setminus V_{(\rmvdeadends{\epsilon})} \subseteq V' \setminus V'_{(\rmvdeadends{\epsilon})}$ in the monotonicity (Lemma~\ref{le:approx_monotonicity}) and compatibility (Lemma~\ref{le:approx_compatibility}). However, this does not affect the proof of the tightness statement in Theorem~\ref{thm:tightness_theorem}\eqref{it:tightness} carried out in \cite{amy2025tightness} as there we only consider regions $V$ that consist of ``linear chains of bubbles'' in the sense that there are $x,y \in \partial V$ such that $V$ is the region bounded between two simple paths from $x$ to $y$. Moreover, we only need to consider the situation when $V' \supseteq V$ and $x,y \notin V'_{(\rmvdeadends{\epsilon})}$ (we have set $\rmetapproxres{\epsilon}{V'}{\cdot}{\cdot}{\Gamma}$ to $0$ in the components of $V'_{(\rmvdeadends{\epsilon})}$ anyway), in which case we always have $V \setminus V_{(\rmvdeadends{\epsilon})} \subseteq V' \setminus V'_{(\rmvdeadends{\epsilon})}$.

Note also that the condition for $V_x$ in Lemma~\ref{le:approx_parallel_law} is slightly relaxed, so that the statement of Lemma~\ref{le:approx_parallel_law} is slightly stronger than the generalized parallel law in Section~\ref{se:approx_scheme_general_def}. This formulation of Lemma~\ref{le:approx_parallel_law} will be important in Section~\ref{se:rmet_construction}.

We now explain that the approximation scheme $(\rmetapprox{\epsilon}{\cdot}{\cdot}{\Gamma})_{\epsilon > 0}$ satisfies the condition~\eqref{eq:eps_bound_ass}, \eqref{eq:approx_error_asymp} for a good approximation scheme as $\epsilon \to 0$. We start by giving in Lemma~\ref{le:approx_scale_ub} the asymptotic upper bound~\eqref{eq:approx_error_asymp} on the resistances between points with distance $\epsilon$.

In the statements below, we let $\metregions[(\epsilon)] \subseteq \metregions$ be the collection of $V \in \metregions$ such that $V \subseteq B(0,1-\epsilon)$ and $V$ is a union of connected components of $C \setminus (\CL_1 \cup \dots \cup \CL_n)$ where $\diamE(\CL_i) \ge \epsilon$ for each $i$.

\begin{lemma}
\label{le:vertex_density}
Fix $\rateexp > 0$ and suppose that the rate in~\eqref{eq:ppp_rate} is $\lambda_\epsilon = \epsilon^{-\dcle-\rateexp}$. Fix $a>0$ and let $E_\epsilon$ be the event that for each $V \in \metregions[(\epsilon)]$ with $\diamE(V) \ge \epsilon$ and $x \in \Upsilon_\Gamma \cap \ol{V} \setminus V_{(\rmvdeadends{\epsilon})}$ we have
\[ \epsilon^{-\rateexp+a} < \abs{\approxvtcs[V]{\epsilon} \cap \Bpath(x,\epsilon)} < \epsilon^{-\rateexp-a} . \]
Then $\p[E_\epsilon^c] = o^\infty(\epsilon)$ as $\epsilon \searrow 0$.
\end{lemma}

\begin{proof}
This follows from the Propositions~\ref{pr:measure_ub},~\ref{pr:measure_lb}, and the Poisson tails. Indeed, in order to have the lower bound, we need to argue that the set $\Bpath(x,\epsilon) \cap V \setminus V_{(\rmvdeadends{\epsilon})}$ satisfies the condition in Proposition~\ref{pr:measure_lb}. Since $\diamE(V) \ge \epsilon$, there is a point $w \in \partial\Bpath(x,\epsilon) \cap \ol{V} \setminus V_{(\rmvdeadends{\epsilon})}$. The set $\ol{V} \cap \Bpath(x,\epsilon)$ must contain the set $B_{x,w}$ in Proposition~\ref{pr:measure_lb} due to the definition of $\metregions[(\epsilon)]$. Further, since $x,w \notin V_{(\rmvdeadends{\epsilon})}$, each simple admissible path from $x$ to $w$ is also disjoint from $V_{(\rmvdeadends{\epsilon})}$, and since $\rmvdeadends{\epsilon} > \epsilon$, this implies that the interior of $B_{x,w}$ cannot contain any loop in $\Gamma_{V,\rmvdeadends{\epsilon}}$ and any region in $V_{(\rmvdeadends{\epsilon})}$ either.
\end{proof}

Recall that $\geoexp$ is defined in Section~\ref{se:geodesic_metric}.

\begin{lemma}\label{le:approx_scale_ub}
 Fix $\rateexp > 0$ and suppose that the rate in~\eqref{eq:ppp_rate} is $\lambda_\epsilon = \epsilon^{-\dcle-\rateexp}$. Fix $a>0$ and let $E_\epsilon$ be the event that for each $V \in \metregions[(\epsilon)]$ with $\diamE(V) \ge \epsilon$ we have
 \[ \sup_{\dpathY[\ol{V}](x,y) < \epsilon} \rmetapproxres{\epsilon}{V}{x}{y}{\Gamma} \le \epsilon^{\geoexp-a}.\]
 Then $\p[E_\epsilon^c] = o^\infty(\epsilon)$ as $\epsilon \searrow 0$.
\end{lemma}

\begin{proof}
 Since we have set $\rmetapproxres{\epsilon}{V}{\cdot}{\cdot}{\Gamma}$ to $0$ in the connected components of $V_{(\rmvdeadends{\epsilon})}$, we can assume $x,y \notin V_{(\rmvdeadends{\epsilon})}$.

 Let $E_{\epsilon,1}$ be the event from Lemma~\ref{le:vertex_density}. Then $\p[E_{\epsilon,1}^c] = o^\infty(\epsilon)$. Let $E_{\epsilon,2}$ be the event that $\metres{B(x,2\epsilon) \cup B(y,2\epsilon)}{x}{y}{\Gamma} \le \epsilon^{\geoexp-a}$ for each $x,y \in \Upsilon_\Gamma$ with $\dpathY(x,y) \le \epsilon$. By Proposition~\ref{pr:cle_metric_estimates} we have $\p[E_{\epsilon,2}^c] = o^\infty(\epsilon)$.

 Suppose that we are on the event $E_{\epsilon,1} \cap E_{\epsilon,2}$. Then we have $\dpathY[\ol{V}](x,\approxgraph[V]{\epsilon}) \le \dpathY[\ol{V}](x,\approxvtcs[V]{\epsilon}) < \epsilon$ for each $x \in \Upsilon_\Gamma \cap \ol{V} \setminus V_{(\rmvdeadends{\epsilon})}$. In particular, the graph $\approxgraph[V]{\epsilon}$ is connected and we have $\rmetapproxres{\epsilon}{V}{x}{y}{\Gamma} \le \metres{B(x,2\epsilon) \cup B(y,2\epsilon)}{x}{y}{\Gamma} < \epsilon^{\geoexp-a}$ for each $(x,y) \in \approxedges[V]{\epsilon}$. Combining everything implies the result.
\end{proof}

For the next lemmas, we consider the setup described in the paragraph above~\eqref{eq:median_def}. The following lemma gives~\eqref{eq:eps_bound_ass}. In the notation of Section~\ref{se:approx_scheme_general_def}, if we set $\ac{\epsilon} = \epsilon^{\geoexp-o(1)}$, then $\ac{\epsilon} = \epsilon^{\ddouble-\rateexp-o(1)}\median{\epsilon}$ as $\epsilon \searrow 0$.

\begin{lemma}\label{le:approx_scale_lb}
 There is a constant $c>0$ such that the following holds. Fix $\rateexp > 0$ and suppose that the rate in~\eqref{eq:ppp_rate} is $\lambda_\epsilon = \epsilon^{-\dcle-\rateexp}$. Consider the setup described in the paragraph above. Fix $a>0$. Then
 \[
  \p\left[ \sup_{\substack{x=\eta_1(s)=\eta_2(s')\\ y=\eta_1(t)=\eta_2(t')\\ \eta_1[s,t], \eta_2[s',t'] \subseteq B(0,3/4)}} \rmetapproxres{\epsilon}{U}{x}{y}{\Gamma} < \epsilon^{\geoexp-\ddouble+c\rateexp+a} ,\ I \right] \to 0
  \quad\text{as } \epsilon \searrow 0 .
 \]
 In particular, we have $\median{\epsilon} \ge \epsilon^{\geoexp-\ddouble+c\rateexp+o(1)}$ as $\epsilon \searrow 0$.
\end{lemma}

\begin{proof}
Let $\epsilon' = \epsilon^{1+\rateexp/\dcle} < \epsilon$. Note that if we rescale by $\epsilon'$, the intensity measure $\epsilon^{-\dcle-\rateexp}\meas{\cdot}{\epsilon'\Gamma}$ is the pushforward of $\meas{\cdot}{\Gamma}$. For $\epsilon'' < \epsilon'$, the intensity measure $\epsilon^{-\dcle-\rateexp}\meas{\cdot}{\epsilon''\Gamma}$ is smaller than the pushforward of $\meas{\cdot}{\Gamma}$.

Let $M>1$ and let $E_{\epsilon,1}$ be the event that the following holds. For each $z \in (\epsilon')^{1+a}\Z^2 \cap B(0,3/4)$ there exists $r \in [(\epsilon')^{1+a},\epsilon']$ such that the following holds.
\begin{itemize}
 \item There is a collection of at most $M$ points $\{u_i\}$ in $A(z,r,2r)$ that separate $\partial B(z,r)$ from $\partial B(z,2r)$ in $\Upsilon_\Gamma$.
 \item For each $u_i$, the two strands intersecting at $u_i$ intersect at a point $v_i$ near $u_i$ so that if $V_i$ is the region bounded between the two strands from $u_i$ to $v_i$, then $\abs{\approxvtcs{\epsilon} \cap V_i} \le M$ and $\rmetapprox{\epsilon}{u_i}{v_i}{\Gamma} \ge M^{-1}r^{\geoexp}$.
\end{itemize}
The observation in the first paragraph and the independence across scales (see e.g.\ the proof of Proposition~\ref{pr:measure_lb}) imply that if $M$ is sufficiently large (depending on $a$), then $\p[E_{\epsilon,1}^c] \le \epsilon^{100}$.

Let $E_{\epsilon,2}$ be the event that there are at least $(\epsilon')^{-\ddouble+a}$ intersection points in $\eta_1 \cap \eta_2 \cap B(0,1/2)$ with Euclidean distance at least $\epsilon'$ from each other. By~\cite[Theorem~1.5]{mw2017intersections} we have $\p[E_{\epsilon,2}^c \cap I] \to 0$ as $\epsilon \searrow 0$.

Suppose that we are on the event $E_{\epsilon,1} \cap E_{\epsilon,2}$. Then each of the intersection points in $\eta_1 \cap \eta_2 \cap B(0,1/2)$ is contained in an annulus $A(z',r,2r)$ as in the event $E_{\epsilon,1}$. If $x$ (resp.\ $y$) is the first (resp.\ last) intersection point, then by the parallel law and the serial law for effective resistances
\[
 \rmetapprox{\epsilon}{x}{y}{\Gamma} \ge (\epsilon')^{-\ddouble+a} M^{-2}r^{\geoexp} \ge M^{-2} \epsilon^{\geoexp-\ddouble+c\rateexp+ca} .
\]
\end{proof}

\begin{proof}[Proof of Proposition~\ref{pr:rmet_tightness}]
 As we have discussed above Lemma~\ref{le:vertex_density}, the proof of~\cite[Proposition~5.17]{amy2025tightness} works verbatim in our setup. Since $V \in \metregions$ is chosen from the same probability law for each $\epsilon$, we a.s.\ have $V \in \metregions[(\epsilon)]$ for sufficiently small $\epsilon$. Therefore Lemma~\ref{le:approx_scale_ub} applies and this gives the analogue of~\cite[Corollary~5.19]{amy2025tightness} which implies Proposition~\ref{pr:rmet_tightness}.
\end{proof}

\begin{figure}[ht]
\centering
\includegraphics[width=0.4\textwidth,page=1]{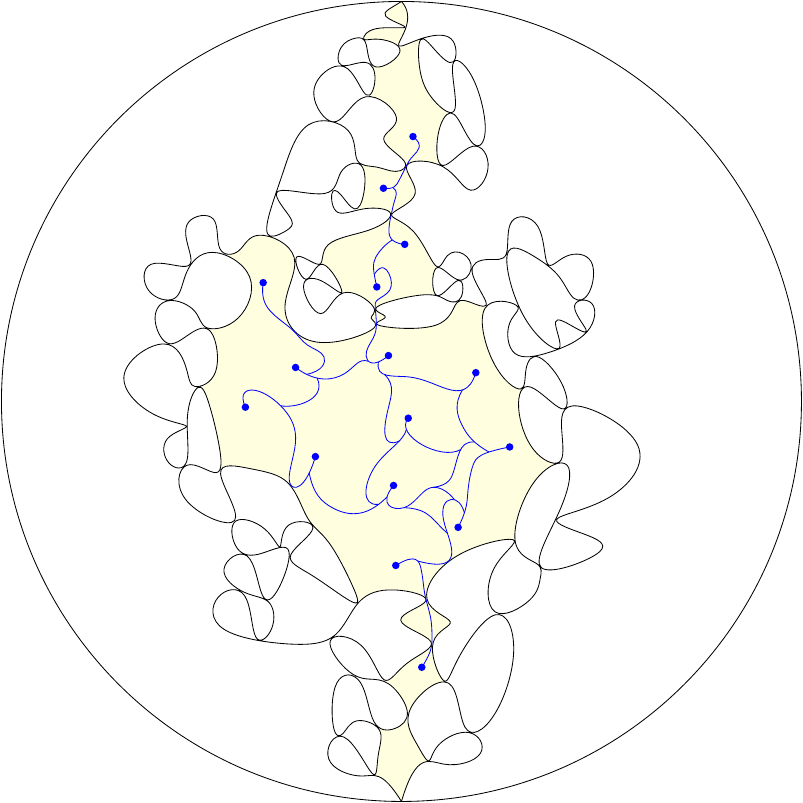}\hspace{0.03\textwidth}\includegraphics[width=0.4\textwidth,page=2]{dead_ends_median.pdf}\hspace{0.03\textwidth}\includegraphics[width=0.4\textwidth,page=3]{dead_ends_median.pdf}
\caption{\textbf{Top left:} The region $U$ between the outer boundaries of $\eta'_1$ and $\eta'_2$ is shown in yellow. This region is used to define the renormalization constant $\median{\epsilon}$. \textbf{Top right:} The additional dead ends in the region $U'$ between $\eta'_1$ and $\eta'_2$ are shown in grey. Note that the cables emanating from the vertices in $U' \setminus U$ create additional edges between the vertices in $U$. \textbf{Bottom:} The dead ends of size at most $\rmvdeadends{\epsilon}$ are shown in pink. Removing them decreases the number of additional edges.}
\label{fi:dead_ends_median}
\end{figure}

We have now shown that the family $(\median{\epsilon}^{-1}\rmetapprox{\epsilon}{\cdot}{\cdot}{\Gamma})_{\epsilon > 0}$ has subsequential limits as $\epsilon \searrow 0$, and by Theorem~\ref{thm:tightness_theorem}\eqref{it:limit_cle_met} (again with some slight modifications which we will explain in Section~\ref{se:rmet_construction} below), each subsequential limit gives rise to a \clekp{} metric. By~\cite[Theorem~1.14]{amy2025tightness}, each subsequential limit is either identically zero or a family of true metrics. However, it is not guaranteed that it is not identically zero. The reason is that the renormalization factors $\median{\epsilon}$ (see Section~\ref{se:approx_scheme_general_def}) are defined via $\rmetapproxres{\epsilon}{U}{\cdot}{\cdot}{\Gamma}$ where $U$ is the region between the \emph{outer boundaries} of two intersecting \clekp{} loops. However, we need that the limits of $\median{\epsilon}^{-1}\rmetapproxres{\epsilon}{U'}{\cdot}{\cdot}{\Gamma}$ to be non-trivial where $U' \supseteq U$ contains additional dead ends that are connected to $U$ in the gasket. In general $\rmetapproxres{\epsilon}{U'}{\cdot}{\cdot}{\Gamma}$ is smaller than $\rmetapproxres{\epsilon}{U}{\cdot}{\cdot}{\Gamma}$ since $\approxgraph[U']{\epsilon}$ contains additional cables emanating from $U' \setminus U$, thus adding extra connections between distinct vertices in $U$ (see Figure~\ref{fi:dead_ends_median}). We need an additional argument to rule out that it becomes significantly smaller.

Consider the setup described in the paragraph above~\eqref{eq:median_def}. Let $\eta'_1$ (resp.\ $\eta'_2$) be the counterflow line with angle $\angledouble+\pi/2$ (resp.\ $-\pi/2$) of $h$ from $i$ to $-i$. Then the right boundary of $\eta'_1$ agrees with $\eta_1$, and the left boundary of $\eta'_2$ agrees with $\eta_2$. Let $U' \supseteq U$ be the collection of regions in each component to the right of the segments of $\eta'_1$ and to the left of the segments of $\eta'_2$. Let $\rmvdeadends{\epsilon}$ be as in~\eqref{eq:rmvdeadends_size}, and let $\wt{U}'_{\rmvdeadends{\epsilon}} \subseteq U'$ be obtained by removing the dead ends of Euclidean diameter at most $\rmvdeadends{\epsilon}$, i.e.\ the points $z$ such that every path from $z$ not crossing $\eta'_1 \cup \eta_1$ (resp.\ $\eta'_2 \cup \eta_2$) has diameter at most $\rmvdeadends{\epsilon}$.

\begin{lemma}\label{le:nondegeneracy}
Consider the setup in the paragraph above. If $\rateexp, a_0 > 0$ in~\eqref{eq:ppp_rate},~\eqref{eq:rmvdeadends_size} are chosen small enough, then
 \[
  \p\left[ \sup_{\substack{x=\eta_1(s)=\eta_2(s')\\ y=\eta_1(t)=\eta_2(t')\\ \eta_1[s,t], \eta_2[s',t'] \subseteq B(0,3/4)}} \rmetapproxres{\epsilon}{\wt{U}'_{\rmvdeadends{\epsilon}}}{x}{y}{\Gamma} < \frac{99}{100} \sup_{\substack{x=\eta_1(s)=\eta_2(s')\\ y=\eta_1(t)=\eta_2(t')\\ \eta_1[s,t], \eta_2[s',t'] \subseteq B(0,3/4)}} \rmetapproxres{\epsilon}{U}{x}{y}{\Gamma} ,\ I \right] \to 0
  \quad\text{as } \epsilon \searrow 0 .
 \]
\end{lemma}

\begin{proof}
We are going to show that
\begin{equation}\label{eq:dead_ends_influence}
 \rmetapproxres{\epsilon}{U}{\cdot}{\cdot}{\Gamma} - \rmetapproxres{\epsilon}{\wt{U}'_{\rmvdeadends{\epsilon}}}{\cdot}{\cdot}{\Gamma}
 \le \epsilon^{\geoexp-2\rateexp-2a_0-o(1)}
\end{equation}
on an event whose probability tends to $1$ as $\epsilon \searrow 0$. This will imply the result due to Lemma~\ref{le:approx_scale_lb}.

Let $a>0$ be a small constant. Note that $\wt{U}'_{\rmvdeadends{\epsilon}} \setminus U$ consists of dead ends of Euclidean diameter at least $\rmvdeadends{\epsilon}$. By \cite[Lemma~2.17]{amy2025tightness}, the probability that there are more than $\rmvdeadends{\epsilon}^{-2-a} = \epsilon^{-2a_0-a}$ of them is $o^\infty(\epsilon)$. Let $V$ be one of them and let $z \in \partial U$ be the point where it is attached to $U$. By our construction of $\approxgraph[\wt{U}'_{\rmvdeadends{\epsilon}}]{\epsilon}$, adding $V$ only adds some cables within $\Bpath(z,2\epsilon)$. Then we can lower bound $\rmetapproxres{\epsilon}{U \cup V}{\cdot}{\cdot}{\Gamma}$ by the effective resistance of the network obtained by contracting $\Bpath(z,2\epsilon)$ to a single vertex. Suppose further that we are on the events from Lemma~\ref{le:vertex_density} and~\ref{le:approx_scale_ub}. Then the total length of all the contracted edges is at most $\epsilon^{\geoexp-2\rateexp-a}$. Therefore, by Lemma~\ref{le:reff_contraction},
\[
 \rmetapproxres{\epsilon}{U}{\cdot}{\cdot}{\Gamma} - \rmetapproxres{\epsilon}{U \cup V}{\cdot}{\cdot}{\Gamma}
 \le \epsilon^{\geoexp-2\rateexp-a} .
\]
Repeating this for every dead end $V$ in $\wt{U}'_{\rmvdeadends{\epsilon}} \setminus U$ gives~\eqref{eq:dead_ends_influence}.
\end{proof}

\subsection{Construction of weak \clekp{} resistance forms}
\label{se:rmet_construction}

Our next goal is to show that each converging subsequence of the $(\rmetapprox{\epsilon}{\cdot}{\cdot}{\Gamma})_{\epsilon > 0}$ gives rise to a weak \clekp{} resistance form. By Theorem~\ref{thm:tightness_theorem}\eqref{it:limit_cle_met} each subsequential limit is associated with a \clekp{} metric in the sense defined in Section~\ref{se:approx_scheme_general_def}. We need to review the construction of the \clekp{} metric from the subsequential limits given in~\cite[Section~6]{amy2025tightness}.

\begin{figure}[ht]
\centering
\includegraphics[width=0.5\textwidth]{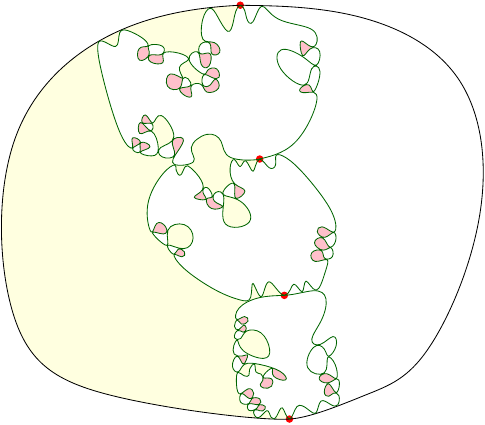}
\caption{Shown is a region $V$ and some of the loops in $\ol{V}$. The small dead ends in $V_{(\rmvdeadends{\epsilon})}$ cut off by these loops are shown in pink. When we divide $V$ into subcomponents (one is shown in yellow), we need to remove the dead ends also from the subcomponents in order to make their graph approximations match.}
\label{fi:rmet_constr_dead_ends}
\end{figure}

We will however need a slight modification of the construction. Recall that we have proved the non-degeneracy in Lemma~\ref{le:nondegeneracy} only when the dead ends of size at most $\rmvdeadends{\epsilon}$ are removed in the approximation, therefore we need to make this modification in the construction of $\rmet{\cdot}{\cdot}{\Gamma}$ below. This is also the reason why need to remove the dead ends $V_{(\rmvdeadends{\epsilon})}$ in the approximations given in Section~\ref{se:rmet_approx_scheme}. Indeed, if we split up a component of $V$ into multiple components $V_i \in \metregions$, we need to remove the same dead ends in the approximations of $\rmetres{V}{\cdot}{\cdot}{\Gamma}$ as in $\rmetres{V_i}{\cdot}{\cdot}{\Gamma}$ in order to make them compatible (see Figure~\ref{fi:rmet_constr_dead_ends}).

Consider a countable ``dense'' collection $\FQ$ of Jordan domains (e.g.\ the polygonal regions with rational vertices). For each $Q \in \FQ$ and $\epsilon > 0$, define $\wt{Q}_\epsilon$ such that
\begin{itemize}
 \item $\wt{Q}_\epsilon$ is a disjoint union of regions in $\metregions[Q]$,
 \item $\wt{Q}_\epsilon$ is separated from $\Upsilon_\Gamma \setminus \ol{\wt{Q}_\epsilon}$ by a set of points $\{z_i\}$ with $\abs{z_i-z_{i'}} \ge 5\epsilon$ for each $i \neq i'$,
 \item for each $Q_1 \Subset Q$, the set $Q_1 \cap \Upsilon_\Gamma$ is contained in $\ol{\wt{Q}_\epsilon}$ for small enough $\epsilon$.
\end{itemize}

We now define the following modification. Let $\rmvdeadends{\epsilon}$ be as in~\eqref{eq:rmvdeadends_size}. Let $\wt{Q}_{\epsilon,\rmvdeadends{\epsilon}} \subseteq \wt{Q}_\epsilon$ be obtained by removing the points $u \in \wt{Q}_\epsilon$ for which there exists $z \in \ol{\wt{Q}_\epsilon}$ such that $u$ is in a component of $\ol{\wt{Q}_\epsilon} \setminus \{z\}$ with Euclidean diameter at most $\rmvdeadends{\epsilon}$. The construction implies the following important property.

\begin{lemma}\label{le:modified_regions_monotone}
 If $Q,Q' \in \FQ$ and $\wt{Q}_\epsilon \subseteq \wt{Q}'_\epsilon$, then $\wt{Q}_{\epsilon,\rmvdeadends{\epsilon}} \subseteq \wt{Q}'_{\epsilon,\rmvdeadends{\epsilon}}$ and $\wt{Q}_{\epsilon,\rmvdeadends{\epsilon}} \setminus (\wt{Q}_{\epsilon,\rmvdeadends{\epsilon}})_{(\rmvdeadends{\epsilon})} \subseteq \wt{Q}'_{\epsilon,\rmvdeadends{\epsilon}} \setminus (\wt{Q}'_{\epsilon,\rmvdeadends{\epsilon}})_{(\rmvdeadends{\epsilon})}$.
\end{lemma}

We will also need the following.

\begin{lemma}\label{le:neighborhood_compatible}
 Let $V \in \metregions$ and $Q,Q' \in \FQ$ such that $V \Subset Q \Subset Q'$, and suppose that for each $u \in Q' \cap \Upsilon_\Gamma \setminus \ol{V}$ there is $z \in \partial V$ that separates $u$ from $V$ in $Q' \cap \Upsilon_\Gamma$. Let $r < \distE(V,\partial Q)$ and let $V_r$ be the Euclidean $r$-neighborhood of $V$. Then $\wt{Q}_{\epsilon,\rmvdeadends{\epsilon}} \cap V_r = \wt{Q}'_{\epsilon,\rmvdeadends{\epsilon}} \cap V_r$ for sufficiently small $\epsilon$.
\end{lemma}

\begin{proof}
 We clearly have $\wt{Q}_\epsilon \subseteq \wt{Q}'_\epsilon$ for sufficiently small $\epsilon$, and the additional condition implies that $\wt{Q}_\epsilon \cap V_r = \wt{Q}'_\epsilon \cap V_r$. We need to argue that $V_r \cap \wt{Q}_\epsilon \setminus \wt{Q}_{\epsilon,\rmvdeadends{\epsilon}} = V_r \cap \wt{Q}'_\epsilon \setminus \wt{Q}'_{\epsilon,\rmvdeadends{\epsilon}}$. Clearly, we have $V_r \cap \wt{Q}'_\epsilon \setminus \wt{Q}'_{\epsilon,\rmvdeadends{\epsilon}} \subseteq V_r \cap \wt{Q}_\epsilon \setminus \wt{Q}_{\epsilon,\rmvdeadends{\epsilon}}$ since dead ends in $\wt{Q}'_\epsilon$ are also dead ends in $\wt{Q}_\epsilon \subseteq \wt{Q}'_\epsilon$. Conversely, if $u \in V_r \cap \wt{Q}_\epsilon \setminus \wt{Q}_{\epsilon,\rmvdeadends{\epsilon}}$, it is contained in a dead end contained in $V_{r+\rmvdeadends{\epsilon}}$. Therefore, if $\epsilon$ is sufficiently small so that $V_{r+\rmvdeadends{\epsilon}} \cap \Upsilon_\Gamma \subseteq \ol{\wt{Q}_\epsilon}$, the dead end does not touch $\wt{Q}'_\epsilon \setminus \wt{Q}_\epsilon$, hence is also contained in $\wt{Q}'_\epsilon \setminus \wt{Q}'_{\epsilon,\rmvdeadends{\epsilon}}$.
\end{proof}

Let $R^Q_\epsilon = \median{\epsilon}^{-1}\rmetapproxres{\epsilon}{\wt{Q}_{\epsilon,\rmvdeadends{\epsilon}}}{\cdot}{\cdot}{\Gamma}$. By Proposition~\ref{pr:rmet_tightness}, every sequence of $\epsilon \searrow 0$ contains a subsequence $(\epsilon_n)$ such that the joint law of $(\Upsilon_\Gamma, (R^Q_{\epsilon_n})_{Q \in \FQ})$ converges (where the convergence of $R^Q_{\epsilon_n}$ is in the sense that its restriction to each fixed $\wt{Q}_{\epsilon_0,\rmvdeadends{\epsilon_0}}$ converges in the topology defined in Section~\ref{subsec:tightness_setup}). Since $\lim_{\epsilon \to 0} \rmvdeadends{\epsilon} = 0$, the limit $R^Q$ is a random function defined on $Q \cap \Upsilon_\Gamma$.

Suppose that $(\Upsilon_\Gamma, (R^Q)_{Q \in \FQ})$ has the limiting law. We now explain carefully that the construction in \cite[Section~6]{amy2025tightness} still gives us a \clekp{} metric, despite the modified construction and the additional condition required in the Lemmas~\ref{le:approx_monotonicity} and~\ref{le:approx_compatibility}. First, thanks to Lemma~\ref{le:modified_regions_monotone}, the analogue of \cite[Lemma~6.1(ii)]{amy2025tightness} still holds. That is, if $Q \Subset Q'$, then $R^{Q'} \le R^Q$ on $Q \cap \Upsilon_\Gamma$. Therefore we can define
\[
 \rmetres{V}{x}{y}{\Gamma} = \max_{Q \in \FQ,\, Q \Supset V} R^Q(x,y) ,\quad x,y \in \ol{V} \cap \Upsilon_\Gamma
\]
for each $V \in \metregions$. Moreover, thanks to Lemma~\ref{le:neighborhood_compatible}, the maximum above is attained and we have $\rmetres{V}{x}{y}{\Gamma} = R^Q(x,y)$ on $\ol{V} \cap \Upsilon_\Gamma$ when $Q \Supset V$ is contained in a sufficiently small neighborhood of $\ol{V}$.

We now explain that \cite[Theorem~6.2]{amy2025tightness} also holds in the present case, i.e.\ $(\rmetres{V}{\cdot}{\cdot}{\Gamma})_{V \in \metregions}$ satisfies the definition of a \clekp{} metric given in Section~\ref{se:approx_scheme_general_def} with $\epsilon=0$. The proof of the compatibility \cite[Proposition~6.4]{amy2025tightness} goes through thanks to Lemma~\ref{le:neighborhood_compatible}. The remaining proofs stay unchanged, only the proof of the generalized parallel law \cite[Proposition~6.13]{amy2025tightness} requires an additional argument. Let $x,y,z_1,\ldots,z_N \in \ol{V} \cap \Upsilon_\Gamma$ be such that $x,y$ are separated in $\ol{V} \cap \Upsilon_\Gamma \setminus\{z_1,\ldots,z_N\}$, and let $V_x \in \metregions$ be the minimal region such that $\ol{V_x}$ contains the connected component of $\ol{V} \cap \Upsilon_\Gamma \setminus \{z_1,\ldots,z_N\}$ containing $x$. Let $Q \in \FQ$, $Q \Supset V$, and let $\ol{\wt{V}_x}$ contain a small neighborhood of the connected component of $Q \cap \Upsilon_\Gamma \setminus \{z_1,\ldots,z_N\}$ containing $x$. Let $Q_x \in \FQ$ be such that $Q_x \Supset \wt{V}_x$. In general, $(\wt{Q}_x)_{\epsilon,\rmvdeadends{\epsilon}}$ will not contain $\wt{Q}_{\epsilon,\rmvdeadends{\epsilon}} \cap \wt{V}_x$ since there are components of $(\wt{Q}_x)_\epsilon \setminus (\wt{Q}_x)_{\epsilon,\rmvdeadends{\epsilon}}$ bounded by the loops $\CL_1,\ldots,\CL_m \subseteq \ol{V}$ intersecting $\{z_1,\ldots,z_N\}$. However, if $\rmvdeadends{\epsilon} < \min_i \diamE(\CL_i)$, then $\CL_1,\ldots,\CL_m \in \Gamma_{V,\rmvdeadends{\epsilon}}$. Therefore the components of $(\wt{Q}_x)_\epsilon \setminus (\wt{Q}_x)_{\epsilon,\rmvdeadends{\epsilon}}$ bounded by each $\CL_i$ are contained in $V_{(\rmvdeadends{\epsilon})}$, hence Lemma~\ref{le:approx_parallel_law} still applies to the approximations to $\rmetres{V}{\cdot}{\cdot}{\Gamma}$ and $\rmetres{V_x}{\cdot}{\cdot}{\Gamma}$.

Now, we conclude by~\cite[Proposition~6.14]{amy2025tightness} that each $\rmetres{V}{\cdot}{\cdot}{\Gamma}$ is a continuous (with respect to $\dpathY[\ol{V}]$) metric on $\ol{V} \cap \Upsilon_\Gamma$ provided they are not all identically zero (which we have verified in Lemma~\ref{le:nondegeneracy}, noting that the setup there is absolutely continuous with respect to the setup in Section~\ref{subsec:tightness_setup} as explained in~\cite[Section~3.1]{amy2025tightness}).

The remainder of the subsection is dedicated to showing the following result.

\begin{proposition}\label{pr:lim_weak_rform}
 Let $(\rmetres{V}{\cdot}{\cdot}{\Gamma})_{V \in \metregions}$ be the collection of metrics constructed just above (for some subsequence $(\epsilon_n)$). Then we have for each $V \in \metregions$ that $\rmetres{V}{\cdot}{\cdot}{\Gamma}$ is a resistance metric on~$\ol{V} \cap \Upsilon_\Gamma$. The associated family of resistance forms $(\rformres{V}{\cdot}{\cdot}{\Gamma})_{V \in \metregions}$ is a weak \clekp{} resistance form.
\end{proposition}

In the remainder of this subsection, we assume that the random vector $(\Upsilon_\Gamma, (R^Q)_{Q \in \FQ})$ is coupled with the sequence $(\Upsilon_{\Gamma_n}, (R^Q_{\epsilon_n})_{Q \in \FQ})$ so that it converges almost surely. Whenever we consider the GHf convergence $R^Q_{\epsilon_n} \to R^Q$, we implicitly assume that they are isometrically embedded in a common compact metric space (cf.\ \cite[Lemma~A.2]{amy2025tightness}). To ease notation, we will just write $\epsilon$ instead of $\epsilon_n$.

\begin{lemma}\label{le:lim_resistance_metric}
 Almost surely, for each $V \in \metregions$ we have that $\rmetres{V}{\cdot}{\cdot}{\Gamma}$ is a (non-degenerate) resistance metric on~$\ol{V} \cap \Upsilon_\Gamma$.
\end{lemma}

\begin{proof}
By Lemma~\ref{le:nondegeneracy} and~\cite[Theorem~1.14]{amy2025tightness}, each $\rmetres{V}{\cdot}{\cdot}{\Gamma}$ is a true metric. We need to argue that for each finite subset $A \subseteq \ol{V} \cap \Upsilon_\Gamma$ there is a weight function $w\colon A \times A \to [0,\infty)$ such that $\rmetres{V}{\cdot}{\cdot}{\Gamma}|_{A \times A}$ is the effective resistance metric associated with $(A,w)$.

Let $Q \Supset V$ be such that $\rmetres{V}{\cdot}{\cdot}{\Gamma} = R^Q$ on $\ol{V} \cap \Upsilon_\Gamma$. By the GHf convergence $R^Q_\epsilon \to R^Q$ and the continuity of $R^Q$, there is for each $x \in A$ a sequence $(x_n)$ in $\Upsilon_{\Gamma_n}$ with $x_n \to x$ such that $R^Q_\epsilon(x_n,y_n) \to R^Q(x,y)$ for each $x,y \in \ol{V} \cap \Upsilon_\Gamma$. Let $\approxgraph[\wt{Q}_{\epsilon,\rmvdeadends{\epsilon}}]{\epsilon}$ be the cable graph from the construction of the approximation scheme. Recall from Lemma~\ref{le:vertex_density} that for each sufficiently small $\epsilon$, for each $x \in \ol{V} \cap \Upsilon_{\Gamma_n}$ there is some $x_\epsilon \in \approxgraph[\wt{Q}_{\epsilon,\rmvdeadends{\epsilon}}]{\epsilon}$ with $\dpathY[\ol{\wt{Q}}_\epsilon](x,x_\epsilon) < \epsilon$. We have $\rmetapproxres{\epsilon}{\wt{Q}_{\epsilon,\rmvdeadends{\epsilon}}}{x}{y}{\Gamma} = \rmetapproxres{\epsilon}{\wt{Q}_{\epsilon,\rmvdeadends{\epsilon}}}{x_\epsilon}{y_\epsilon}{\Gamma}$ by our construction, and the latter is a resistance metric on $\approxgraph[\wt{Q}_{\epsilon,\rmvdeadends{\epsilon}}]{\epsilon}$. Let $w_{A_\epsilon}$ be the weight function associated with $R^Q_\epsilon|_{A_\epsilon \times A_\epsilon}$ where $A_\epsilon = \{ (x_n)_\epsilon \mid x \in A \}$. The convergence $R^Q_\epsilon \to R^Q$ implies that $R^Q_\epsilon|_{A_\epsilon \times A_\epsilon} \to R^Q|_{A \times A}$. By Lemma~\ref{le:reff_lim}, we have that $w_{A_\epsilon}$ converge to a weight function $w$ and its associated effective resistance metric is $R^Q|_{A \times A}$.
\end{proof}

Note that since $\rmetres{V}{\cdot}{\cdot}{\Gamma}$ is continuous with respect to $\dpathY[\ol{V}]$, by \cite[Lemma~1.10]{amy2025tightness}, the resistance metric space $(\ol{V} \cap \Upsilon_\Gamma, \rmetres{V}{\cdot}{\cdot}{\Gamma})$ is compact. Let $(\rformres{V}{\cdot}{\cdot}{\Gamma}, \rfdomainres{V}{\Gamma})$ be the associated resistance form. We now show that the family $(\rformres{V}{\cdot}{\cdot}{\Gamma})_{V \in \metregions}$ satisfies the axioms in the definition of a weak \clekp{} resistance form by verifying the conditions in Proposition~\ref{pr:cle_rmet_char}. We have already verified~\eqref{it:rmet_cont}--\eqref{it:rmet_compatibility}. Property~\eqref{it:rmet_cut_point} follows by \cite[Proposition~6.4 and~6.12]{amy2025tightness} which we have explained above to hold also in our construction. It remains to verify~\eqref{it:rmet_cut_loop}.

\begin{lemma}\label{le:internal_rmetrics}
 Almost surely, there exists a countable set $\{z_m\}_{m \in \N} \subseteq \Upsilon_\Gamma$ that is dense with respect to $\dpathY$ and such that the following holds. Suppose that $V \in \metregions$ is separated from $\Upsilon_\Gamma \setminus \ol{V}$ by a finite set of points $Z$ contained in $\{z_m\}_{m \in \N}$. Then the following holds for sufficiently large $m \in \N$. Let $A = \{z_1,\ldots,z_m\}$ and $A_V = A \cap \ol{V}$, let $w$ (resp.\ $w_V$) be the weight function associated with $\rmet{\cdot}{\cdot}{\Gamma}|_{A}$ (resp.\ $\rmetres{V}{\cdot}{\cdot}{\Gamma}|_{A_V}$). Then
 \begin{itemize}
  \item $w(x,y) = w_V(x,y)$ for each $x,y \in A_V$,
  \item $w(x,y) = 0$ for each $x \in A_V \setminus Z$, $y \in A \setminus \ol{V}$,
  \item $w(x,y) = 0$ for $x,y \in Z$.
 \end{itemize}
\end{lemma}

\begin{proof}
 Consider the countable collection $\CK$ of loop segments $\ell \subseteq \CL$ starting and ending at rational times of all $\CL \in \Gamma$. For each non-overlapping, intersecting pair $\ell,\ell' \in \CK$ the intersection set $\ell \cap \ell'$ is compact. Let $\{z_m\}$ be any countable dense set containing a dense set of $\ell \cap \ell'$ for each such pair $\ell,\ell' \in \CK$. We argue that such a set satisfies the properties stated in the lemma.

 Suppose that $V$ is as given in the lemma statement, and $m$ is large enough so that the set $Z$ is contained in $A = \{z_1,\ldots,z_m\}$. We now use the same setup and notation as in the proof of Lemma~\ref{le:lim_resistance_metric}. Let $Q \Supset V$ be such that $\rmetres{V}{\cdot}{\cdot}{\Gamma} = R^Q$ on $\ol{V} \cap \Upsilon_\Gamma$. We can assume that distinct points of $Z$ are not connected in $\Upsilon_\Gamma \cap Q$ outside of $\ol{V}$. For each $x \in A$, let $(x_n)$ in $\Upsilon_{\Gamma_n}$ be as in the proof of Lemma~\ref{le:lim_resistance_metric}. Let $A_\epsilon = \{ (x_n)_\epsilon \mid x \in A \}$, $A_{V,\epsilon} = \{ (x_n)_\epsilon \mid x \in A_V \}$, and $Z_\epsilon = \{ x_n \mid x \in Z \}$. Recalling that the topology of convergence (see Section~\ref{subsec:tightness_setup}) keeps track of the finite separation sets, we can choose the points $(x_n)$ for $x \in Z$ such that for each $\epsilon$, the set $Z_\epsilon$ separates $A_{V,\epsilon}$ from $A_\epsilon \setminus A_{V,\epsilon}$ in $\Upsilon_{\Gamma_n}$. By our construction of the cable graph $\approxgraph[\wt{Q}_{\epsilon,\rmvdeadends{\epsilon}}]{\epsilon}$ and Lemma~\ref{le:vertex_density}, when $\epsilon$ is sufficiently small, the set $Z_\epsilon$ lies on $\approxgraph[\wt{Q}_{\epsilon,\rmvdeadends{\epsilon}}]{\epsilon}$. The set $Z_\epsilon$ also separates $A_{V,\epsilon}$ from $A_\epsilon \setminus A_{V,\epsilon}$ in $\approxgraph{\epsilon}$. Assuming $\rmvdeadends{\epsilon}$ is smaller than the diameters of the loops bounding $V$ so that they are in $\Gamma_{C,\rmvdeadends{\epsilon}}$, the component of $\approxgraph{\epsilon} \setminus Z_\epsilon$ containing $A_{V,\epsilon}$ agrees with the component of $\approxgraph[\wt{Q}_{\epsilon,\rmvdeadends{\epsilon}}]{\epsilon} \setminus Z_\epsilon$ containing $A_{V,\epsilon}$. We can further suppose that $m$ is large enough so that $A$ contains further points that separate each pair of points of $Z$ in $\Upsilon_\Gamma$, and that the corresponding $(x_n)$ in $A_\epsilon$ also separate each pair of points of $Z_\epsilon$ in $\Upsilon_{\Gamma_n}$.

 Let $w_{A_\epsilon}$ (resp.\ $w_{A_{V,\epsilon}}$) be the weight function associated with $R_\epsilon|_{A_\epsilon \times A_\epsilon}$ (resp.\ $R^Q_\epsilon|_{A_{V,\epsilon} \times A_{V,\epsilon}}$). Recall from the proof of Lemma~\ref{le:lim_resistance_metric} that $w_{A_\epsilon} \to w$ and $w_{A_{V,\epsilon}} \to w_V$. By the paragraph above we have $w_{A_\epsilon}(x,y) = w_{A_{V,\epsilon}}(x,y)$ for $x,y \in A_{V,\epsilon}$, and $w_{A_\epsilon}(x,y) = 0$ for $x \in A_{V,\epsilon} \setminus Z_\epsilon$, $y \in A_\epsilon \setminus A_{V,\epsilon}$, and $w_{A_\epsilon}(x,y) = 0$ for $x,y \in Z_\epsilon$. Taking the limit, we conclude that these properties remain true for $w$ and $w_V$.
\end{proof}

\begin{proof}[Proof of Proposition~\ref{pr:lim_weak_rform}]
This now follows from Proposition~\ref{pr:cle_rmet_char} and Lemma~\ref{le:internal_rmetrics} by choosing the countable dense set $\{z_m\}_{m \in \N} \subseteq \Upsilon_\Gamma$ as in the proof of Lemma~\ref{le:internal_rmetrics}.
\end{proof}

\section{Bi-Lipschitz equivalence}
\label{sec:bilipschitz}

The goal of this section is to establish that any two weak \clekp{} resistance forms as defined in Definition~\ref{def:weak_cle_rform} are bi-Lipschitz equivalent with deterministic constants. This is the first step in proving that weak \clekp{} resistance forms are unique (up to constants) and are in fact \clekp{} resistance forms as defined in Definition~\ref{def:cle_rform}.

\subsection{Setup and main statement}
\label{se:bilipschitz_setup}

In this section, we will assume that we have the setup described above Theorem~\ref{thm:unique_metric}.  Let $\Gamma_\D$ be a nested $\CLE_{\kappa'}$ on $\D$, let $\CL$ be the outermost loop of $\Gamma$ such that $0$ is inside $\CL$, and let $C$ be the regions inside $\CL$. We let $\Gamma_C$ be the loops of $\Gamma_\D$ contained in $\ol{C}$ and let $\Gamma = \{ \CL \} \cup \Gamma_C$. We equip the gasket $\Upsilon_\Gamma$ of $\Gamma_C$ with the metric $\dpathY$ defined in~\eqref{eq:dpath}.

Throughout, we suppose that $(\rformres{V}{\cdot}{\cdot}{\Gamma}_{V \in \metregions}$ and $(\trformres{V}{\cdot}{\cdot}{\Gamma})_{V \in \metregions}$ are two weak $\CLE_{\kappa'}$ resistance forms that are conditionally independent given $\Gamma$. We let $(\rmetres{V}{\cdot}{\cdot}{\Gamma})_{V \in \metregions}$ and $(\rmettres{V}{\cdot}{\cdot}{\Gamma})_{V \in \metregions}$ be the associated resistance metrics. In addition, we assume that the two metrics have comparable scaling constants in the sense described below. We note that the results of this section still hold if the \clekp{} resistance forms are defined only on the event that $\CL \cap \partial\D = \emptyset$.

For $\delta \in (0,1]$, let $h^\delta$ be a GFF on $\delta\D$ with boundary values so that if $\varphi\colon \h \to \delta\D$ is a conformal map, then $h^\delta \circ \varphi - \chi\arg\varphi'$ has the boundary values $-\lambda-\angledouble\chi$ (resp.\ $+\lambda$) on $\R_-$ (resp.\ $\R_+$) where $\angledouble$ is defined in~\eqref{eq:angledouble}. Let $\eta^\delta_1$ (resp.\ $\eta^\delta_2$) be the angle $\angledouble$ (resp.\ $0$) flow line of $h^\delta$ from $-i$ to $i$. Let $U^\delta$ be the regions bounded between $\eta^\delta_1,\eta^\delta_2$. Given $\eta^\delta_1,\eta^\delta_2$, let $\Gamma^\delta$ be the conditionally independent collection of \clekp{} in the connected components of $U^\delta$. Let $I^\delta$ be the event that $\eta^\delta_1 \cap \eta^\delta_2 \cap B(0,\delta/2) \neq \emptyset$. Let $\median[\delta]{}$ be the median of the random variable
\begin{equation*}
 \sup_{\substack{x=\eta^\delta_1(s)=\eta^\delta_2(s')\\ y=\eta^\delta_1(t)=\eta^\delta_2(t')\\ \eta^\delta_1[s,t], \eta^\delta_2[s',t'] \subseteq B(0,3\delta/4)}} \rmetres{U^\delta}{x}{y}{\Gamma^\delta} \quad\text{conditioned on } I^\delta .
\end{equation*}
It is explained in~\cite[Section~3.1]{amy2025tightness} that $\rmetres{U^\delta}{\cdot}{\cdot}{\Gamma}$ is well-defined on compact subsets of $\delta\D$ by absolute continuity. Let $\mediant[\delta]{}$ be defined analogously for $\rmett{\cdot}{\cdot}{\Gamma}$.

Assume that there is a constant $c>1$ such that
\begin{equation}\label{eq:ass_medians_comparable}
 c^{-1} \le \frac{\median[\delta]{}}{\mediant[\delta]{}} \le c
 \quad\text{for every } \delta \in (0,1] .
\end{equation}

We state the main result of this section.

\begin{proposition}
\label{prop:weak_bi_lipschitz}
Consider the setup described just above, and assume~\eqref{eq:ass_medians_comparable}. Then there exist deterministic constants $0 < c_1 < c_2 < \infty$ so that almost surely $\rfdomainres{V}{\Gamma} = \trfdomainres{V}{\Gamma}$ and
\[ c_1 \rformres{V}{f}{f}{\Gamma} \leq  \trformres{V}{f}{f}{\Gamma} \leq c_2 \rformres{V}{f}{f}{\Gamma}  \quad\text{for all } V \in \metregions \text{ and } f \in \rfdomainres{V}{\Gamma} .\]
\end{proposition}

The proof of Proposition~\ref{prop:weak_bi_lipschitz} is based on constructing a covering of the space by ``good'' regions inside each of which the behavior of harmonic functions with respect to $\rformres{V}{\cdot}{\cdot}{\Gamma}$, $\trformres{V}{\cdot}{\cdot}{\Gamma}$ is comparable.  In order to carry this out, we will first collect some general estimates on the resistances in Section~\ref{se:resistance_bounds}. We then define a quality of annuli in Section~\ref{subsec:good_annuli}, and use them to find the good regions in Section~\ref{se:good_regions}. Finally, we complete the proof of Proposition~\ref{prop:weak_bi_lipschitz} in Section~\ref{subsec:weak_bilipschitz_proof}.

\newcommand*{\Esep}{E^{\mathrm{sep}}}
\newcommand*{\dsep}{r_{\mathrm{sep}}}

\newcommand*{\Ebreak}{E^{\mathrm{break}}}
\newcommand*{\Fbreak}{F^{\mathrm{break}}}
\newcommand*{\pbreak}{p_{\mathrm{break}}}

The proofs in this and the next section make use of the independence across scales argument established in~\cite{amy-cle-resampling}. For this, we set some notation.

For $z \in \D$, $j\in\N$ we write
\[ A_{z,j} = A(z,2^{-j-1},2^{-j}) . \]
Let $\Gamma_\outside^{*,B(z,3\cdot 2^{-j}),B(z,2\cdot 2^{-j})}$ be the partial exploration of $\Gamma$ and $B(z,3\cdot 2^{-j})^{*,B(z,2\cdot 2^{-j})}$ the unexplored region as defined in Section~\ref{se:mcle}. Let $\Gamma_\inside^{*,B(z,3\cdot 2^{-j}),B(z,2\cdot 2^{-j})}$ be the unexplored part, and let $\alpha^*_{z,j}$ be the interior link pattern in $B(z,3\cdot 2^{-j})^{*,B(z,2\cdot 2^{-j})}$ induced by $\Gamma_\inside^{*,B(z,3\cdot 2^{-j}),B(z,2\cdot 2^{-j})}$. We let
\begin{itemize}
 \item $\wt{\CF}_{z,j}$ be the $\sigma$-algebra generated by $\Gamma_\outside^{*,B(z,3\cdot 2^{-j}),B(z,2\cdot 2^{-j})}$.
 \item $\CF_{z,j}$ be the $\sigma$-algebra generated by $\Gamma_\outside^{*,B(z,3\cdot 2^{-j}),B(z,2\cdot 2^{-j})}$, $\alpha^*_{z,j}$, and $(\rformres{V}{\cdot}{\cdot}{\Gamma})_{V \in \metregions[\C \setminus \ol{B}(z,3\cdot 2^{-j})]}$, $(\trformres{V}{\cdot}{\cdot}{\Gamma})_{V \in \metregions[\C \setminus \ol{B}(z,3\cdot 2^{-j})]}$.
\end{itemize}

The following lemma is a consequence of Theorem~\ref{thm:cle_partially_explored}, the Markovian property of the \clekp{} resistance form, and the conditional independence of $(\rformres{V}{\cdot}{\cdot}{\Gamma}_{V \in \metregions}$, $(\trformres{V}{\cdot}{\cdot}{\Gamma})_{V \in \metregions}$. It will be crucial for establishing independence across scales for the internal metrics.

\begin{lemma}
Let $U \subseteq B(z,2^{-j})$. The conditional law of $(\rformres{V}{\cdot}{\cdot}{\Gamma})_{V \in \metregions[U]}$ given $\CF_{z,j}$ is given by sampling $U^*$ according to the law of a multichordal \clekp{} in $B(z,3\cdot 2^{-j})^{*,B(z,2\cdot 2^{-j})}$ conditionally on $\alpha^*_{z,j}$ and then sampling $(\rformres{V}{\cdot}{\cdot}{\Gamma})_{V \in \metregions[U]}$ via the Markov property given $U^*$.
\end{lemma}

The results in~\cite{amy-cle-resampling} tell us that the conditional probabilities given $\CF_{z,j}$ can be uniformly controlled when the strands are ``separated''. We now give the definition of separation and recall the key lemma giving us the independence across scales property.

Fix $\dsep > 0$. For each $z,j$, let $\Esep_{z,j} = \Esep_{z,j}(\dsep)$ denote the event that $\abs{y_i-y_{i'}} \ge \dsep 2^{-j}$ for each distinct pair $y_i,y_{i'}$ of the marked points of $\Gamma_\outside^{*,B(z,3\cdot 2^{-j}),B(z,2\cdot 2^{-j})}$.

\begin{lemma}[{\cite[Lemma~4.4]{amy-cle-resampling}}]\label{le:separation_event}
For any $b>1$ there exists $\dsep>0$ and $c>0$ such that the following is true. Let $z\in\D$, $j_0 \in \N$ such that $B(z,2^{-j_0}) \subseteq \D$. Then for each $k \in \N$ the probability that more than $1/10$ fraction of the events $(\Esep_{z,j})^c$ where $j = j_0+1,\ldots,j_0+k$ occur is at most $ce^{-bk}$.
\end{lemma}

Finally, we recall the following.

\begin{proposition}[{\cite[Corollary~5.22]{amy2025tightness}}]\label{pr:quantiles_metric}
Let $\median[\delta]{}$ be as defined above. We have
\[
\lambda^{\dsle+o(1)}\median[\delta]{} \le \median[\lambda\delta]{} \le \lambda^{\ddouble+o(1)}\median[\delta]{}
\quad\text{for } \lambda,\delta \in (0,1]
\]
where $\ddouble$, $\dsle$ are defined in~\eqref{eq:ddouble},\eqref{eq:dsle}.
\end{proposition}

\subsection{Resistance bounds}
\label{se:resistance_bounds}

As an intermediate step in the proofs, we consider the following setup. Let $D \subseteq \D$ be a simply connected domain, and let $\Gamma_D$ be a \clekp{} in $D$, and $\p_D$ be its law. Let $\Upsilon_D$ be the gasket of $\Gamma_D$, i.e.\ the set of points connected to $\partial D$ by a path not crossing any loop, equipped with the metric $\dpathY$ as defined in Section~\ref{subsec:main_results}. For every $U \Subset D$, the internal metrics $(\rmetres{V}{\cdot}{\cdot}{\Gamma_D})_{V \in \metregions[U]}$ are defined due to absolute continuity. (In fact, our domains of interest will be complementary connected components of \clekp{} loops, so in that case the metric on all $\Upsilon_D$ is defined.)

In the statement of the lemma below, we use the following notation. For $U \subseteq \C$ and $r>0$, we let $\metregions[U,r] \subseteq \metregions[U]$ be the collection of regions $V \in \metregions[U]$ such that $V$ is a union of connected components of $D \setminus (\CL_1 \cup \cdots \cup \CL_n)$ for some $\CL_1,\ldots,\CL_n \in \Gamma_D$ with $\diamE(\CL_i) \ge r$ for each $i$.

\begin{lemma}\label{le:resistance_ub_cle}
There exists $\zeta>0$ such that for each $r>0$, $b>1$ there exists a constant $c>0$ such that the following is true. Let $D \subseteq \D$ be open, simply connected, and $z \in \D$, $j \in \N$ such that $\distE(z,\partial D) \in [2^{-j+1},2^{-j+2}]$. Let $\Gamma_D$ be a \clekp{} in $D$. Given $\Gamma_D$, sample the internal metrics $(\rmetres{V}{\cdot}{\cdot}{\Gamma_D})_{V \in \metregions[B(z,2^{-j})]}$. Let $G$ denote the event that for each $V \in \metregions[B(z,2^{-j}),r2^{-j}]$ and $x,y \in \ol{V} \cap \Upsilon_D$ we have
\[ \rmetres{V}{x}{y}{\Gamma_D} \le M\median[2^{-j}]{}(\dpathY[\ol{V}](x,y)/2^{-j})^\zeta . \]
Then
\[ \p[G^c] \le cM^{-b} . \]
\end{lemma}

\begin{proof}
 By considering the rescaled metric $\rmet{2^{-j}\cdot}{2^{-j}\cdot}{\Gamma_D}$ \cite[Lemma~1.7]{amy2025tightness}, it suffices to consider the case when $j=1$. The statement is essentially \cite[Proposition~5.17]{amy2025tightness}, with the difference that \cite[Proposition~5.17]{amy2025tightness} is stated for the setup described at the beginning of Section~\ref{se:bilipschitz_setup}. But the proof gives the same estimate uniformly on each cluster and for all domains $D \subseteq \D$ when we restrict to regions away from the boundary of $D$.
\end{proof}

\begin{lemma}\label{le:resistance_lb_cle}
 For every $r>0$, $q > 0$ there exists $m > 0$ such that the following is true. Let $D \subseteq \D$ be open, simply connected, and $z \in \D$, $j \in \N$ such that $\distE(z,\partial D) \in [2^{-j+1},2^{-j+2}]$. Let $\Gamma_D$ be a \clekp{} in $D$. Given $\Gamma_D$, sample the internal metrics $(\rmetres{V}{\cdot}{\cdot}{\Gamma_D})_{V \in \metregions[B(z,2^{-j})]}$. Let $G$ be the event that for each $V \in \metregions[B(z,2^{-j})]$ and $x,y \in \ol{V} \cap \Upsilon_D$ with $\abs{x-y} \ge r2^{-j}$ we have
 \[ \rmetres{V}{x}{y}{\Gamma_D} \ge m\,\median[2^{-j}]{} . \]
 Then
 \[ \p[G^c] \le q . \]
\end{lemma}

\begin{proof}
 This follows from the same argument as the analogue for the geodesic \clekp{} metric in~\cite[Lemma~4.3]{my2025geouniqueness}. The only difference is that we apply the argument in the proof of~\cite[Lemma~4.3]{my2025geouniqueness} to each annulus $A(w,(r/4)2^{-j},(r/2)2^{-j})$ with $w \in (r/4)2^{-j}\Z^2 \cap B(z,2^{-j})$. As a result, we can make $m>0$ small enough so that on an event with probability at least $1-q$, each such annulus contains a bounded number of at most $m^{-1}$ regions, each bounded between a pair of intersecting strands, separating $\bIn A(w,(r/4)2^{-j},(r/2)2^{-j})$ from $\bOut A(w,(r/4)2^{-j},(r/2)2^{-j})$ in $\Upsilon_D$, and such that the resistance across each region is at least $m\median[2^{-j}]{}$. By the generalized parallel law Lemma~\ref{le:weak_rmet_gen_parallel_law} and the monotonicity Lemma~\ref{le:weak_rmet_monotonicity}, we get that $\rmetres{V}{x}{y}{\Gamma_D} \ge m^2\median[2^{-j}]{}$ for each $x,y$ with $\abs{x-y} \ge r2^{-j}$.
\end{proof}

\subsection{Good annulus event}
\label{subsec:good_annuli}

\newcommand*{\rextremitylb}{r_1}
\newcommand*{\rextremityub}{r_2}
\newcommand*{\Eann}{E^{1}}
\newcommand*{\Eannharm}{E^{2}}

\begin{figure}[ht]
\centering
\includegraphics[width=0.49\textwidth,page=1]{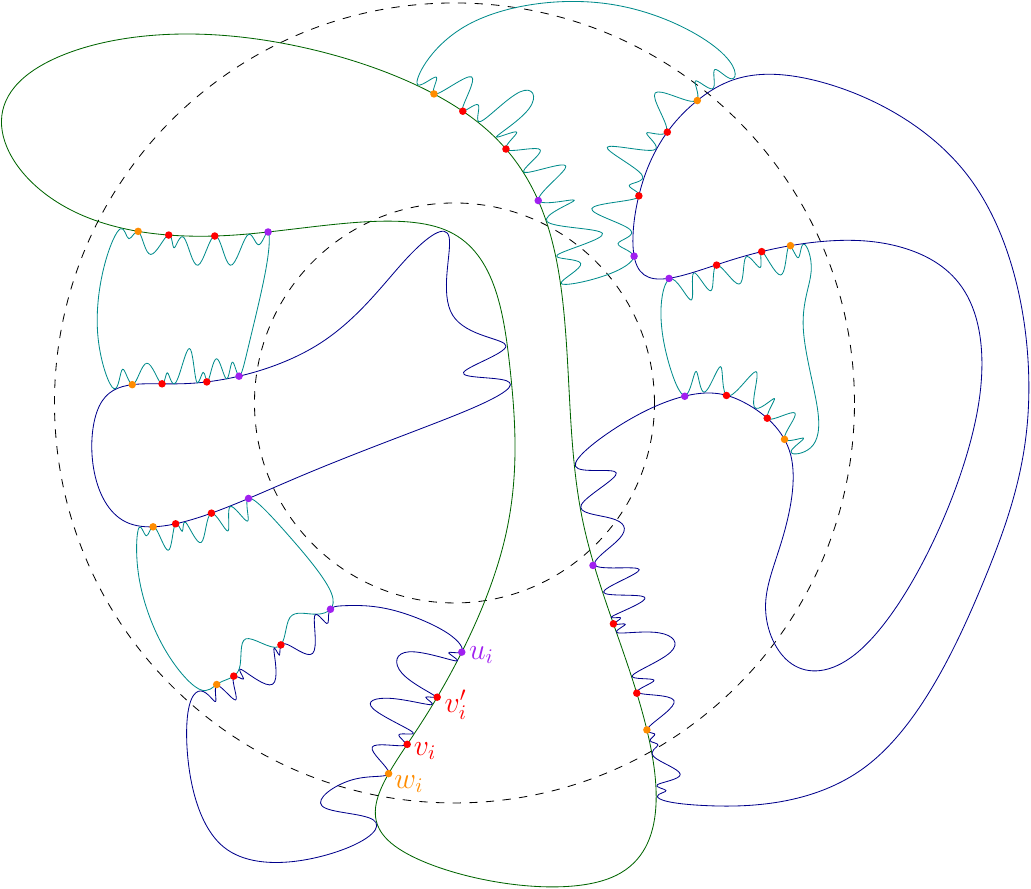}\hspace{0.01\textwidth}\includegraphics[width=0.49\textwidth,page=2]{good_region.pdf}
\caption{Illustration of the event $\Eann_{z,j}$, the points $w_i,v_i,v'_i,u_i$, and the regions $V_i,V'_i$.}
\label{fi:good_annulus}
\end{figure}

We consider now again the setup in Section~\ref{se:bilipschitz_setup}. We are now going to define a ``good'' event that occurs on an annulus~$A_{z,j}$. Let $M>1$ and $\rextremityub > \rextremitylb > 0$ be parameters, and $z \in \D$, $j \in \N$ such that $B(z,2^{-j}) \subseteq \D$. We let $\Eann_{z,j} = \Eann_{z,j}(M,\rextremitylb,\rextremityub)$ be the following event.
\begin{enumerate}[(i)]
\item There is a collection of at most $M$ strands $\gamma_1^L,\gamma_1^R,\ldots,\gamma_m^L,\gamma_m^R$ of loops of $\Gamma$ that do not overlap each other, are contained in $A_{z,j}$, and such that the following holds. We orient the strands so that the left side of $\gamma_i^L$ and the right side of $\gamma_i^R$ are in the gasket of $\Gamma$, and suppose that the left side of $\gamma_i^L$ intersects the right side of $\gamma_i^R$.
\item For each $i$ there are points $w_i,v_i,v'_i,u_i \in \gamma_i^L \cap \gamma_i^R$ in the same order as visited by the strands with $\abs{w_i-v_i}, \abs{v_i-v'_i}, \abs{v'_i-u_i} \in [\rextremitylb 2^{-j}, \rextremityub 2^{-j}]$, and so that $\{u_1,\ldots,u_m\}$ separates $\bIn A_{z,j}$ from $\bOut A_{z,j}$ in $\Upsilon_\Gamma$. We further suppose that $\{u_1,\ldots,u_m\}$ separates $\{v_1,\ldots,v_m\}$ from $\bIn A_{z,j}$ in $\Upsilon_\Gamma$ (i.e.\ the points $w_i,v_i,v'_i,u_i$ are ordered from the outside to the inside).
\item We suppose that $w_i,v_i$ are two consecutive intersection points, i.e.\ there is no further intersection point of $\gamma_i^L \cap \gamma_i^R$ between them.
\end{enumerate}

We let $\Eannharm_{z,j}$ be the event that $\Eann_{z,j}$ occurs and we can select the points so that additionally the following holds. For each $i=1,\ldots,m$, let $V_i$ (resp.\ $V'_i$) be the region bounded between the segments of $\gamma_i^L,\gamma_i^R$ from $v_i$ to $v'_i$ (resp.\ $v'_i$ to $u_i$). Then
\begin{equation}\label{eq:extremities_energy}
 \rmetres{V_i}{v_i}{v'_i}{\Gamma},\ \rmetres{V'_i}{v'_i}{u_i}{\Gamma} \in [M^{-1}\median[2^{-j}]{}, M\median[2^{-j}]{}] ,
\end{equation}
and likewise for $\rmett{\cdot}{\cdot}{\Gamma}$.

See Figure~\ref{fi:good_annulus} for an illustration of the setup and the notation which we will make use of later in the article.

Recall the $\sigma$-algebras $\CF_{z,j}$, $\wt{\CF}_{z,j}$ and the events $
\Esep_{z,j}$ (depending on the parameter $\dsep>0$) defined in Section~\ref{se:bilipschitz_setup}.

\begin{lemma}
\label{le:good_ann_event}
For each $p \in (0,1)$, $\dsep>0$ there exist~$M,\rextremitylb,\rextremityub$ so that
\[ \p[ \Eann_{z,j} \giv \CF_{z,j}] \one_{\Esep_{z,j}} \geq p \one_{\Esep_{z,j}} \]
for each $z,j$ with $B(z,2^{-j}) \subseteq \D$.
\end{lemma}

\begin{proof}
Let us first show the statement with $\wt{\CF}_{z,j}$ in place of $\CF_{z,j}$. Recall (cf.\ \cite[Lemma~5.14]{amy-cle-resampling}) that for the multichordal \clekp{} in a fixed marked domain containing $B(z,2^{-j})$, there almost surely exists a finite set of points that separate $\bIn A_{z,j}$ from $\bOut A_{z,j}$. Since two intersecting strands intersect infinitely many times, we can select the parameters $M,\rextremitylb,\rextremityub$ so that the probability of $\Eann_{z,j}$ is as close to $1$ as we want for the multichordal \clekp{} in a fixed marked domain. By the local continuity of the multichordal \clekp{} law in Proposition~\ref{prop:mccle_tv_convergence_int}, we can make the probability close to $1$ uniformly among all $\dsep$-separated marked domains. By scaling, the statement does not depend on~$j$.

This proves the statement with $\wt{\CF}_{z,j}$ in place of $\CF_{z,j}$. The conditional law under $\CF_{z,j}$ is given by additionally conditioning on the interior link pattern. By Proposition~\ref{pr:link_probability} the probability of each interior link pattern is uniformly bounded from below among $\dsep$-separated marked point configurations. This implies that the same result holds for $\CF_{z,j}$.
\end{proof}

\begin{lemma}
\label{le:good_ann_event_harm}
For each $p \in (0,1)$, $\dsep>0$ there exist~$M,\rextremitylb,\rextremityub$ so that
\[ \p[ \Eannharm_{z,j} \giv \CF_{z,j}] \one_{\Esep_{z,j}} \geq p \one_{\Esep_{z,j}} \]
for each $z,j$ with $B(z,2^{-j}) \subseteq \D$.
\end{lemma}

\begin{proof}
The proof is the same as for Lemma~\ref{le:good_ann_event}, but we additionally need to argue that the resistances can be controlled uniformly in $j$. As before, it suffices to prove the statement with $\wt{\CF}_{z,j}$ in place of $\CF_{z,j}$. We are therefore going to show that for each $\dsep>0$, $q>0$, we can make $M$ large enough so that
\begin{equation}\label{eq:good_intersections_region}
 \p[ (\Eann_{z,j} \setminus \Eannharm_{z,j}) \giv \wt{\CF}_{z,j}] \one_{\Esep_{z,j}} \le q .
\end{equation}
We use a resampling argument to deduce it from the analogous statements for ordinary \clekp{} Lemmas~\ref{le:resistance_ub_cle},~\ref{le:resistance_lb_cle}.

Suppose this were not true. Consider the resampling procedure within $A_{z,j-1}$ in Lemma~\ref{lem:break_loops}, and let $\Fbreak_{z,j}$ be the event that all crossings of $A_{z,j-1}$ in the resampled \clekp{} $\Gamma^\resampled_{z,j-1}$ are broken up. For some $\pbreak > 0$ let $\Ebreak_{z,j}$ be the event (for $\Gamma$) that $\p[\Fbreak_{z,j} \mid \Gamma] \ge \pbreak$. By the proof of Lemma~\ref{lem:break_loops} given in \cite{amy-cle-resampling} we can find for each $\dsep>0$, $q>0$ some $\pbreak>0$ so that for a suitable resampling procedure we have
\[ \mcclelaw{D;\ul{x};\beta}[ (\Ebreak_{z,j})^c ] < q/2 \]
uniformly for each $(D;\ul{x};\beta)$ where $\distE(z,\partial D) \in [2^{-j+1},2^{-j+2}]$ and $\ul{x}$ is $\dsep$-separated (note that this statement does not depend on $j$ due to scale-invariance). If~\eqref{eq:good_intersections_region} were false, there would exist some $(D;\ul{x};\beta)$ and a set of boundary arcs corresponding to the gasket so that (suppressing the set of boundary arcs in the notation below)
\[ \mcclelaw{D;\ul{x};\beta}[ (\Eann_{z,j} \setminus \Eannharm_{z,j}) \cap \Ebreak_{z,j} ] \ge q/2 . \]
Since the resampling procedure is a measure-preserving transformation, we then have
\[ \mcclelaw{D;\ul{x};\beta}[ (\Eann_{z,j} \setminus \Eannharm_{z,j}) \cap \Fbreak_{z,j} ] \ge \pbreak q/2 . \]
If we condition on the chords of a multichordal \clekp{}, the remainder has the conditional law of an ordinary \clekp{}. Hence, there exists some $\wt{D}$ with $\distE(z,\partial \wt{D}) \in [2^{-j},2^{-j+1}]$ such that
\[ \p_{\wt{D}}[ \Eann_{z,j} \setminus \Eannharm_{z,j} ] \ge \pbreak q/2 \]
where $\p_{\wt{D}}$ denotes the law of \clekp{} in $\wt{D}$. For sufficiently large $M$, this contradicts Lemma~\ref{le:resistance_ub_cle} or~\ref{le:resistance_lb_cle}.
\end{proof}

\subsection{Covering by good regions}
\label{se:good_regions}

\newcommand*{\Gregion}{G^1}
\newcommand*{\Gresbdball}{G^{\FR}}

The main goal of this subsection is to prove the following result which implies that with probability $1$ we can cover the space with a finite number of ``good'' regions in which the internal resistance forms for both $\rform{\cdot}{\cdot}{\Gamma}$ and $\trform{\cdot}{\cdot}{\Gamma}$ are comparable to their typical behavior.

\begin{figure}[ht]
\centering
\includegraphics[width=0.49\textwidth,page=3]{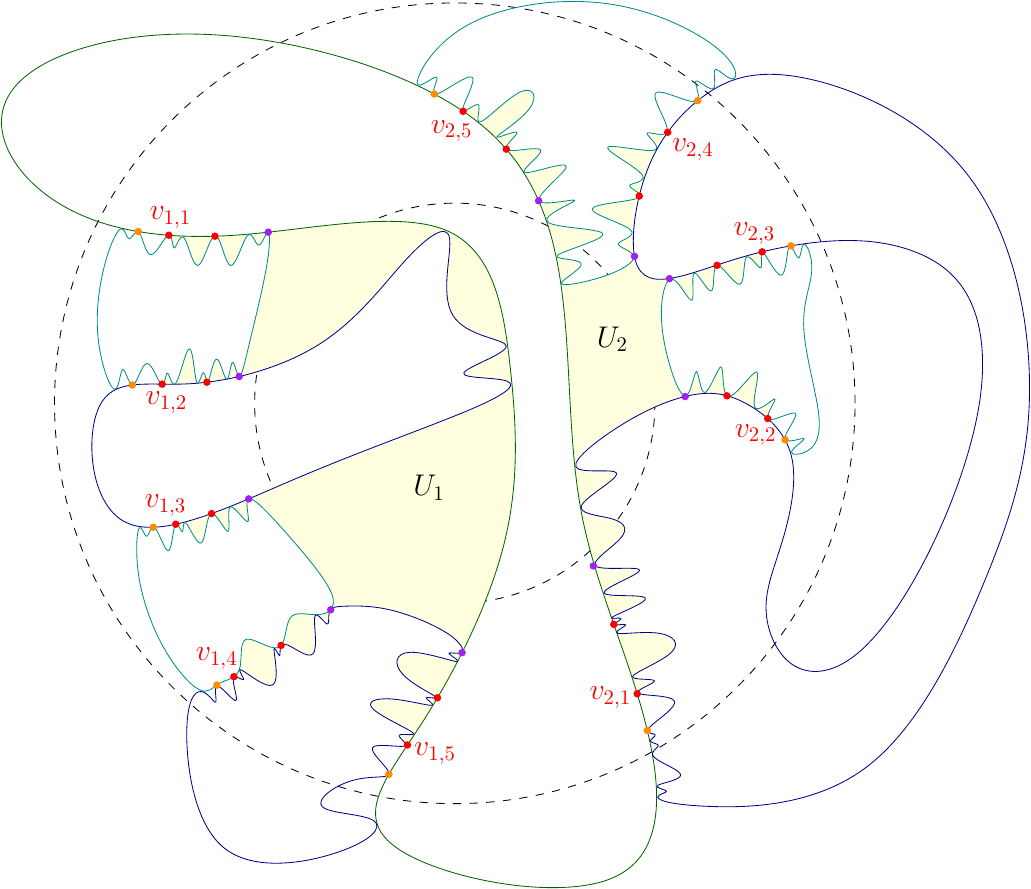}\hspace{0.01\textwidth}\includegraphics[width=0.49\textwidth,page=4]{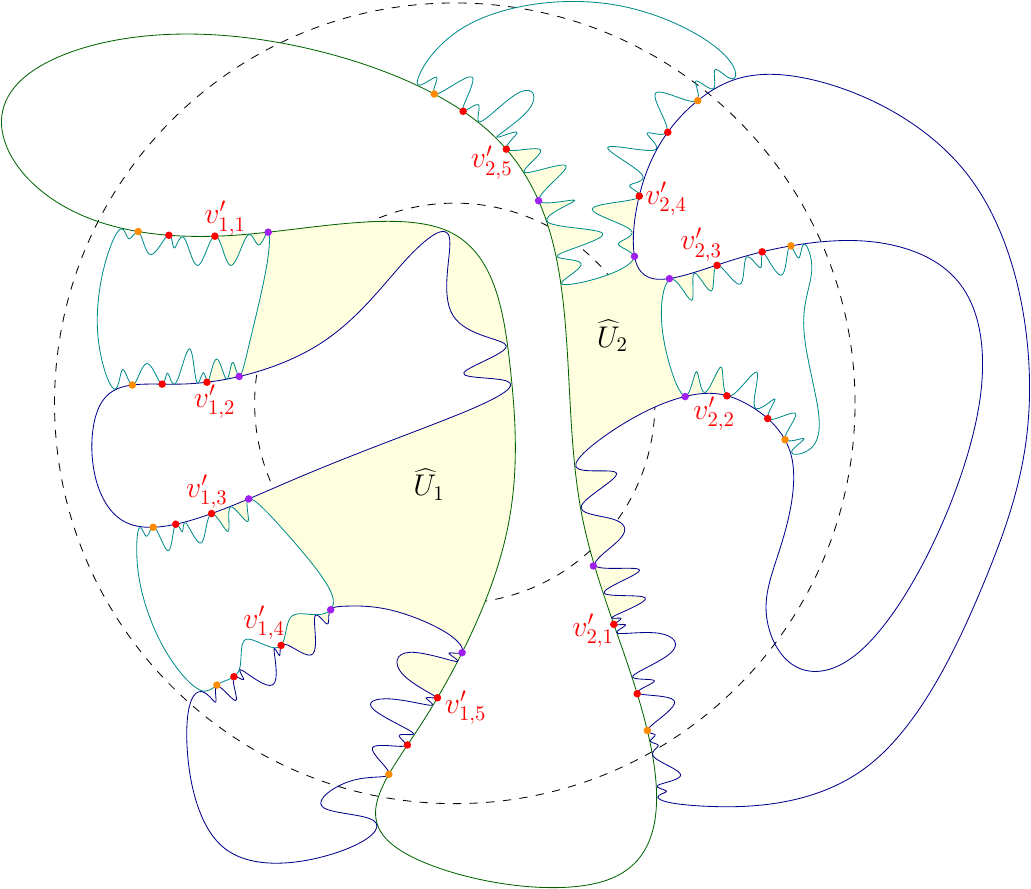}
\caption{\label{fi:good_region} Continuation of Figure~\ref{fi:good_annulus}. The loops involved in the event $\Eann_{z,j}$ may split $B(z,2^{-j-1})$ into several components. In the figure, there are two components. We let $U_l$ be the regions bounded by the $\{v_i\}$, and let $\wh{U}_l$ be the regions bounded by the $\{v'_i\}$. We relabel the points $v_i,v'_i$, etc.\ according to the components they are adjacent to.}
\end{figure}

To state the result, we define a ``good'' event for each ball $B(z,2^{-j}) \subseteq \D$. Suppose that we are on the event $\Eann_{z,j}$ defined in Section~\ref{subsec:good_annuli}. We keep using the notation for the points $w_i,v_i,v'_i,u_i \in \gamma_i^L \cap \gamma_i^R$ and the regions $V_i, V'_i$ from the definition of $\Eann_{z,j}$ and $\Eannharm_{z,j}$. Further, let $K_1,\ldots,K_L$ be the connected components of $\Upsilon_\Gamma \setminus \{v_i\}$ that intersect $B(z,2^{-j-1})$. For each $l=1,\ldots,L$, we let $U_l \in \metregions$ be such that $\ol{U}_l = \Fill(\ol{K_l})$. We let $\{u_{l,1},u_{l,2},\ldots\}$ be the subset of $\{u_i\}$ that are adjacent to $U_l$. We define $w_{l,i},v_{l,i},v'_{l,i},u_{l,i},V_{l,i},V'_{l,i}$ analogously. We further write $\wh{U}_l = U_l \setminus \bigcup_i V_{l,i}$ (i.e.\ $\{\wh{U}_l\}$ are the filled connected components of $\Upsilon_\Gamma \setminus \{v'_i\}$ that intersect $B(z,2^{-j-1})$). See Figure~\ref{fi:good_region} for an illustration.

We will refer to the set $\{v_{l,1},v_{l,2},\ldots\}$ as the boundary of $U_l$, and the set $\{v'_{l,1},v'_{l,2},\ldots\}$ as the boundary of $\wh{U}_l$. We say that a function is harmonic in $U_l$ (resp.\ $\wh{U}_l$) if it is harmonic except at its boundary points (i.e.\ the $\rformres{U_l}{\cdot}{\cdot}{\Gamma}$-energy minimizers among the functions with prescribed boundary values, see Section~\ref{se:rmet_rform}). Here we have slightly abused notation and identified $U_l$ with $\Upsilon_\Gamma \cap \ol{U}_l$. Note that the additivity of energies (Lemma~\ref{le:rform_additivity}) implies that a harmonic function in a larger region is also harmonic in $U_l$. Recall that harmonic functions attain their maximum and minimum values on the boundary, due to~(RF\ref{it:rf_markov}).

We will sometimes write $K^j_l$ (resp.\ $U^j_l$) to emphasize the dependence on $j$.

Let $M>1$ and $\rextremityub > \rextremitylb > 0$. For each $z \in \D$, $j \in \N$ with $B(z,2^{-j}) \subseteq \D$, we let $\Gregion_{z,j} = \Gregion_{z,j}(M,\rextremitylb,\rextremityub)$ be the event that the event $\Eannharm_{z,r}$ from Section~\ref{subsec:good_annuli} occurs and for each $l$, for any non-constant $\rformres{U_l}{\cdot}{\cdot}{\Gamma}$-harmonic function $f$ we have that
\begin{equation}\label{eq:good_region_energy}
M^{-1} \le \frac{\median[2^{-j}]{} \rformres{U_l}{f}{f}{\Gamma}}{(\sup_{K_l} f - \inf_{K_l} f)^2} \le M .
\end{equation}

\begin{proposition}\label{pr:good_region}
For each $b>0$ there exist $M>1$, $\rextremityub > \rextremitylb > 0$, and $c>0$ such that the following is true. Let $z \in \D$, $j \in \N$ with $B(z,2^{-j}) \subseteq \D$. Then for each $k \in \N$ we have
\[
 \p\left[ \text{$(\Gregion_{z,j'})^c$ occurs for more than a $1/5$ fraction of $j' = j,\ldots,j+k$} \right] \le ce^{-bk} .
\]
\end{proposition}

In contrast to the event $\Eannharm_{z,j}$, the condition~\eqref{eq:good_region_energy} involves looking into the regions inside $B(z,2^{-j-1})$ and not just $A_{z,j}$. Therefore we do not have the independence between subsequent scales any more. This lack of independence can be compensated by using the superpolynomial tail for the resistances proved in \cite{amy2025tightness}. An analogous result for the geodesic \clekp{} metric has been proved in \cite{my2025geouniqueness}. We state the resistance metric analogue of \cite[Proposition~3.1]{my2025geouniqueness}.

\begin{lemma}\label{le:good_resistance_bound_ball}
 For every $b>0$ there exist $M>1$, $\rextremityub > \rextremitylb > 0$, and $c>0$ such that the following is true. Let $z \in \D$, $j \in \N$ with $B(z,2^{-j}) \subseteq \D$. Let $\Gresbdball_{z,j}$ be the event that $\Eannharm_{z,j}$ occurs and additionally
 \[ \max_l \sup_{x,y \in K^j_l} \rmetres{U^j_l}{x}{y}{\Gamma} \le M\median[2^{-j}]{} . \]
 Then for each $k \in \N$ we have
 \[ \p\left[ \text{$\Eannharm_{z,j'} \cap (\Gresbdball_{z,j'})^c$ occurs for more than a $1/10$ fraction of $j' = j,\ldots,j+k$} \right] \le ce^{-bk} . \]
\end{lemma}

\begin{proof}
The proof of Lemma~\ref{le:good_resistance_bound_ball} is completely analogous to that of \cite[Proposition~3.1]{my2025geouniqueness}. The main inputs are the polynomial scaling of $(\median[\lambda]{})$ stated in Proposition~\ref{pr:quantiles_metric}, the superpolynomial tail for the resistance metric stated in Lemma~\ref{le:resistance_ub_cle}, which is utilized through a resampling argument, and the Markovian property of the internal metrics. Since all these properties are proved in \cite{amy2025tightness} for a general class of \clekp{} metrics including both the geodesic and the resistance metric, the same proof as for \cite[Proposition~3.1]{my2025geouniqueness} carries over to the case of the resistance metric. The only difference to \cite[Proposition~3.1]{my2025geouniqueness} is the regions for which we estimate the internal metrics. Here, we consider the following regions. Suppose that $j \in \N$ is a scale where $\Eannharm_{z,j}$ occurs, and let $j_+ > j$ be the next scale where $\Eannharm_{z,j_+}$ occurs. Let $\{v_i\}$ be the points in the definition of $\Eann_{z,j}$, and let $\{v^+_i\}$ be the points in the definition of $\Eann_{z,j_+}$. Let $K^{j,+}_1,K^{j,+}_2,\ldots$ be the connected components of $\Upsilon_\Gamma \setminus (\{v_i\} \cup \{v^+_i\})$ that intersect $A(z,2^{-j_+},2^{-j-1})$. Letting $U^{j,+}_l \in \metregions$ be such that $\ol{U}^{j,+}_l = \Fill(\ol{K}^{j,+}_l)$ for each $l$, the proof of \cite[Proposition~5.17]{amy2025tightness} shows superpolynomial upper tails for the internal metrics $(\median[2^{-j}]{})^{-1}\rmetres{U^{j,+}_l}{\cdot}{\cdot}{\Gamma}$. Since each $\wh{U}^j_l$ decomposes into a collection of $U^{j',+}_{l'}$ for some $j' \ge j$ and $l'$, by the monotonicity Lemma~\ref{le:weak_rmet_monotonicity} we obtain the desired upper bound on $\rmetres{\wh{U}^j_l}{\cdot}{\cdot}{\Gamma}$ by summing up the resistances across the subregions as in the proof of \cite[Proposition~3.1]{my2025geouniqueness}.
\end{proof}

\begin{proof}[Proof of Proposition~\ref{pr:good_region}]
The upper bound in~\eqref{eq:good_region_energy} holds automatically on the event $\Eannharm_{z,r}$. Indeed, suppose that $f$ is a $\rformres{U_l}{\cdot}{\cdot}{\Gamma}$-harmonic function and $\sup_{U_l} u - \inf_{U_l} u = 1$. The energy of $f$ is bounded from above by the energy of a function $f'$ that is constant on $\wh{U}_l$ and harmonically interpolates in each $V'_{l,i}$. On the event $\Eann_{z,r}$, the number of $V'_{l,i}$ is bounded by $M$, and on the event $\Eannharm_{z,r}$, we have $\rformres{V'_{l,i}}{f'}{f'}{\Gamma} \le \rmetres{V'_{l,i}}{v'_{l,i}}{v_{l,i}}{\Gamma}^{-1} \le M(\median[2^{-j}]{})^{-1}$ for each $i$. By Lemma~\ref{le:rform_additivity}, we then have
\[ \rformres{U_l}{f'}{f'}{\Gamma} \le \sum_i \rformres{V'_{l,i}}{f'}{f'}{\Gamma} \le M^2(\median[2^{-j}]{})^{-1} \]
as required.

The lower bound in~\eqref{eq:good_region_energy} holds on the event $\Gresbdball_{z,j}$ from Lemma~\ref{le:good_resistance_bound_ball}. Indeed, suppose that $f \in \rfdomainres{U_l}{\Gamma}$ for some $l$ and $x,y \in K_l$. Recall that
\[ \rmetres{U_l}{x}{y}{\Gamma} = \sup\left\{ \frac{(f(x)-f(y))^2}{\rformres{U_l}{f}{f}{\Gamma}} \mmiddle| f \in \rfdomainres{U_l}{\Gamma},\ \rformres{U_l}{f}{f}{\Gamma} > 0 \right\} . \]
Therefore, if $\rmetres{U_l}{x}{y}{\Gamma} \le M\median[2^{-j}]{}$, then
\[
\median[2^{-j}]{} \rformres{U_l}{f}{f}{\Gamma} \ge M^{-1}(f(x)-f(y))^2 
\]
for each $x,y \in K_l$ as required.
\end{proof}

\begin{corollary}
\label{co:covering}
There exist $M,\rextremitylb,\rextremityub$ such that the following holds almost surely. For each $\delta > 0$ there exists a finite collection of balls $\{ B(z_k,2^{-j_k}) \}$ with $2^{-j_k} < \delta$ such that $B(0,1-\delta) \subseteq \bigcup_k B(z_k,2^{-j_k-1})$ and $\Gregion_{z_k,j_k}$ occurs for each $k$.
\end{corollary}

\begin{proof}
Applying Proposition~\ref{pr:good_region} with $b=5$ and $k=j$ (say) for each $j \ge \log_2(\delta^{-1})$ and each $z \in 2^{-2j}\Z^2 \cap B(0,1-\delta)$, the union bound yields that the probability that $B(0,1-\delta)$ is not covered by a suitable collection of balls of scales in $\{j,\ldots,2j\}$ is $O(2^{-j})$. Sending $j \to \infty$ yields the claim.
\end{proof}

Consider the cover in Corollary~\ref{co:covering}. For each $B(z_k,2^{-j_k})$, we consider the regions $U_l$ (resp.\ $\wh{U}_l$) defined at the beginning of Section~\ref{se:good_regions}. Let us enumerate these regions associated to all $B(z_k,2^{-j_k})$, and (with a slight abuse of notation) we again denote them by $\{ U_l \}$ (resp.\ $\{ \wh{U}_l \}$) where distinct $U_l,U_{l'}$ may arise from the same or from different $z_k,j_k$. We let $w_{l,i},v_{l,i},v'_{l,i},u_{l,i},V_{l,i},V'_{l,i}$ be defined in the same way as before. The fact that $B(0,1-\delta)$ is covered by $\{ B(z_k,2^{-j_k-1}) \}$ means that $\Upsilon_\Gamma \cap B(0,1-\delta) \subseteq \bigcup_l U_l$.

We conclude this section with a few properties of the cover defined above. Some of the properties will be used in Section~\ref{sec:uniqueness}. The next lemma says that the overlap between the sets in the cover is bounded by a fixed constant.

\begin{lemma}\label{le:cover_overlap}
 For each $M,\rextremitylb$ there is a constant $c$ such that the following is true. Let $\{ U_l \}$ be a cover as above, and assume that it is minimal (i.e.\ $\{ U_l \mid l \neq l_0 \}$ is not a cover for any $l_0$). Then for each $l$, the number of $l'$ with $U_l \cap U_{l'} \neq \emptyset$ is at most $c$.
\end{lemma}

As a consequence of Lemma~\ref{le:cover_overlap} and Lemma~\ref{le:rform_additivity}, we have
\begin{equation}\label{eq:energy_sum_cover}
c^{-1}\sum_l \rformres{U_l}{f}{f}{\Gamma} \le \rform{f}{f}{\Gamma} \le \sum_l \rformres{U_l}{f}{f}{\Gamma}
\end{equation}
for each minimal cover $\{ U_l \}$ as above.

\begin{proof}[Proof of Lemma~\ref{le:cover_overlap}]
 By the definition of the event $\Eann_{z,j}$, the number of components $K_l$ for each given $z,j$ is bounded by $M$. For each $B(z,2^{-j})$, the number of $B(z',2^{-j'})$ with comparable size that intersect it is bounded by a fixed constant, and therefore for each $U_l$ the number of associated $U_{l'}$ with comparable size as $U_l$ is also bounded by a constant depending on $M$. We therefore only need to consider the intersections of $U_l$ with $U_{l'}$ with very different sizes.
 
 Recall that we have $\rextremitylb 2^{-j} \le \abs{v_i-v'_i} \le \rextremityub 2^{-j}$ for each pair $v'_i,v_i$ in the definition of $\Eann_{z,j}$, and that the next intersection point $w_i$ of the same strands outside $U_l$ also satisfies $\abs{w_i-v_i} \ge \rextremitylb 2^{-j}$. We argue below that if $U_{l_1},U_{l_2},U_{l_3}$ have sizes all differing by at least some large factor of $M$, then $U_{l_1} \cap U_{l_2} \cap U_{l_3} = \emptyset$. To conclude the proof of the lemma, observe that if $U_{l'}$ intersects $U_l$, it must intersect one of the $v_{l,i}$ since otherwise it would be contained in $U_l$, contradicting the minimality of the cover. By the claim above, the number of $U_{l'}$ that intersect a given $v_{l,i}$ is bounded by a constant depending only on $M$. Since there are at most $M$ many $v_{l,i}$, we conclude.
 
 It remains to justify the claim above. Suppose that $U_{l_1} \cap U_{l_2} \cap U_{l_3} \neq \emptyset$ and their sizes all differ by more than some large factor of $M$, and say $U_{l_1}$ is the largest of them. Then $U_{l_2},U_{l_3}$ cannot intersect $\wh{U}_{l_1}$, otherwise they would be contained in $U_{l_1}$, contradicting the minimality of the cover. Therefore one of the points $v_{l,i_1}$ must be contained in both $U_{l_2},U_{l_3}$. But the condition on $w_{l,i_1}$ implies that in fact $v_{l,i_1} \in \wh{U}_{l_2},\wh{U}_{l_3}$ since in case $v_{l,i_1} \in \ol{V}_{l',i_2}$ (resp.\ $\ol{V}_{l',i_3}$) the $V_{l',i_2}$ (resp.\ $V_{l',i_3}$) must terminate at $v_{l,i_1}$ as it is too short to reach $w_{l,i_1}$. Therefore $U_{l_2},U_{l_3}$ must have comparable sizes, otherwise one would be contained in the other, again contradicting the minimality of the cover.
\end{proof}

\begin{lemma}\label{le:disjoint_subcover}
 Let $c$ be the constant in Lemma~\ref{le:cover_overlap}. Suppose that $\{ U_l \}$ is a minimal cover as in Lemma~\ref{le:cover_overlap}. Then for each $f \in \rfdomain{\Gamma}$ there is a subcollection $(l_k)$ such that the regions $\{ U_{l_k} \}$ are mutually disjoint and
 \[ \sum_k \rformres{U_{l_k}}{f}{f}{\Gamma} \le \rform{f}{f}{\Gamma} \le c\sum_k \rformres{U_{l_k}}{f}{f}{\Gamma} . \]
\end{lemma}

\begin{proof}
 Let us order the collection $(U_l)$ such that $\rformres{U_l}{f}{f}{\Gamma}$ is in decreasing order as $l$ increases. We define $(U_{l_k})$ inductively. Let $U_{l_1} = U_1$. Suppose that $U_{l_1},\ldots,U_{l_k}$ are defined. Let
 \[ l_{k+1} \defeq \min\{ l > l_k \mid U_l \cap (U_{l_1} \cup \cdots \cup U_{l_k}) = \emptyset \} . \]
 Then the collection $(U_{l_k})$ is mutually disjoint by definition. Moreover, for each $l \notin \{l_k\}$ there is $l_k < l$ with $U_{l_k} \cap U_l \neq \emptyset$. By Lemma~\ref{le:cover_overlap}, each $U_{l_k}$ intersects at most $c$ other $U_l$, and since we have ordered the regions according to $\rformres{U_l}{f}{f}{\Gamma}$, we have $\rformres{U_l}{f}{f}{\Gamma} \le \rformres{U_{k_l}}{f}{f}{\Gamma}$ for each $l > l_k$. Therefore we have
 \[ \rform{f}{f}{\Gamma} \le \sum_l \rformres{U_l}{f}{f}{\Gamma} \le c\sum_k \rformres{U_{l_k}}{f}{f}{\Gamma} . \]
\end{proof}

\begin{lemma}\label{le:energy_in_extremities}
 For each $M$ there is a constant $c>0$ such that the following is true. Let $U_l$ be a good region as in the definition of the event $\Gregion_{z,j}$ from Section~\ref{se:good_regions}. Let $f$ be a $\rformres{U_l}{\cdot}{\cdot}{\Gamma}$-harmonic function. Then
 \[ \max_i \rformres{V_{l,i}}{f}{f}{\Gamma} \ge c\rformres{U_l}{f}{f}{\Gamma} . \]
\end{lemma}

\begin{proof}
 Upon relabeling, it suffices to assume that $f(v_{l,1}) = 0$, $f(v_{l,2}) = 1$, and $0 \le f \le 1$ on $K_l$. By~\eqref{eq:extremities_energy},~\eqref{eq:good_region_energy}, we are done if we show that $f(v'_{l,1}) \ge c$ for a constant $c$ depending only on $M$. We explain that if $f(v'_{l,1}) < c$, then $\rformres{V_l}{f}{f}{\Gamma}$ cannot be optimal. Throughout the proof, $c$ denotes a constant depending only on $M$ whose value may change from line to line.
 
 For each $a \in [0,1]$, let $f_a$ be the function with $f_a(v'_{l,i}) = f(v'_{l,i}) \vee a$ for each $i$, and harmonically interpolate in each of the regions $V_{l,i}$ and in $\wh{U}_l$. Let $w_{ij} \ge 0$ be the edge conductances corresponding to the trace resistance form $\rformrestr{\wh{U}_l}{\{v'_{l,i}\}}{\cdot}{\cdot}{\Gamma}$~\eqref{eq:trace_form}. Since $\rmetres{\wh{U}_l}{v'_{l,1}}{v'_{l,2}}{\Gamma} \le M\median[2^{-j}]{}$ by the definition of the event $\Gregion_{z,j}$, by the max-flow-min-cut theorem (see e.g.~\cite[Lemma~2.5]{dw2024tightness}), there is some path from $v'_{l,1}$ to $v'_{l,2}$ in the corresponding graph using edges with conductances at least $\wt{c}(\median[2^{-j}]{})^{-1}$. It follows that if $a < c$, the energy $\rformres{\wh{U}_l}{f_a}{f_a}{\Gamma}$ decreases linearly in $a$. On the other hand, $\rformres{V_{l,i}}{f_a}{f_a}{\Gamma} = \rmetres{V_{l,i}}{v'_{l,i}}{v_{l,i}}{\Gamma}^{-1}(f_a(v'_{l,i})-f_a(v_{l,i}))^2$ increases at most quadratically in $a$ for each $i$ with $f_a(v'_{l,i}) < c$. Therefore, by Lemma~\ref{le:rform_additivity}, the energy $\rformres{U_l}{f_a}{f_a}{\Gamma}$ decreases in $a \in [0,c]$. This gives the desired contradiction and shows that $f(v'_{l,1}) \ge c$ when $f$ is harmonic.
\end{proof}

\subsection{Proof of bi-Lipschitz equivalence}
\label{subsec:weak_bilipschitz_proof}

We use the ``good regions'' defined in Section~\ref{se:good_regions} to complete the proof of Proposition~\ref{prop:weak_bi_lipschitz}. Recall the events $\Gregion_{z,j}$ defined above Proposition~\ref{pr:good_region}.

\begin{figure}[ht]
\centering
\includegraphics[width=0.48\textwidth,page=1]{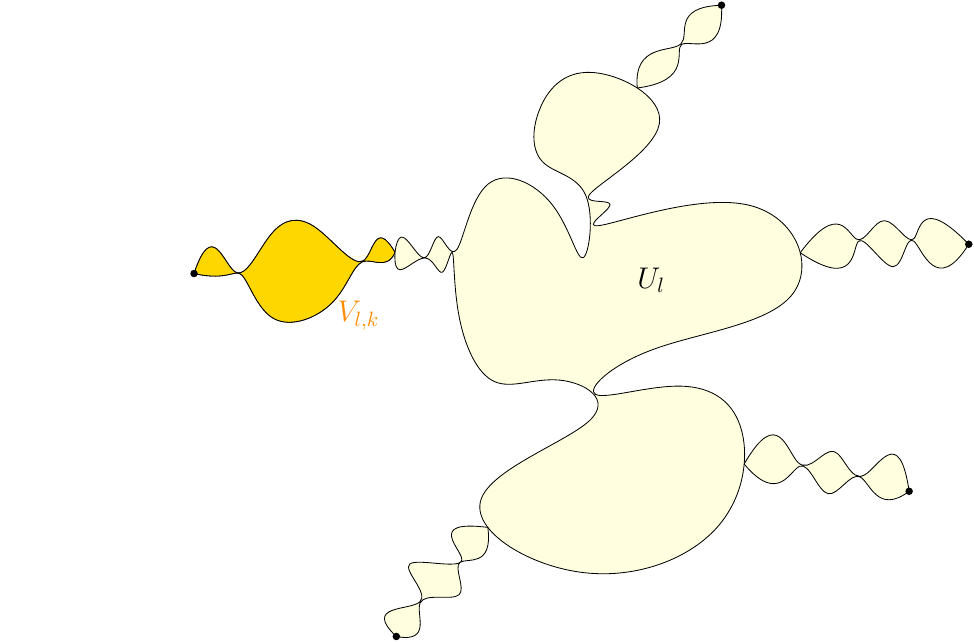}\hspace{0.03\textwidth}\includegraphics[width=0.48\textwidth,page=3]{f_construction1.pdf}
\caption{\label{fi:f_construction_opposite_way} In the proof of Proposition~\ref{prop:weak_bi_lipschitz}, if $V_{l,k} \in \CI_{n-1}$ and $v_{n,i} \in V_{l,k}$, it might be that $\wt{g}_{n-1}(v_{n,j}) < \inf_{K_n} f$. In the depicted case, we set $\wt{g}_n$ to be constant in the region tiled in blue.}
\end{figure}

\begin{figure}[ht]
\centering
\includegraphics[width=0.48\textwidth,page=1]{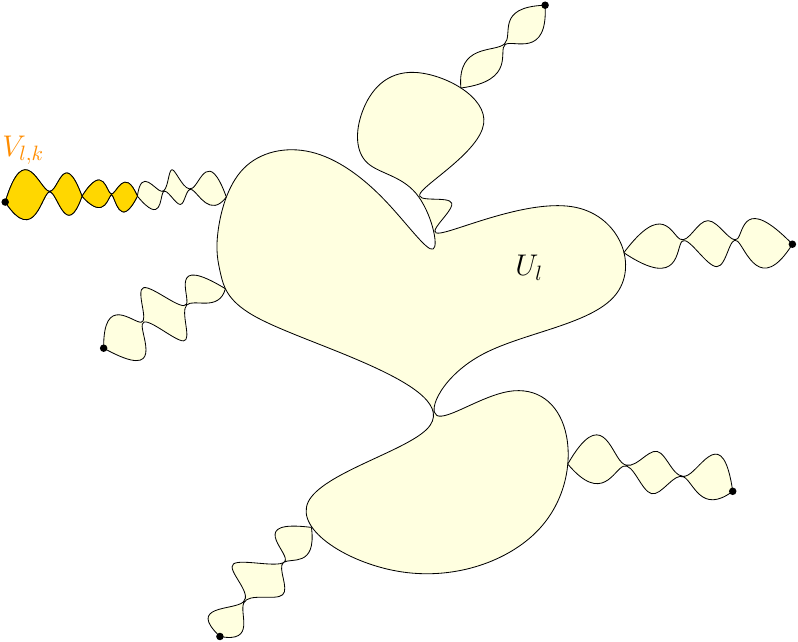}\hspace{0.03\textwidth}\includegraphics[width=0.48\textwidth,page=3]{f_construction2.pdf}
\caption{\label{fi:f_construction_same_way} In the proof of Proposition~\ref{prop:weak_bi_lipschitz}, if $V_{l,k} \in \CI_{n-1}$ and $v_{n,i} \in V_{l,k}$, it might be that $\wt{g}_{n-1}(v_{n,j}) < \inf_{K_n} f$. In the depicted case, we harmonically interpolate in $V'_{l,k}$ (tiled in blue) instead of $V_{n,i}$.}
\end{figure}

\begin{proof}[Proof of Proposition~\ref{prop:weak_bi_lipschitz}]
Let $M,\rextremitylb,\rextremityub$ be such that the conclusion of Corollary~\ref{co:covering} holds. Let $\{U_l\}$ be a minimal associated cover. We show that there is a constant $c$ depending on $M$ such that
\begin{equation}\label{eq:energies_comparable}
 \trformres{V}{\cdot}{\cdot}{\Gamma} \le c\rformres{V}{\cdot}{\cdot}{\Gamma} 
 \quad\text{for each } V \in \metregions .
\end{equation}
Suppose that $V \in \metregions$. By~\eqref{eq:rmet_to_form_sep}, it suffices to show that for each finite set $A \subseteq \Upsilon_\Gamma \cap \ol{V}$ we have
\begin{equation}\label{eq:energies_comparable_fin}
 \trformrestr{V}{A}{\cdot}{\cdot}{\Gamma} \le c\rformrestr{V}{A}{\cdot}{\cdot}{\Gamma} 
\end{equation}
where $\rformrestr{V}{A}{\cdot}{\cdot}{\Gamma}, \trformrestr{V}{A}{\cdot}{\cdot}{\Gamma}$ are the trace resistance forms on $A$.

To this end, we let $f$ be a $\rformres{V}{\cdot}{\cdot}{\Gamma}$-harmonic function with prescribed values on $A$. For each $\epsilon > 0$, we will construct a function $\wt{f}$ with $\abs{\wt{f}(x)-f(x)} < \epsilon$ for each $x \in A$ and $\trformres{V}{\wt{f}}{\wt{f}}{\Gamma} \le c\rformres{V}{f}{f}{\Gamma}$ where $c$ is a constant depending only on $M$. This will imply
\[ \trformrestr{V}{A}{\wt{f}}{\wt{f}}{\Gamma} \le \trformres{V}{\wt{f}}{\wt{f}}{\Gamma} \le c\rformres{V}{f}{f}{\Gamma} = c\rformrestr{V}{A}{f}{f}{\Gamma} . \]
Since resistance forms on finite sets are continuous (due to the explicit expression~\eqref{eq:df_discrete}), this will imply~\eqref{eq:energies_comparable_fin}.

We can assume that $V = \bigcup_l U_l$, otherwise we consider $\wt{V} = V \cup \bigcup_{l: U_l \cap V \neq \emptyset} U_l$ instead. Note that by the definition of $\metregions$, if $\delta$ is smaller than the diameters of the loops bounding $V$, then $\wt{V} \setminus V$ consists of a collection of dead ends in $\wt{V}$ relative to $V$, so that functions on $V$ can be extended to $\wt{V}$ by setting them constant in each of the dead ends, and their energies are unchanged due to Lemma~\ref{le:weak_rform_dead_ends}.

Recall~\eqref{eq:fin_energy_continuous} that every $f \in \rfdomainres{V}{\Gamma}$ is continuous. Therefore we can let $\delta > 0$ be small enough so that
\begin{equation}\label{eq:small_cover}
 \sup_{\Bpath(x,2\delta)} f - \inf_{\Bpath(x,2\delta)} f < \epsilon
 \quad\text{for every } x \in A .
\end{equation}
Let $\{ U_l \}$ be the cover given in Corollary~\ref{co:covering} with this $\delta$. We order the sets so that
\[ \sup_{K_l} f \]
is in increasing order as $l$ increases. We recall the notation $w_{l,i},v_{l,i},v'_{l,i},u_{l,i},V_{l,i},V'_{l,i}$. We let $j_l \in \N$ be the scale of the good region $U_l$ (i.e.\ $U_l$ arises from the event $\Gregion_{z_l,j_l}$).

We will define a sequence of functions $(\wt{g}_n)$ that satisfy the following conditions inductively.
\begin{enumerate}[(i)]
\item\label{it:gn_domain} The function $\wt{g}_n$ is defined on $\bigcup_{l=1}^n K_l$.
\item\label{it:gn_osc} For each $l \le n$ and $z \in K_l$ we have
\[
\min_{l' \le n : K_{l'} \cap K_l \neq \emptyset} \inf_{K_{l'}} f \le \wt{g}_n(z) \le \max_{l' \le n : K_{l'} \cap K_l \neq \emptyset} \sup_{K_{l'}} f .
\]
\item\label{it:gn_energy} We have that
\[
\trform{\wt{g}_n}{\wt{g}_n}{\Gamma} \le \trform{\wt{g}_{n-1}}{\wt{g}_{n-1}}{\Gamma}+c\rformres{U_n}{f}{f}{\Gamma}
\]
for a constant $c$ depending only on $M$.
\newcounter{saveenumi}
\setcounter{saveenumi}{\value{enumi}}
\end{enumerate}
This will conclude the proof since by~\eqref{it:gn_energy} and~\eqref{eq:energy_sum_cover} we have with $\wt{g}$ the final $\wt{g}_n$ generated by the algorithm before it terminates
\[ \trform{\wt{g}}{\wt{g}}{\Gamma} \lesssim \sum_n \rformres{U_n}{f}{f}{\Gamma} \lesssim \rform{f}{f}{\Gamma} . \]
Further, by~\eqref{it:gn_osc} and~\eqref{eq:small_cover} we have that $\abs{\wt{g}(x)-f(x)} < \epsilon$ on $A$.

We now define the functions $\wt{g}_n$ satisfying the properties~\eqref{it:gn_domain}--\eqref{it:gn_energy}. For technical reasons, we also define for each $n$ a collection $\CI_n$ of some of the $V_{l,i}$ with $l \le n$ so that the following properties hold.
\begin{enumerate}[(i)]
\setcounter{enumi}{\value{saveenumi}}
 \item\label{it:g_ge_f} For each $z \in \bigcup_{l=1}^n K_l \setminus \bigcup \CI_n$ we have $\wt{g}_n(z) \ge f(z)$.
 \item\label{it:max_f_inside} For each $V_{l,i} \in \CI_n$ we have $\wt{g}_n(v'_{l,i}) = \sup_{K_l} f$.
\end{enumerate}
The reason we need to introduce $\CI_n$ is that it is not clear that we can make~\eqref{it:g_ge_f} hold on $\bigcup_{l=1}^n K_l$.

We set $\wt{g}_1|_{K_1} = \sup_{K_1} f$ and $\CI_1 = \emptyset$.  Then we clearly have that~\eqref{it:gn_domain}--\eqref{it:max_f_inside} are satisfied for $n=1$.

Suppose that $n \in \N$ and we have defined $\wt{g}_1,\ldots,\wt{g}_{n-1}$ and $\CI_{n-1}$ satisfying~\eqref{it:gn_domain}--\eqref{it:max_f_inside}. We now define $\wt{g}_n$ as follows. We let $\CI_n$ be the same as $\CI_{n-1}$ except for the differences mentioned below. Let $\{v_{n,i}\}$ be the boundary points of $U_n$. For each $v_{n,i}$, we consider the following cases.
\begin{enumerate}[\textbullet]
 \item If $v_{n,i} \notin \bigcup_{l=1}^{n-1} U_l$, then set $\wt{g}_n(v_{n,i}) = \sup_{K_n} f$.
 \item If $v_{n,i} \in \bigcup_{l=1}^{n-1} U_l \setminus \bigcup \CI_{n-1}$, then set $\wt{g}_n(v_{n,i}) = \wt{g}_{n-1}(v_{n,i})$, $\wt{g}_n(v'_{n,i}) = \sup_{K_n} f$, and harmonically interpolate in $V_{n,i}$. Add $V_{n,i}$ to $\CI_n$. Note that by~\eqref{it:g_ge_f},~\eqref{eq:extremities_energy}, and~\eqref{eq:good_region_energy}, we have
 \[ \median[2^{-j_n}]{} \trformres{V_{n,i}}{\wt{g}_n}{\wt{g}_n}{\Gamma} \lesssim \left(\sup_{K_n} f - f(v_{n,i})\right)^2 \lesssim \median[2^{-j_n}]{} \rformres{U_n}{f}{f}{\Gamma} . \]
 In particular, all the required items hold.
 \item If $v_{n,i} \in V_{l,k}$ for some $V_{l,k} \in \CI_{n-1}$, we distinguish the following cases.
 \begin{enumerate}[(a)]
  \item $U_n$ faces in the opposite way from $U_l$ (see Figure~\ref{fi:f_construction_opposite_way}). Set $\wt{g}_n$ to be constant $\wt{g}_{n-1}(v'_{l,k})$ on $V_{l,k} \setminus U_n$ (this step only decreases energy), and then harmonically interpolate between $v_{n,i}$ and $v'_{n,i}$ as before. Remove $V_{l,k}$ from $\CI_n$ and add $V_{n,i}$ to $\CI_n$. By~\eqref{it:max_f_inside}, we have in particular
  \[ \wt{g}_n(v_{n,i}) = \wt{g}_{n-1}(v'_{l,k}) = \sup_{K_l} f \ge f(v_{n,i}) , \]
  so all the required items hold again.
  \item $U_n$ faces in the same way as $U_l$ (see Figure~\ref{fi:f_construction_same_way}). We set $\wt{g}_n = \wt{g}_{n-1}$ on $V_{l,k}$ and harmonically interpolate in the region $V'_{l,k}$ between $v'_{l,k}$ and $u_{l,k}$ instead, with $\wt{g}_n(v'_{l,k}) = \wt{g}_{n-1}(v'_{l,k})$ and $\wt{g}_n(u_{l,k}) = \sup_{K_n} f$. We keep $V_{l,k}$ in $\CI_n$. Note that since $\wt{g}_{n-1}(v'_{l,k}) = \sup_{K_l} f$ by~\eqref{it:max_f_inside}, the condition~\eqref{it:g_ge_f} still holds for $\wt{g}_n$. We claim that this also respects the condition~\eqref{it:gn_energy}. Indeed, note that the region $U_n$ cannot be much bigger than $U_l$, due to the condition on $w_{l,i}$ in the event $\Eann_{z_n,j_n}$. (And $U_n$ cannot be much smaller than $U_l$ either as it would be covered by $U_l$ otherwise.) Therefore we have $\median[2^{-j_l}]{} \asymp \median[2^{-j_n}]{}$ by Proposition~\ref{pr:quantiles_metric}. This implies, by~\eqref{eq:extremities_energy},~\eqref{eq:good_region_energy},
  \[
  \median[2^{-j_l}]{} \trformres{V'_{l,k}}{\wt{g}_n}{\wt{g}_n}{\Gamma} \asymp (\wt{g}_n(u_{l,k})-\wt{g}_n(v'_{l,k}))^2 \le \left(\sup_{K_n} f - f(v'_{l,k})\right)^2 \lesssim \median[2^{-j_n}]{} \rformres{U_n}{f}{f}{\Gamma} .
  \]
 \end{enumerate}
\end{enumerate}
Set $\wt{g}_n$ to be constant $\sup_{K_n} f$ in the remaining parts of $K_n$. Finally, remove the regions $V_{l,k} \in \CI_{n-1}$ from $\CI_n$ that are covered by $U_n$.
\end{proof}

\section{Uniqueness}
\label{sec:uniqueness}

In this section we complete the proof that any two weak \clekp{} resistance forms are scalar multiples of each other. This will also show that they are in fact (strong) \clekp{} resistance forms, thus completing the proof of Theorem~\ref{thm:unique_metric}.

In the previous section, we proved that if we have two weak $\CLE_{\kappa'}$ resistance forms $(\rformres{V}{\cdot}{\cdot}{\Gamma})_{V \in \metregions}$ and $(\trformres{V}{\cdot}{\cdot}{\Gamma})_{V \in \metregions}$ with comparable scaling factors, then they are bi-Lipschitz equivalent with deterministic constants $0 < c_1 \le c_2 < \infty$. We show in Section~\ref{subsec:uniqueness} that in this setting one can take $c_1 = c_2$. Then we show in Section~\ref{subsec:exponent_unique} that there is a unique scaling exponent $\resexp > 0$ depending only on $\kappa'$ such that every \clekp{} resistance form satisfies scaling covariance with this exponent $\resexp > 0$. This in particular lifts the assumption~\eqref{eq:ass_medians_comparable} on the comparability of the scaling factors.

\subsection{Uniqueness}
\label{subsec:uniqueness}

\begin{proposition}
\label{prop:weak_uniqueness}
Suppose that we have the same setup as in Section~\ref{se:bilipschitz_setup} where $(\rformres{V}{\cdot}{\cdot}{\Gamma})_{V \in \metregions}$ and $(\trformres{V}{\cdot}{\cdot}{\Gamma})_{V \in \metregions}$ are two conditionally independent weak \clekp{} resistance forms such that~\eqref{eq:ass_medians_comparable} holds. Then there exists a constant $c_0 \in (0,\infty)$ so that $\rformres{V}{\cdot}{\cdot}{\Gamma} = c_0 \trformres{V}{\cdot}{\cdot}{\Gamma}$ for all $V \in \metregions$.
\end{proposition}

In the following, we assume that $0 < c_1 \leq c_2 < \infty$ are deterministic constants so that
\begin{equation}\label{eq:bilipschitz_ass}
 c_1\rformres{V}{\cdot}{\cdot}{\Gamma} \le \trformres{V}{\cdot}{\cdot}{\Gamma} \le c_2\rformres{V}{\cdot}{\cdot}{\Gamma} .
\end{equation}
for every $V \in \metregions$.

We have shown in Proposition~\ref{prop:weak_bi_lipschitz} that such constants $c_1,c_2$ exist. Our goal in this section is to show that we can we can take $c_1=c_2$. We assume throughout that $c_1$ (resp.\ $c_2$) is the largest (resp.\ smallest) constant so that~\eqref{eq:bilipschitz_ass} holds. We note that the results of this section still hold if the \clekp{} resistance forms are defined only on the event that $\CL \cap \partial\D = \emptyset$.

We assume that $c_1 < c_2$ and argue that on sufficiently many scales, the probability is positive that there exist ``shortcuts'' in the sense that $\rmetres{V}{v}{v'}{\Gamma}/\rmettres{V}{v}{v'}{\Gamma}$ is bounded away from $c_1$ or from $c_2$. Using a spatial independence argument, we will see that with probability $1$ such shortcuts exist everywhere. We begin by observing that on each scale $2^{-j}$, for a ``typical'' region $V$ bounded between two intersecting loops of $\Gamma$ that has Euclidean size $2^{-j}$, the ratio $\rmetres{V}{v}{v'}{\Gamma}/\rmettres{V}{v}{v'}{\Gamma}$ is either bounded away from $c_2$ or from $c_1$ with positive probability. We will distinguish between the scales where the former (resp.\ latter) is the case. This is the subject of the following definition.

Recall that if $D \subseteq \C$ is a simply connected domain and $x,y \in \partial D$ are distinct, then there is a unique law on pairs of curves $(\eta'_1,\eta'_2)$ from $x$ to $y$ such that the conditional law of $\eta'_2$ given $\eta'_1$ (resp.\ $\eta'_1$ given $\eta'_2$) is that of an \slekp{} in the complementary connected component to the right of $\eta'_1$ (resp.\ to the left of $\eta'_2$). See \cite[Section~2.2.1]{my2025geouniqueness}. We refer to $(\eta'_1,\eta'_2)$ as the \emph{bichordal \slekp{}} in $(D,x,y)$.

\newcommand*{\shortcutscales}{\CJ}

Let $j \in \N$, and let $(\eta'_1,\eta'_2)$ be a bichordal \slekp{} in $(2^{-j}\D, -i2^{-j}, 2^{-j})$. Let $I_j$ be the event that $\eta'_1 \cap \eta'_2 \cap B(0,1/2) \neq \emptyset$, and note that by scale-invariance $\p[I_j]$ does not depend on $j$. On the event $I_j$, let $x_j$ (resp.\ $y_j$) the first (resp.\ last) intersection point of $\eta'_1 \cap \eta'_2 \cap B(0,1/2)$, and let $V_j$ be the region bounded between the segments of $\eta'_1,\eta'_2$ from $x_j$ to $y_j$. Let $\shortcutscales$ be the set of $j \in \N$ such that
\begin{equation}\label{eq:shortcut_scale}
 \p\left[ \frac{\rmetres{V_j}{x_j}{y_j}{\Gamma}}{\rmettres{V_j}{x_j}{y_j}{\Gamma}} \le \sqrt{c_1 c_2} \mmiddle| I_j \right] \ge 1/2 .
\end{equation}
The idea is that if a positive density of integers is in $\CJ$, then we show that we can find shortcuts everywhere that imply $c_2$ is not optimal. On the other hand, if a positive density of integers is in $\N \setminus \CJ$, then $c_1$ is not optimal.

The following lemma states that if a shortcut exists with positive probability, it exists with a probability arbitrarily close to $1$.

\begin{lemma}\label{le:shortcut_bichordal_sle}
 For each $p \in (0,1)$ there exist $N \in \N$ and $\epsilon > 0$ such that the following is true. Let $D$ be a simply connected domain and $\varphi\colon \h \to D$ a conformal transformation. Let $j,k \in \N$ be such that $\abs{\varphi'(i)} \ge 2^{-j}$, and suppose that
 \[ \abs{\shortcutscales \cap \{j,\ldots,j+k\}} \ge N . \]
 Let $(\eta'_1,\eta'_2)$ be a bichordal \slekp{} in $(D,\varphi(0),\varphi(\infty))$, and given $(\eta'_1,\eta'_2)$, sample a \clekp{} and the internal metrics in the region between $\eta'_1,\eta'_2$. Let $G_{D,\varphi}$ be the event that there are $x,y \in \eta'_1 \cap \eta'_2$ with
 \[ 
 \frac{\rmet{x}{y}{\Gamma}}{\rmett{x}{y}{\Gamma}} \le \sqrt{c_1 c_2} 
 \quad\text{and}\quad
 \rmet{x}{y}{\Gamma} \ge \epsilon\,\median[2^{-j-k}]{} .
 \]
 Then
 \[ \p[ G_{D,\varphi} ] \ge p . \]
\end{lemma}

The proof of Lemma~\ref{le:shortcut_bichordal_sle} is exactly the same as that of \cite[Lemma~5.7]{my2025geouniqueness}. Recall that the right boundary of $\eta'_1$ and the left boundary of $\eta'_2$ can be described by a pair of flow lines intersecting with angle~$\angledouble$~\eqref{eq:angledouble}. The idea is that if $j \in \shortcutscales$, then by absolute continuity, an event analogous to~\eqref{eq:shortcut_scale} for such a pair of flow lines has positive probability $p_0$ bounded from below. By the independence across scales, if $N$ is large enough compared to $p_0$, then we have many (almost) independent trials to create such an event in $D$.

\begin{lemma}\label{le:shortcut_bichordal_cle_unif}
 For any $p \in (0,1)$ and $r > 0$ there exist $N \in \N$ and $\epsilon > 0$ such that the following is true. Let $D$ be a simply connected domain and $\varphi\colon \D \to D$ a conformal transformation. Let $x_1,x_2,x_3,x_4 \in \partial D$ such that $\abs{\varphi^{-1}(x_i)-\varphi^{-1}(x_{i'})} \ge r$ for each $i \neq i'$. Let $j,k \in \N$ be such that $\abs{\varphi'(0)} \ge 2^{-j}$, and suppose that
 \[ \abs{\shortcutscales \cap \{j,\ldots,j+k\}} \ge N . \]
 Let $(\eta'_1,\eta'_2,\Gamma)$ be a bichordal \clekp{} in $(D;x_1,x_2,x_3,x_4)$ conditioned on the event $\{ \eta'_1 \cap \eta'_2 \neq \emptyset \}$, and sample the internal metrics in the region between $\eta'_1,\eta'_2$. Let $G_{D,\ul{x}}$ be the event that there are $x,y \in \eta'_1 \cap \eta'_2$ with
 \[ 
 \frac{\rmet{x}{y}{\Gamma}}{\rmett{x}{y}{\Gamma}} \le \sqrt{c_1 c_2} 
 \quad\text{and}\quad
 \rmet{x}{y}{\Gamma} \ge \epsilon\,\median[2^{-j-k}]{} .
 \]
 Then
 \[ \p[ G_{D,\ul{x}} ] \ge p . \]
\end{lemma}

\begin{proof}
 Let $x'$ (resp.\ $y'$) be the first (resp.\ last) intersection point of $\eta'_1 \cap \eta'_2$. The conditional law of the segments between $x'$ and $y'$ is a bichordal \slekp{} in $(D',x',y')$ where $D'$ is a connected component after removing from $D$ the segments before $x'$ and after $y'$. Let $k_0 \in \N$ be large enough such that the following holds. Let $E'$ be the event that there is a conformal map $\wt{\varphi}$ mapping $(\h,0,\infty)$ to $(D',x',y')$ such that $\abs{\wt{\varphi}'(i)} \ge 2^{-j-k_0}$. Then $\p[E'] \ge (p+1)/2$. By the continuity of the bichordal \clekp{} law with respect to the marked points (Proposition~\ref{prop:mccle_tv_convergence_int}), we can choose $k_0$ (depending on $r$) such that this holds uniformly for all choices of $(D;x_1,x_2,x_3,x_4)$.
 
 Suppose that $N$ is large enough so that $k_0 < N/2$ and that the statement of Lemma~\ref{le:shortcut_bichordal_sle} holds with $p$ replaced by $(p+1)/2$. We obtain
 \[ \p[ G_{D,\ul{x}} ] \ge \p[ G_{D',\wt{\varphi}} \mid E' ] \,\p[E'] \ge ((p+1)/2)^2 \ge p . \]
\end{proof}

Let $\CF_{z,j},\wt{\CF}_{z,j}$ and $\Esep_{z,j}$ be as defined in Section~\ref{se:bilipschitz_setup}.

\newcommand*{\Eannshortcut}{E^{\mathrm{sc}}}
\newcommand*{\Gregionsc}{G^{\mathrm{sc}}}

\begin{lemma}\label{le:annulus_with_shortcuts}
For each $p \in (0,1)$ and $\dsep > 0$ there exist $N \in \N$ and $\epsilon > 0$ such that the following is true. Let $z,j$ be such that $B(z,2^{-j}) \subseteq \D$, and suppose that
\[ \abs{\shortcutscales \cap \{j,\ldots,j+5N\}} \ge N . \]
Let $\Eannshortcut_{z,j}$ be the event that the event $\Eannharm_{z,j}$ from Section~\ref{subsec:good_annuli} occurs and we can select the regions $V_i$ for each $i$ such that
\[ \frac{\rmetres{V_i}{v_i}{v'_i}{\Gamma}}{\rmettres{V_i}{v_i}{v'_i}{\Gamma}} \le c_2-\epsilon . \]
Then
\[ \p[ \Eannshortcut \mid \CF_{z,j} ] \one_{\Esep_{z,j}} \ge p\one_{\Esep_{z,j}} . \]
\end{lemma}

\newcommand*{\Eexpl}{E^3}
\newcommand*{\rexplstart}{r_3}
\newcommand*{\rexplend}{r_4}
\newcommand*{\rexplsep}{r_5}

\begin{proof}
We will deduce Lemma~\ref{le:annulus_with_shortcuts} from Lemma~\ref{le:shortcut_bichordal_cle_unif}. Let $\rexplstart > \rexplend > 0$ and $\rexplsep > 0$ be small constants whose values will be decided upon later. Consider the following partial exploration of the \clekp{} in $A_{z,j}$. Let $w \in \rexplend\Z^2 \cap A_{z,j}$ and consider the partial exploration $\Gamma_\outside^{*, B(w,\rexplstart 2^{-j}), B(w,\rexplend 2^{-j})}$. Let $\Eexpl_{z,j}$ be the event that $\Eannharm_{z,j}$ occurs and for each $i$ there is a cut point $x$ of $V_i$ such that the following holds. Let $w \in \rexplend\Z^2 \cap A_{z,j}$ be the point closest to $x$. Then the partial exploration $\Gamma_\outside^{*, B(w,\rexplstart 2^{-j}), B(w,\rexplend 2^{-j})}$ has exactly $4$ marked points which correspond to the two strands intersecting at $x$, and the marked points have distance at least $\rexplsep 2^{-j}$.

Recall that intersecting strands of a \clekp{} intersect each other infinitely many times, and that the marked points of a partially explored \clekp{} are almost surely distinct \cite[Lemma~3.18]{amy-cle-resampling}. Therefore, by Proposition~\ref{prop:mccle_tv_convergence_int}, for each $\dsep > 0$ there are constants $\rexplstart,\rexplend,\rexplsep$ such that
\[ \p[ \Eannharm_{z,j} \cap (\Eexpl_{z,j})^c \mid \CF_{z,j} ] \one_{\Esep_{z,j}} \le (1-p)/2 \]
for each $z,j$ with $B(z,2^{-j}) \subseteq \D$.

Once we have fixed $\rexplstart,\rexplend,\rexplsep$, we apply Lemma~\ref{le:shortcut_bichordal_cle_unif} with $1-\rexplend^2 (1-p)/2$ in place of $p$ and $r = \rexplsep$. Taking a union bound over $w \in \rexplend\Z^2 \cap A_{z,j}$, we obtain that with conditional probability at least $p$, for each $i$ there are cut points $x,y$ of $V_i$ with
\[ 
 \frac{\rmetres{V_i}{x}{y}{\Gamma}}{\rmettres{V_i}{x}{y}{\Gamma}} \le \sqrt{c_1 c_2} 
 \quad\text{and}\quad
 \rmetres{V_i}{x}{y}{\Gamma} \ge \epsilon\,\median[2^{-j-5N}]{} \ge \epsilon\,c(N)\median[2^{-j}]{} 
\]
where we have used Proposition~\ref{pr:quantiles_metric} in the last inequality.
Since we are on the event $\Eannharm_{z,j}$, we conclude by~\eqref{eq:extremities_energy} and our standing assumption~\eqref{eq:bilipschitz_ass}.
\end{proof}

\begin{proof}[Proof of Proposition~\ref{prop:weak_uniqueness}]
 We note that at least one of the following must hold. The set of $n \in \N$ such that
 \begin{equation}\label{eq:shortcut_density}
  \frac{\abs{\shortcutscales \cap \{ n,\ldots,2n \}}}{n} \ge \frac{1}{2} 
 \end{equation}
 is infinite, or the set of $n$ where it is $\le 1/2$ is infinite. Suppose the former is true. We will show that there is a deterministic constant $\epsilon > 0$ such that
 \begin{equation}\label{eq:energy_improve}
  \frac{\trformres{V}{\cdot}{\cdot}{\Gamma}}{\rformres{V}{\cdot}{\cdot}{\Gamma}} \le c_2-\epsilon
 \end{equation}
 for every $V \in \metregions$. This contradicts the minimality of $c_2$. Conversely, if the latter is true, then by swapping the roles of $\trform{}{}{}$ and $\rform{}{}{}$ we obtain a contradiction to the maximality of $c_1$.
 
 Note that for fixed $N \in \N$, if $n \in \N$ is sufficiently large and~\eqref{eq:shortcut_density} holds, then the number of $j \in \{n,\ldots,2n\}$ with
 \[ \abs{\shortcutscales \cap \{j,\ldots,j+5N\}} \ge N \]
 is at least $n/4$.
 
 Let $\Gregionsc_{z,j}$ be defined as $\Gregion_{z,j}$ in Section~\ref{se:good_regions} with the stronger condition $\Eannshortcut_{z,j}$ from Lemma~\ref{le:annulus_with_shortcuts}. Combining Proposition~\ref{pr:good_region} with Lemma~\ref{le:annulus_with_shortcuts} and Lemma~\ref{le:separation_event}, we see that we can choose the parameters of the event $\Gregionsc_{z,j}$ such that the following holds almost surely. For each $\delta > 0$ there exists a finite collection of balls $\{ B(z_k,2^{-j_k}) \}$ with $2^{-j_k} < \delta$ such that $B(0,1-\delta) \subseteq \bigcup_k B(z_k,2^{-j_k-1})$ and $\Gregionsc_{z_k,j_k}$ occurs for each $k$. Let $\{U_l\}$ be a minimal associated cover as described below Corollary~\ref{co:covering}.
 
 Suppose that $V \in \metregions$, $A \subseteq \Upsilon_\Gamma \cap \ol{V}$ is a finite set, and $f$ is a $\rformres{V}{\cdot}{\cdot}{\Gamma}$-harmonic function with prescribed values on $A$. We argue that there is a function $\wt{f}$ with $\wt{f} = f$ on $A$ and
 \[ \trformres{V}{\wt{f}}{\wt{f}}{\Gamma} \le (c_2-\epsilon)\rformres{V}{f}{f}{\Gamma} . \]
 By~\eqref{eq:rmet_to_form_sep}, this will imply~\eqref{eq:energy_improve} and conclude the proof.
 
 We can assume that $V = \bigcup_l U_l$, otherwise we consider $\wt{V} = V \cup \bigcup_{l: U_l \cap V \neq \emptyset} U_l$ instead. Note that by the definition of $\metregions$, if $\delta$ is smaller than the diameters of the loops bounding $V$, then $\wt{V} \setminus V$ consists of a collection of dead ends, so that functions on $V$ can be extended to $\wt{V}$ by setting them constant in each of the dead ends, and their energies are unchanged due to Lemma~\ref{le:rform_additivity}.
 
 In the following, we let $\epsilon > 0$ denote an absolute constant (depending on the parameters of $\Gregionsc_{z,j}$ chosen above) whose value may change from line to line.
 
 Consider a disjoint sub-collection $\{ U_{l_k} \}$ of $\{ U_l \}$ as in Lemma~\ref{le:disjoint_subcover} with
 \[ \sum_k \rformres{U_{l_k}}{f}{f}{\Gamma} \ge \epsilon\,\rformres{V}{f}{f}{\Gamma} . \]
 By Lemma~\ref{le:energy_small_nbhd}, we can also assume that $\delta$ is small enough so that less than $1/2$ fraction of the energy of $f$ is in the $U_{l_k}$ that intersect $A$. In each of the $U_{l_k}$ that do not intersect $A$, by Lemma~\ref{le:energy_in_extremities}, there is a $V_{l_k,i}$ with
 \[ \rformres{V_{l_k,i}}{f}{f}{\Gamma} \ge \epsilon\,\rformres{U_{l_k}}{f}{f}{\Gamma} . \]
 We denote $\wt{V}_k = V_{l_k,i}$. Let $\wt{f} = f$ in $V \setminus \bigcup_k \wt{V}_k$, and harmonically interpolate in each $\wt{V}_k$. By the definition of $\Gregionsc_{z,j}$ we have
 \[ \trformres{\wt{V}_k}{\wt{f}}{\wt{f}}{\Gamma} \le (c_2-\epsilon)\rformres{\wt{V}_k}{f}{f}{\Gamma} . \]
 For the remaining regions we have, by our standing assumption~\eqref{eq:bilipschitz_ass},
 \[ \trformres{V \setminus \bigcup_k \wt{V}_k}{\wt{f}}{\wt{f}}{\Gamma} \le c_2\rformres{V \setminus \bigcup_k \wt{V}_k}{f}{f}{\Gamma} . \]
 By Lemma~\ref{le:rform_additivity} we conclude that
 \[ \begin{split}
  \trformres{V}{\wt{f}}{\wt{f}}{\Gamma} 
  &\le (c_2-\epsilon) \sum_k \rformres{\wt{V}_k}{f}{f}{\Gamma} + c_2\rformres{V \setminus \bigcup_k \wt{V}_k}{f}{f}{\Gamma} \\
  &\le c_2\rformres{V}{f}{f}{\Gamma} - \epsilon \sum_k \rformres{\wt{V}_k}{f}{f}{\Gamma} \\
  &\le c_2\rformres{V}{f}{f}{\Gamma} - \epsilon^2 \rformres{V}{f}{f}{\Gamma}
 \end{split} \]
 which concludes the proof.
\end{proof}

\subsection{Uniqueness of the exponent}
\label{subsec:exponent_unique}

In this section we show the scaling covariance of \clekp{} resistance forms with a scaling factor $\lambda^{\resexp}$ where $\resexp > 0$ depends only on $\kappa'$ (Proposition~\ref{prop:exponent_unique}).

\begin{proposition}
\label{prop:exponent_unique}
There exists $\resexp \in [\ddouble,\dsle]$ so that the following is true. Suppose that $(\rformres{V}{\cdot}{\cdot}{\Gamma})_{V \in \metregions}$ is a weak \clekp{} resistance form, then it is a (strong) \clekp{} resistance form with the scaling exponent~$\resexp$.
\end{proposition}

We will prove Proposition~\ref{prop:exponent_unique} in two steps.  The first step, stated and proved in Lemma~\ref{lem:scaling_factor_power} just below, is to show that each weak \clekp{} resistance form is a (strong) \clekp{} resistance form with some scaling exponent $\resexp > 0$. We will then prove that the value of $\resexp > 0$ does not depend on the choice of the \clekp{} resistance form, but only on the choice of $\kappa' \in (4,8)$.

\begin{lemma}
\label{lem:scaling_factor_power}
Suppose that we have the setup of Section~\ref{se:bilipschitz_setup} but condition on the event that $\CL \cap \partial\D = \emptyset$.\footnote{This conditioning implies that for each $\lambda \in (0,1]$, the law of~$\lambda\Gamma$ is absolutely continuous with respect to the law of~$\Gamma$, see \cite[Lemma~7.2]{my2025geouniqueness}.} Let $(\rformres{V}{\cdot}{\cdot}{\Gamma})_{V \in \metregions}$ be a weak \clekp{} resistance form. Then there exists $\resexp \in [\ddouble,\dsle]$ such that the following holds. Let $\lambda \in (0,1]$ and define $(\trformres{V}{f}{f}{\Gamma})_{V \in \metregions} = (\rformres{\lambda V}{f(\lambda^{-1}\cdot)}{f(\lambda^{-1}\cdot)}{\lambda\Gamma})_{V \in \metregions}$. Then almost surely $\trformres{V}{\cdot}{\cdot}{\Gamma} = \lambda^{-\resexp} \rformres{V}{\cdot}{\cdot}{\Gamma}$ for each $V \in \metregions$.
\end{lemma}

\begin{proof}
Note that $\trform{\cdot}{\cdot}{\Gamma} = (\trformres{V}{\cdot}{\cdot}{\Gamma})_{V \in \metregions}$ is also a weak \clekp{} resistance form. Let $\rmett{\cdot}{\cdot}{\Gamma} = (\rmettres{V}{\cdot}{\cdot}{\Gamma})_{V \in \metregions}$ be the associated resistance metrics, and let $\mediant[\delta]{}$ be as defined in Section~\ref{se:bilipschitz_setup}. Then $\mediant[\delta]{}/\median[\delta]{} = \median[\lambda\delta]{}/\median[\delta]{} \in [\lambda^{\dsle+o(1)},\lambda^{\ddouble+o(1)}]$ uniformly in $\delta \in (0,1]$ by Proposition~\ref{pr:quantiles_metric}. Therefore the condition~\eqref{eq:ass_medians_comparable} is satisfied, so by Proposition~\ref{prop:weak_uniqueness} we have $\trform{\cdot}{\cdot}{\Gamma} = c(\lambda)\rform{\cdot}{\cdot}{\Gamma}$ almost surely for a constant $c(\lambda) > 0$.

It follows that if $\lambda,\lambda' \in (0,1]$, then almost surely
\[ \rform{f((\lambda'\lambda)^{-1}\cdot)}{f((\lambda'\lambda)^{-1}\cdot)}{\lambda'\lambda\Gamma} = c(\lambda')\rform{f(\lambda^{-1}\cdot)}{f(\lambda^{-1}\cdot)}{\lambda\Gamma} = c(\lambda')c(\lambda)\rform{f}{f}{\Gamma} \]
and therefore $c(\lambda'\lambda) = c(\lambda')c(\lambda)$. Moreover, $\lambda \mapsto c(\lambda)$ is strictly decreasing with $\lim_{\lambda \to 0} c(\lambda) = \infty$ due to Proposition~\ref{pr:quantiles_metric}.
It follows that $c(\lambda) = \lambda^{-\resexp}$ for some $\resexp > 0$. It also follows from Proposition~\ref{pr:quantiles_metric} that $\resexp \in [\ddouble,\dsle]$.
\end{proof}

\begin{proof}[Proof of Proposition~\ref{prop:exponent_unique}]
It remains to show that if $\rform{\cdot}{\cdot}{\Gamma}$ and $\trform{\cdot}{\cdot}{\Gamma}$ are two weak \clekp{} resistance forms with respective scaling exponents $\resexp, \wt{\resexp}$, then $\resexp = \wt{\resexp}$. Suppose that $\resexp > \wt{\resexp}$. Following the proof of Proposition~\ref{pr:good_region}, there is a covering by good regions~$\{U_l\}$ analogously to Corollary~\ref{co:covering} such that
\[
 M^{-1}2^{j\resexp} \le \frac{\rformres{U_l}{f}{f}{\Gamma}}{(\sup_{K_l} f - \inf_{K_l} f)^2} \le M2^{j\resexp}
\]
for each non-constant $\rformres{U_l}{\cdot}{\cdot}{\Gamma}$-harmonic function $f$ and
\[
 M^{-1}2^{j\wt{\resexp}} \le \frac{\trformres{U_l}{f}{f}{\Gamma}}{(\sup_{K_l} f - \inf_{K_l} f)^2} \le M2^{j\wt{\resexp}}
\]
for each non-constant $\trformres{U_l}{\cdot}{\cdot}{\Gamma}$-harmonic function $f$.

By the exact same proof as for Proposition~\ref{prop:weak_bi_lipschitz}, we obtain that $\trform{\cdot}{\cdot}{\Gamma} = 0$ which is a contradiction.
\end{proof}

\section{The \clekp{} Brownian motion}
\label{se:diffusion}

We now turn to prove that the process defined by the $\CLE_{\kappa'}$ resistance form is a \clekp{} Brownian motion and likewise a \clekp{} Brownian motion defines a $\CLE_{\kappa'}$ resistance form, so as to complete the proof of Theorem~\ref{thm:unique_process}.  We will divide the two directions in the proof of Theorem~\ref{thm:unique_process} into Propositions~\ref{prop:resistance_defines_bm} and~\ref{prop:bm_defines_resistance}, stated and proved below. We then show in
Section~\ref{se:non_conf_inv} that the \clekp{} Brownian motion is not conformally invariant.

\subsection{$\CLE_{\kappa'}$ resistance form defines a \clekp{} Brownian motion}
\label{subsec:resistance_defines_bm}

\begin{proposition}
\label{prop:resistance_defines_bm}
Suppose that $(\rform{\cdot}{\cdot}{\Gamma}, \rfdomain{\Gamma})$ is a \clekp{} resistance form in the sense of Definition~\ref{def:cle_rform}, and let $\meas{\cdot}{\Gamma}$ be the \clekp{} gasket measure (see Section~\ref{se:cle_measure}). Let $\clebm$ be the Hunt process associated with the Dirichlet form $(\rform{\cdot}{\cdot}{\Gamma}, \rfdomain{\Gamma})$ on $L^2(\Upsilon_\Gamma, \meas{\cdot}{\Gamma})$. Then $\clebm$ is a \clekp{} Brownian motion in the sense of Definition~\ref{def:cle_brownian_motion}.
\end{proposition}

\begin{proof}
 By the definition of a \clekp{} resistance form, for each $V \in \metregions$, the space $(\ol{V} \cap \Upsilon_\Gamma, \rmetres{V}{\cdot}{\cdot}{\Gamma})$ is a compact resistance metric space. By the general theory (see Section~\ref{se:rmet_rform}), this defines a regular Dirichlet form on $L^2(\ol{V} \cap \Upsilon_\Gamma, \meas{\cdot}{\Gamma_V})$, and hence a $\meas{\cdot}{\Gamma_V}$-symmetric Hunt process $\clebm^V$ which is a diffusion process with infinite lifetime by the locality of the form and the compactness of the space (recall also that the two metrics $\dpathY$ and $\rmet{\cdot}{\cdot}{\Gamma}$ are topologically equivalent).

 For $U \subseteq \C$, let $\clebm_U$ denote the process killed upon exiting $U$ (called the \emph{part process} in \cite{cf2012dform}). We claim that the process $\clebm_U$ is determined by $(\rformres{V}{\cdot}{\cdot}{\Gamma})_{V \in \metregions[U]}$. Since the process $\clebm$ is continuous, it suffices to show for each $U_1 \Subset U$ that $\clebm_{U_1}$ is determined by $(\rformres{V}{\cdot}{\cdot}{\Gamma})_{V \in \metregions[U]}$. For notational simplicity, let us also pretend that $U_1 \cap \Upsilon_\Gamma$ has just one connected component (otherwise we consider each connected component separately). By \cite[Lemma~5.14]{amy-cle-resampling}, there is $V_1 \in \metregions[U]$ with $U_1 \cap \Upsilon_\Gamma \subseteq \ol{V_1} \cap \Upsilon_\Gamma$. By \cite[Theorem~3.3.8]{cf2012dform}, the Dirichlet form associated with $\clebm_{U_1}$ is given by
 \begin{align*}
  (\rfdomain{\Gamma})_{U_1} &= \{ f \in \rfdomain{\Gamma} \mid f=0 \text{ on } \Upsilon_\Gamma \setminus U_1 \} ,\\
  (\rform{f}{f}{\Gamma})_{U_1} &= \rform{f}{f}{\Gamma} \quad\text{for } f \in (\rfdomain{\Gamma})_{U_1} .
 \end{align*}
 In particular, by the property~\ref{it:rform_additive}, it agrees with the Dirichlet form associated with $\clebm^{V_1}_{U_1}$. The latter is determined by $\rformres{V_1}{\cdot}{\cdot}{\Gamma}$. This shows the claim.

 The translation invariance and scale invariance now follows from the corresponding properties of the resistance forms and the measure.
\end{proof}

\subsection{\clekp{} Brownian motion defines a $\CLE_{\kappa'}$ resistance form}
\label{subsec:bm_defines_resistance}

\begin{proposition}
\label{prop:bm_defines_resistance}
Suppose that $\clebm$ is a \clekp{} Brownian motion in the sense of Definition~\ref{def:cle_brownian_motion}. Let $(\rform{\cdot}{\cdot}{\Gamma}, \rfdomain{\Gamma})$ be the associated Dirichlet form. Then $(\rform{\cdot}{\cdot}{\Gamma}, \rfdomain{\Gamma})$ is a \clekp{} resistance form in the sense of Definition~\ref{def:cle_rform}.
\end{proposition}

\begin{proof}
 Let~$(\rform{\cdot}{\cdot}{\Gamma}, \rfdomain{\Gamma})$ be the Dirichlet form of the process $\clebm$. By its definition, $(\rform{\cdot}{\cdot}{\Gamma}, \rfdomain{\Gamma})$ is a resistance form and its associated resistance metric $\rmet{\cdot}{\cdot}{\Gamma}$ is a.s.\ continuous with respect to~$\dpathY$.

 We now construct a resistance metric $\rmetres{V}{\cdot}{\cdot}{\Gamma}$ for each $V \in \metregions$ and show that the collection $(\rmetres{V}{\cdot}{\cdot}{\Gamma})_{V \in \metregions}$ satisfies the characterization in Proposition~\ref{pr:cle_rmet_char}.

 Let $V \in \metregions$. By the definition of $\metregions$, there is a finite set of points $v_1,\ldots,v_n$ such that each point in $\Upsilon_\Gamma \setminus \ol{V}$ is either separated from $V$ by $\{v_1,\ldots,v_n\}$ or is contained in a dead end. Let $A \subseteq \ol{V} \cap \Upsilon_\Gamma$ be a finite set that contains $\{v_1,\ldots,v_n\}$ and further points that separate each pair $v_i,v_j$ in $\ol{V} \cap \Upsilon_\Gamma$. Let $w$ be the weight function associated with $\rmet{\cdot}{\cdot}{\Gamma}|_A$ and let
 \[\begin{split}
  w_V(x,y) =
  \begin{cases}
   w(x,y) , & \{x,y\} \nsubseteq \{v_1,\ldots,v_n\} ,\\
   0 , & \{x,y\} \subseteq \{v_1,\ldots,v_n\} .
  \end{cases}
 \end{split}\]
Define $\rmetres{V}{\cdot}{\cdot}{\Gamma}|_A$ to be the associated resistance metric on $A$. We claim that this definition of $\rmetres{V}{\cdot}{\cdot}{\Gamma}$ is consistent, i.e.\ if $A \subseteq A'$ are two sets as above, then they define the same metric when restricted to $A$. Suppose that $A' = A \cup \{x'\}$ and $w'$ is the weight function associated with $\rmet{\cdot}{\cdot}{\Gamma}|_{A'}$. We claim that $w(v_i,v_j) = w'(v_i,v_j)$ for each $i \neq j$.

Let $f_i,f'_i \in \rfdomain{\Gamma}$ be the $\rform{\cdot}{\cdot}{\Gamma}$-harmonic functions with boundary values $\delta_{v_i}$ on $A$, $A'$, respectively. Then $w(v_i,v_j) = -\rform{f_i}{f_j}{\Gamma}$ and $w'(v_i,v_j) = -\rform{f'_i}{f'_j}{\Gamma}$. Moreover, since $(A,w)$ is the trace of the network $(A',w')$, we have $\rformtr{A}{f_i}{f_j}{\Gamma} = \rformtr{A'}{f_i}{f_j}{\Gamma}$.

Since $A$ is chosen to contain points that separate $v_i,v_j$ in $\ol{V} \cap \Upsilon_\Gamma$, the point $x'$ can be connected to at most one of $v_i,v_j$. Assume that it is separated from $v_j$. Since the process $\clebm$ is continuous, by~\eqref{eq:harm_ext}, we have $w'(x',v_j) = 0$ and $f_j = f'_j$. It follows that
\begin{align*}
\rformtr{A'}{f'_i}{f'_j}{\Gamma}
&= \rformtr{A'}{f_i}{f'_j}{\Gamma} \quad\text{(since $f_i$ differs from $f_i'$ only at $x'$ and $w'(x',v_j) = 0$)}\\
&= \rformtr{A'}{f_i}{f_j}{\Gamma} \quad\text{(since $f_j = f_j'$)}
\end{align*}
and therefore $w(v_i,v_j) = w'(v_i,v_j)$ as claimed.

In particular, this shows that $\rmetres{V}{\cdot}{\cdot}{\Gamma}$ indeed defines a resistance metric on $\ol{V} \cap \Upsilon_\Gamma$.

Next, we argue that for each $V \in \metregions[U]$, the metric $\rmetres{V}{\cdot}{\cdot}{\Gamma}$ as constructed above is determined by $\clebm_U$ and $\meas{\cdot}{\Gamma_{U^*}}$. Since both are supposed to be Markovian and translation invariant, we get the Markovianity and translation invariance of $(\rmetres{V}{\cdot}{\cdot}{\Gamma})_{V \in \metregions[U]}$.

To see the claim, let us note that by \cite[Theorem~3.3.8]{cf2012dform}, the process $\clebm_U$ and $\meas{\cdot}{\Gamma_{U^*}}$ together determine $\rform{f}{f}{\Gamma}$ for each $f \in (\rfdomain{\Gamma})_U =  \{ f \in \rfdomain{\Gamma} \mid f=0 \text{ on } \Upsilon_\Gamma \setminus U \}$. Let $A$ be as in the construction above, and let $x \in A$, $y \in A \setminus \{v_1,\ldots,v_n\}$. Let $f_x,f_y$ be the harmonic functions with boundary value $\delta_x$ resp.\ $\delta_y$ on $A$. Since the process $\clebm$ is continuous, by~\eqref{eq:harm_ext}, the function $f_y$ is supported on $\ol{V}$. Let $\wt{f}_x \in \rfdomain{\Gamma}$ be a function that agrees with $f_x$ in a neighborhood of $V$ and is zero on the complement of $U$. By the locality, we have $\rform{f_x}{f_y}{\Gamma} = \rform{\wt{f}_x}{f_y}{\Gamma}$. Since $\wt{f}_x,f_y \in (\rfdomain{\Gamma})_U$, this shows the claim.

Finally, we need to check the items~\eqref{it:rmet_compatibility}--\eqref{it:rmet_cut_loop} in Proposition~\ref{pr:cle_rmet_char}.

We show~\eqref{it:rmet_compatibility}. Let $V,V' \in \metregions$, $V \subseteq V'$, and $x,y \in \ol{V} \cap \Upsilon_\Gamma$ such that for every $u \in V' \setminus V$ there is a point $z \in \ol{V}$ that separates $u$ from $x,y$ in $\ol{V'} \cap \Upsilon_\Gamma$. Let $A$ (resp.\ $A'$) be a set as in the construction of $\rmetres{V}{x}{y}{\Gamma}$ (resp.\ $\rmetres{V'}{x}{y}{\Gamma}$) above where we can assume $A = A' \cap \ol{V}$. Moreover, we can add further points so that for each $u \in A'$ that is separated from $x,y$ in $\ol{V'} \cap \Upsilon_\Gamma$ by a single point $z \in \ol{V}$, the point $z$ is also in $A$. Since the process $\clebm$ is continuous, by~\eqref{eq:harm_ext}, the point $z$ also separates $u$ from $x,y$ in the network $(A',w_{V'})$. By the same argument as above, the weights $w_V$ and $w_{V'}$ agree on the $2$-connected component of $x,y$. Therefore $\rmetres{V}{x}{y}{\Gamma} = \rmetres{V'}{x}{y}{\Gamma}$ as desired.

The items~\eqref{it:rmet_cut_point} and~\eqref{it:rmet_cut_loop} follow similarly by including $z$ (resp.\ $a,b$) into the set $A$ and noting that $z$ separates $x$ from $y$ (resp.\ $A \cap \ol{V_{a,b}}$ from $A \cap \ol{V_{b,a}}$) in the network.
\end{proof}

\subsection{Non-conformal invariance}
\label{se:non_conf_inv}

We now explain that the \clekp{} Brownian motion is not conformally invariant. Let us consider the conformal transformation $\varphi\colon \h \to \D$ with $\varphi(0)=-i$, $\varphi(\infty)=i$, $\varphi(i)=0$. Let $\Gamma$ be a nested \clekp{} in $\h$.

Suppose that $x_1,\ldots,x_n \in \Upsilon_\Gamma$ are such that for each $i$, the point $x_i$ separates $\{x_j \mid j < i\}$ from $\{x_j \mid j > i\}$ in $\Upsilon_\Gamma$. Let $\clebm$ be the \clekp{} Brownian motion on $\Upsilon_\Gamma$, and let $(\clebmlaw{x})_{x \in \Upsilon_\Gamma}$ denote its law. By~\eqref{eq:harm_ext}, we have $\clebmlaw{x_i}[\sigma_{x_{i-1}} < \sigma_{x_{i+1}}] = \rmet{x_i}{x_{i+1}}{\Gamma}/\rmet{x_{i-1}}{x_{i+1}}{\Gamma}$. Therefore, if $\varphi(\clebm)$ had the same law as the \clekp{} Brownian motion on $\varphi(\Upsilon_\Gamma)$ (modulo time change), there would be some $c>0$ such that
\begin{equation}\label{eq:conf_inv_implication}
 \rmet{\varphi(x_i)}{\varphi(x_j)}{\varphi(\Gamma)} = c\,\rmet{x_i}{x_j}{\Gamma}
 \quad\text{for each } i,j .
\end{equation}
We will construct a positive probability event for $\Gamma$ on which~\eqref{eq:conf_inv_implication} fails.

For $k \in \Z$, let $\Gamma_\outside^{*, \h \setminus A(0,2^{k-1},2^{k+1}), \h \setminus A(0,2^{k-2},2^{k+2})}$ be the partial exploration of $\Gamma$ as defined in Section~\ref{se:mcle}. Let $G^k_0$ be the event that $\Gamma_\outside^{*, \h \setminus A(0,2^{k-1},2^{k+1}), \h \setminus A(0,2^{k-2},2^{k+2})}$ contains exactly two strands $\ell^k_1,\ell^k_2$ crossing $A(0,2^{k-2},2^{k+2})$, and that $\ell^k_1 \cap \ell^k_2 \neq \emptyset$. Note that by scale-invariance $p_0 = \p[G^k_0] > 0$ does not depend on $k$.

Let $x_k$ (resp.\ $y_k$) be the first (resp.\ last) intersection point of $\ell^k_1 \cap \ell^k_2$. By the scaling covariance of the \clekp{} resistance metric, the conditional law of $\rmet{x_k}{y_k}{\Gamma}$ given $G^k_0$ is the same as the conditional law of $2^{\resexp k}\rmet{x_0}{y_0}{\Gamma}$ given $G^0_0$ where $\resexp > 0$ is the exponent in Theorem~\ref{thm:unique_metric}. Therefore, if we let
\[
 E^k_1 = G^k_0 \cap \{ M^{-1}2^{\resexp k} \le \rmet{x_k}{y_k}{\Gamma} \le M2^{\resexp k} \} ,
\]
then $\p[G^k_0 \setminus E^k_1]$ does not depend on $k$ and $\lim_{M\to\infty} \p[G^k_0 \setminus E^k_1] = 0$. By applying the independence across scales argument from \cite[Proposition~4.10]{amy2025tightness}, we can choose $M$ so that for some $k$ that can be made as large as we want we have $\p[E^k_1 \cap E^{-k}_1] > 0$.

Next, we let
\[
 E^k_2 = G^k_0 \cap \{ M^{-1}\abs{\varphi'(0)}^{\resexp} 2^{-\resexp \abs{k}} \le \rmet{\varphi(x_k)}{\varphi(y_k)}{\varphi(\Gamma)} \le M\abs{\varphi'(0)}^{\resexp} 2^{-\resexp \abs{k}} \} .
\]
By applying Lemma~\ref{prop:mccle_tv_convergence_int}, we see that $\lim_{M\to\infty} \sup_{k<0} \p[G^k_0 \setminus E^k_2] = 0$. Moreover, as a consequence of Theorem~\ref{thm:unique_metric}, the \clekp{} resistance metric is reflection symmetric, hence $\p[E^k_2] = \p[E^{-k}_2]$. Therefore, by the independence across scales argument as before, for some $M$ we have $\p[E^k_1 \cap E^k_2 \cap E^{-k}_1 \cap E^{-k}_2] > 0$ for some $k$ that can be made as large as we want.

Finally, by Lemma~\ref{pr:link_probability}, the probability that $E^k_1 \cap E^k_2 \cap E^{-k}_1 \cap E^{-k}_2$ occurs and the strands $\ell^k_1,\ell^k_2,\ell^{-k}_1,\ell^{-k}_2$ link to one single loop of $\Gamma$ is positive. If $k$ is chosen sufficiently large, we see that~\eqref{eq:conf_inv_implication} fails on this event as on $E_2^k \cap E^{-k}_2$ we have that $\rmet{\varphi(x_k)}{\varphi(y_k)}{\varphi(\Gamma)}$ is comparable to $\rmet{\varphi(x_{-k})}{\varphi(y_{-k})}{\varphi(\Gamma)}$ but on $E_1^k \cap E^{-k}_1$ we have that $\rmet{x_k}{y_k}{\Gamma}$ is not comparable to $\rmet{x_{-k}}{y_{-k}}{\Gamma}$.

\bibliographystyle{alpha}
\providecommand{\noopsort}[1]{}

\end{document}